\documentclass[12pt]{article}
\usepackage[utf8]{inputenc}
\usepackage{csquotes,xpatch}
\usepackage[T1]{fontenc}
\usepackage{amsmath, amsfonts, amssymb, amsthm, dsfont, mathtools, mathrsfs}
\usepackage{enumitem}
\usepackage{xcolor,colortbl} 
\usepackage{graphicx}
\usepackage[hypertexnames=false]{hyperref}
\hypersetup{
	colorlinks   = true, 
	urlcolor     = blue, 
	linkcolor    = blue, 
	citecolor    = red   
}
\usepackage[nameinlink]{cleveref}
\usepackage{caption}
\usepackage{subcaption}
\usepackage{authblk}
\usepackage[thinlines]{easytable}
\usepackage{tikz}
\usetikzlibrary{arrows}
\usepackage{pdflscape} 
\usepackage{nccmath} 
\usepackage{calc}
\usepackage{accents}
\usepackage{multirow}
\usepackage{todonotes}
\usepackage{makecell}
\usepackage{stmaryrd} 

\theoremstyle{plain}
	\newtheorem{theorem}{Theorem}[section]
	\newtheorem{proposition}[theorem]{Proposition}
	\newtheorem{corollary}[theorem]{Corollary}
	
	\newtheorem{assumption}[theorem]{Assumption}
	\newtheorem{lemma}[theorem]{Lemma}
\theoremstyle{definition}
	\newtheorem{remark}[theorem]{Remark}
	\newtheorem{example}[theorem]{Example}

\usepackage[margin=1in]{geometry}

\providecommand{\keywords}[1]{\small\textbf{Keywords---} #1} 

\usepackage[
style=authoryear,
backend=biber,
dashed=false,
maxcitenames=2,
mincitenames=1,
maxbibnames=99,
giveninits=true
]{biblatex}

\DeclareNameAlias{sortname}{family-given}

\DeclareNameFormat{labelname}{%
	\nameparts{#1}%
	\usebibmacro{name:family}
	{\namepartfamily}
	{}
	{\namepartprefix}
	{\namepartsuffix}%
	\usebibmacro{name:andothers}%
}

\AtEveryBibitem{\clearfield{issn}}
\AtEveryBibitem{\clearfield{note}}
\DeclareFieldFormat[article,incollection,inproceedings]{title}{#1}
\DeclareFieldFormat{journaltitle}{\mkbibemph{#1}} 
\DeclareFieldFormat{title}{#1}                    
\renewbibmacro{in:}{}                             

\DeclareCiteCommand{\cite}
	{\usebibmacro{prenote}}
	{%
		\bibhyperref{%
			\printnames{labelname}%
			\setunit{\space}%
			\printtext[brackets]{\printfield{labelyear}}%
		}%
	}
	{\multicitedelim}
	{\usebibmacro{postnote}}
\newcommand{\hquad}{\hspace{0.5em}}

\newcommand{\Romannum}[1]{\textup{\uppercase\expandafter{\romannumeral#1}}}
\newcommand{\romannum}[1]{\textup{\lowercase\expandafter{\romannumeral#1}}}

\newcommand{\E}{\mathbb{E}}

\newcommand{\cov}{\mathrm{Cov}}

\newcommand{\set}[1]{\left\{ #1 \right\}}
\newcommand{\op}{{\operatorname{op}}}
\newcommand{\HS}{{\operatorname{HS}}}
\newcommand{\diag}{{\operatorname{diag}}}
\newcommand{\spec}{{\operatorname{spec}}}

\newcommand{\specd}{{\spec_{\downarrow}}}

\newcommand{\tr}{{\operatorname{tr}}}
\newcommand{\F}{{\operatorname{F}}}
\newcommand{\e}{\varepsilon}
\renewcommand{\P}{\mathbb{P}}
\newcommand{\Pn}{\P_n}
\newcommand{\LP}{{\mathbb{L}^2(\P)}}
\newcommand{\LPn}{{\mathbb{L}^2(\Pn)}}
\newcommand{\OP}{O_{\P}}

\newcommand{\C}{\mathbb{C}}
\newcommand{\M}{\mathbb{M}}
\newcommand{\N}{\mathbb{N}}
\newcommand{\R}{\mathbb{R}}

\newcommand{\mbH}{\mathbb{H}}

\newcommand{\bfH}{\mathbf{H}}
\newcommand{\bfP}{\mathbf{P}}
\newcommand{\bfQ}{\mathbf{Q}}
\newcommand{\bfa}{\mathbf{a}}
\newcommand{\bfb}{\mathbf{b}}
\newcommand{\bfr}{\mathbf{r}}
\newcommand{\bfUpsilon}{\boldsymbol{\Upsilon}}
\newcommand{\bfsigma}{\boldsymbol{\sigma}}
\newcommand{\bfvarsigma}{\boldsymbol{\varsigma}}
\newcommand{\bfkappa}{\boldsymbol{\kappa}}
\newcommand{\bfvarkappa}{\boldsymbol{\varkappa}}

\newcommand{\hbfH}{\hat{\mathbf{H}}}
\newcommand{\hbfP}{\hat{\mathbf{P}}}
\newcommand{\hbfQ}{\hat{\mathbf{Q}}}
\newcommand{\hbfUpsilon}{\hat{\boldsymbol{\Upsilon}}}

\newcommand{\msF}{\mathscr{F}}
\newcommand{\msG}{\mathscr{G}}

\newcommand{\fB}{\mathfrak{B}}
\newcommand{\fR}{\mathfrak{R}}

\newcommand{\cA}{\mathcal{A}}
\newcommand{\cB}{\mathcal{B}}
\newcommand{\cE}{\mathcal{E}}

\newcommand{\cH}{\mathcal{H}}
\newcommand{\cJ}{\mathcal{J}}
\newcommand{\cI}{\mathcal{I}}
\newcommand{\cP}{\mathcal{P}}
\newcommand{\cQ}{\mathcal{Q}}
\newcommand{\cR}{\mathcal{R}}

\newcommand{\cT}{\mathcal{T}}
\newcommand{\cU}{\mathcal{U}}

\newcommand{\cX}{\mathcal{X}}
\newcommand{\cY}{\mathcal{Y}}

\newcommand{\halpha}{\hat{\alpha}}
\newcommand{\hbeta}{\hat{\beta}}

\newcommand{\hDelta}{\hat{\Delta}}

\newcommand{\hlambda}{\hat{\lambda}}

\newcommand{\hPsi}{\hat{\Psi}}
\newcommand{\hpsi}{\hat{\psi}}

\newcommand{\hvarphi}{\hat{\varphi}}

\newcommand{\hcA}{\hat{\mathcal{A}}}

\newcommand{\hcE}{\hat{\mathcal{E}}}

\newcommand{\hcH}{\hat{\mathcal{H}}}
\newcommand{\hcJ}{\hat{\mathcal{J}}}
\newcommand{\hcP}{\hat{\mathcal{P}}}
\newcommand{\hcQ}{\hat{\mathcal{Q}}}
\newcommand{\hcR}{\hat{\mathcal{R}}}
\newcommand{\hcS}{\hat{\mathcal{S}}}

\newcommand{\hfE}{\hat{\mathfrak{E}}}
\newcommand{\hfR}{\hat{\mathfrak{R}}}
\newcommand{\hfa}{\hat{\mathfrak{a}}}

\newcommand{\tmax}{{\mathrm{max}}}
\newcommand{\tmin}{{\mathrm{min}}}

\newcommand{\bRN}[1]{\boldsymbol{\mathrm{\uppercase\expandafter{\romannumeral#1}}}}

\newcommand{\sn}[1]{\llbracket #1 \rrbracket_n} 

\newcommand{\rank}{\operatorname{rank}}

\newcommand{\gammaJ}{\gamma_{\raisebox{-2pt}{$\hspace{-0.3mm}\scriptstyle \cJ$}}}

\newcommand{\gammaJd}{\gamma_{{\raisebox{-2pt}{$\hspace{-0.3mm}\scriptstyle \cJ$}}_{ \raisebox{1pt}{$\scriptscriptstyle \! d$}}} }

\newlist{myenumerate}{enumerate}{1}
\setlist[myenumerate,1]{
	label={},
	leftmargin=0pt,
	itemindent=0pt,
	listparindent=\parindent,
	parsep=0pt
}

\newcommand{\mycase}[1]{\vspace{.5em}\noindent\underline{{#1}:}\hquad}
\newcommand{\citep}[1]{(\cite{#1})}

\newcommand{\refsupp}[1]{\ref{#1}}
\newcommand{\eqrefsupp}[1]{\eqref{#1}}
\newcommand{\seesupp}{the appendix}
\newcommand{\citesupp}[1]{\cite{#1}}

\title{First-order asymptotic expansions for spectral convergence of compact self-adjoint operators on general spectral subsets, with application to kernel Gram matrices}
\author{Eunseong Bae\thanks{Department of Statistics, University of California, Davis.
		Email: \texttt{esbae@ucdavis.edu}}\,}
\author{Wolfgang Polonik\thanks{Department of Statistics, University of California, Davis.
		Email: \texttt{wpolonik@ucdavis.edu}}}
\affil{} 
\date{August 21, 2026}

\begin{document}
\maketitle

\begin{abstract}
	We study the spectral convergence of compact, self-adjoint operators on a separable Hilbert space, and derive the first-order asymptotic expansions for their eigenvalues, eigenvectors and eigenprojections, along with remainder bounds	expressed in terms of weighted perturbation quantities. Our analysis focuses on eigenvalues indexed by a general subset, with minimal restrictions on their selection. In particular, when the eigenvalues of interest are clustered,	i.e., the inner spectral gaps are small, we provide different types of expansions for the eigenvalues, with remainder bounds that are robust to the inner spectral gaps; this contrasts with the existing literature, which mainly focuses on isolated eigenvalues. The usefulness of the provided expansions is illustrated by an application to kernel Gram matrices, deriving concentration inequalities as well as weak convergence results, which, in contrast to existing literature, primarily rely on assumptions on the kernel that are easy to check.
\end{abstract}

\tableofcontents

\section{Introduction}

Spectral analysis is a cornerstone of modern probability, statistics, and machine learning. 
In particular, the study of eigenvalues, eigenvectors and eigenprojections of matrices or operators plays a central role in various fields, as for instance in principal component analysis \citep{jolliffe2002principal}, functional principal component analysis \citep{ramsay2005functional}, and spectral clustering \citep{ng2001spectral}.
A fundamental question in this context is how accurately the spectral structure of a target object can be approximated by the corresponding empirical quantities.
Precise characterizations of this approximation are crucial for both theoretical guarantees and practical implementations.

Several classical results provide asymptotic expansions of the eigenvalues, eigenvectors and eigenprojections of compact self-adjoint operators. 
Notably, works by \cite{rellich1969perturbation} and \cite{kato1995perturbation} establish analytic perturbation theory, including first- and higher-order expansions. 
More recent studies have extended these results to the statistical setting; see, e.g., \cite{mas2003perturbation, hsing2015theoretical, mas2015high, wahl2019perturbation, jirak2020perturbation, jirak2023relative, jirak2024quantitative, wahl2026higher, stankewitz2026higher}.
These results form the theoretical foundation for our analysis, allowing precise characterization of the perturbation effects on both eigenvalues and eigenprojections.

In this paper, we make two major contributions. The first is to develop novel first-order asymptotic expansions for eigenvalues, eigenvectors and eigenprojections of compact self-adjoint operators on a separable Hilbert space under operator perturbations.
Different from existing results, our analysis focuses on eigenvalues indexed by a general subset of the spectrum, along with the corresponding eigenvectors and eigenprojections.
This subset, for example, may consist of a single eigenvalue, a band, or a cluster of eigenvalues.
Our assumptions on the selection of eigenvalues are minimal, allowing for flexible and widely applicable choices.
In particular, when the eigenvalues of interest are clustered, meaning that the inner spectral gaps are small, we obtain distinct forms of eigenvalue expansions whose remainder terms are robust to the inner spectral gaps.
This contrasts with prior work, which largely focuses on single eigenvalues.

We further adopt weighted perturbation-based arguments following \cite{wahl2019perturbation, jirak2024quantitative, wahl2026higher} and \cite{stankewitz2026higher}.
They generalize the classical perturbation approach and yield improved convergence rates.
Also, this approach has broader applicability, as it can be applied to situations in which other approaches from the literature fail to apply (see Example~\ref{spec:perturb:ex1}).
The first-order expansions are also closely connected to weak convergence phenomena underlying spectral convergence \citep{dauxois1982asymptotic, hsing2015theoretical} and provide quantitative control of the associated limiting behavior \citep{koltchinskii2017normal, jirak2024quantitative}.

Our second major contribution consists in the application of the general operator expansions to kernel Gram matrices, where our focus is on using weak assumptions only depending on the kernel that are relatively easy to check. 
We establish both finite-sample concentration inequalities and weak convergence results for eigenvalues, eigenvectors and eigenprojections.
Our results apply to a broad class of kernels and provide more precise asymptotic behaviors through the expansion than those in earlier works such as \cite{blanchard2007statistical}, \cite{rosasco2010learninga}, \cite{stankewitz2026higher}. 
Moreover, under mild conditions, our results yield weak convergence results analogous to those established in  \cite{koltchinskii1998asymptotics} and \cite{koltchinskii2000random}. 
As indicated above, in contrast to our contributions, these existing works require strong control over the eigenvalues and eigenfunctions, which is often unrealistic to check in practice.

The structure of the paper is as follows.
Section~\ref{spec:prelim} introduces basic concepts and notation used throughout the paper.
Section~\ref{spec} develops a general framework for the asymptotic expansions of the eigenvalues and eigenprojections of compact, self-adjoint operators.
Section~\ref{mercer} applies this framework to kernel Gram matrices, yielding both concentration inequalities and weak convergence results.
Finally, Sections~\ref{spec:proof} contains the proofs of the results in Section~\ref{spec},
while the proofs for Section~\ref{mercer}, along with additional technical arguments, are deferred to \seesupp.

\section{Basic notation} \label{spec:prelim}

In this section, we introduce basic notation used throughout the paper.
We denote by $\C$, $\R$ and $\N$ the sets of complex numbers, real numbers, and natural numbers, respectively.
The Kronecker delta $\delta_{k \ell}$ is defined by $\delta_{k \ell} =1$ if $k=\ell$, and $\delta_{k \ell} =0$ otherwise.
For $n \in \N$, we write $[n] = \{1,\dots,n\}$.

{\vspace{.5em}\bf Matrix analysis.} 
We let $\|\cdot\|_2$ denote the Euclidean 2-norm, and let $I_m$ denote the $m \times m$ identity matrix. 
For a matrix $A$, we write $A^\top$, $\|A\|_\op$ and $\|A\|_\F$ for its transpose, operator norm and Frobenius norm, respectively.
If $A$ is a square matrix, we let $\specd(A)$ denote the vector of its eigenvalues arranged in non-increasing order (counting multiplicities). We write $\tr(A)$ for the trace of $A$.

For an ordered finite index set $\cJ \subset \N$, we use the notation $(a_k)_{k \in \cJ}$ for the $|\cJ|$-dimensional vector formed by $\{a_k\}_{k \in \cJ}$ in the order induced by $\cJ$, and $(a_{k\ell})_{k,\ell \in \cJ}$ for the $|\cJ| \times |\cJ|$ matrix whose $(k,\ell)$-entry is $a_{k\ell}$ with rows and columns ordered according to $\cJ$. 
Similarly, $\diag(a_k)_{k \in \cJ}$ denotes the $|\cJ| \times |\cJ|$ diagonal matrix with diagonal entries $a_k$ arranged in the order of $\cJ$. 
To avoid redundant parentheses, we write $ \specd(a_{k\ell})_{k,\ell \in \cJ} = \specd( (a_{k\ell})_{k,\ell \in \cJ} )$.

{\vspace{.5em}\bf Hilbert spaces.} 
For a Hilbert space $\mbH$ with inner product $\langle\cdot,\cdot\rangle_\mbH$ and induced norm $\|\cdot\|_\mbH$, 
we denote the operator norm (OP norm), Hilbert-Schmidt norm (HS norm) and trace of an operator $\cH: \mbH \rightarrow \mbH$ by $\|\cH\|_{\op,\mbH}$, $\|\cH\|_{\HS,\mbH}$ and $\tr(\cH)$, respectively.
For $u,v \in \mbH$, the tensor product $u \otimes_\mbH v: \mbH \rightarrow \mbH$ is defined by
$ (u \otimes_\mbH v) w = \langle v, w \rangle_\mbH u.$
We let $\cI_\mbH: \mbH \to \mbH$ denote the identity operator on $\mbH$.

{\vspace{.5em}\bf $\mathbf{L^2}$ spaces.} 
Given a probability space $(\M,\P)$, $\LP$ denotes the space of real-valued functions $f$ on $\M$ such that 
$\int_\M f^2 \, d\P < \infty,$
equipped with the inner product
$\langle f, g \rangle_\LP = \int_\M f g \, d\P.$

The empirical measure associated with the observations $X_1, \dots, X_n \in \M$ is denoted by $\P_n$, 
defined for any real-valued measurable function $f$ on $\M$ as
$\int_\M f \, d\P_n = \frac{1}{n} \sum_{i=1}^n f(X_i).$
Correspondingly, $\LPn$ denotes the Euclidean space $\R^n$ endowed with the inner product
$\langle u, v \rangle_\LPn = \frac{1}{n} \sum_{i=1}^{n} u_i v_i. $

We also define the sampling or evaluation operator $\sn{\cdot}$ as a mapping from real-valued functions on $\M$ to $\LPn$:
$\sn{f} = (f(X_i))_{i\in[n]}.$
This satisfies the relation
$\langle \sn{f}, \sn{g} \rangle_\LPn = \frac{1}{n} \sum_{i=1}^{n} f(X_i) g(X_i),$
which is the empirical counterpart of the population inner product $\langle f, g\rangle_\LP$.
For brevity, we write
$\P f = \int_\M f \, d\P$, $\Pn f = \int_\M f \, d\P_n$, $(\Pn - \P)f = \Pn f - \P f.$

\section{First-order asymptotic expansions for spectral convergence of compact self-adjoint operators} \label{spec}

We consider two compact self-adjoint operators $\hcH$ and $\cH$ on a Hilbert space $\mbH$.
What we have in mind here is that $\hcH$ is a perturbed version of $\cH$.
Typical examples of $\hcH$ include empirical covariance operators or graph Laplacians.
Our goal is to formulate first-order asymptotic expansions for the eigenvalues, eigenvectors and eigenprojections of $\hcH$
in terms of those of $\cH$ under weighted perturbation regime.

\subsection{Setting}

\begin{assumption} \label{spec:assump1}
	$\mbH$ is a separable Hilbert space equipped with inner product $\langle\cdot,\cdot\rangle_\mbH$.
	$\cH$ and $\hcH$ are compact, self-adjoint operators on $\mbH$.
	For simplicity, we assume that $\mbH$ is a real Hilbert space, and all eigenvalues of $\cH$ and $\hcH$ are non-negative.
\end{assumption}
The assumption assures that $\cH$ admits an eigen-decomposition
\begin{equation*}
	\cH = \sum_{k=1}^\infty \lambda_k \cQ_k, \quad\text{where } \cQ_k = \psi_k \otimes_\mbH \psi_k
\end{equation*}
with eigenvalues $\lambda_1 \geq \lambda_2 \geq \dots \geq 0$ and corresponding $\mbH$-orthonormal eigenvectors $\{\psi_k\}_{k \in \N}$, that is, $\langle \psi_k, \psi_\ell \rangle_\mbH = \delta_{k \ell}$.
By compactness, all eigenvalues have finite multiplicity, and they are repeated accordingly in the sequence. 
Similarly, $\hcH$ can be expressed as
$\hcH = \sum_{k=1}^\infty \hlambda_k \hcQ_k,$ with $\hcQ_k = \hpsi_k \otimes_\mbH \hpsi_k,$
and $\hlambda_1 \geq \hlambda_2 \geq \dots \geq 0$ with $\langle \hpsi_k, \hpsi_\ell \rangle_\mbH = \delta_{k \ell}$.
Again, the eigenvalues are repeated according to their multiplicity. We mention that the results discussed in this paper continue to hold when negative eigenvalues are allowed and eigenvalues are arranged as $\lambda_1 \geq \lambda_2 \geq \dots \geq 0 \geq \dots \geq \lambda_{-2} \geq \lambda_{-1}$.

Next, for the selection of the eigenpairs, we introduce an index set $\cJ \subset \N$ with a positive spectral gap as follows:

\begin{assumption} \label{spec:assump2}
	Let $\cJ \subset \N$ denote a finite index set. The spectral gap between $\{\lambda_k\}_{k \in\cJ}$ and $\{\lambda_k\}_{k \in \cJ^c}$ is strictly positive, i.e.,
	$$ \gammaJ := \displaystyle\min_{k \in \cJ, \ell \notin \cJ} \left|\lambda_k - \lambda_\ell \right| > 0.$$
\end{assumption}

To account for multiplicities of the eigenvalues, let $\theta_\tmax = \theta_1 > \dots > \theta_D = \theta_\tmin $ denote the {\em distinct} eigenvalues of $\cH$ in $\{\lambda_k\}_{k \in \cJ},$ so that the sets
$$\cJ_d := \{k \in \cJ: \lambda_k = \theta_d \}, \quad d = 1, \dots, D, $$
form a partition of $\cJ,$ and $|\cJ_d|$ equals the multiplicity of the eigenvalue $\theta_d$. 
We would like to make it clear that the values $\theta_d$ and the sets $\cJ_d$ correspond to the operator $\cH$, rather than $\hcH$.

Selecting an index set $\cJ$ corresponds to focusing on a subset of eigenvalues of $\cH$---such as a spectral band, or a single eigenvalue with multiplicity larger than 1---that are of interest. 
By grouping these eigenvalues together, we analyze their collective behavior and derive precise first-order asymptotic expansions for the eigenvalues ($\{\lambda_k\}_{k \in \cJ}$ vs $\{\hlambda_k\}_{k \in \cJ}$), eigenvectors ($\{\psi_k\}_{k \in \cJ}$ vs $\{\hpsi_k\}_{k \in \cJ}$) and eigenprojections ($\cP_{\!\cJ}$ vs $\hcP_{\!\cJ}$) associated with $\cJ$, where the eigenprojections are defined as
\begin{equation*}
	\cP_{\!\cJ} = \sum_{k \in \cJ} \cQ_k, \quad	\hcP_{\!\cJ} = \sum_{k \in \cJ} \hcQ_k.
\end{equation*}

\subsection{Weighted perturbation quantities}

We introduce the following weighted perturbation quantities:
\begin{equation} \label{spec:def-alpha}
	\halpha_{\!\cJ} = \big\|\cT_{\!\cJ}^{1/2} (\hcH - \cH) \cT_{\!\cJ}^{1/2} \big\|_{\op,\mbH}, \quad 
	\hbeta_{\!\cJ} = \gammaJ^{-1/2} \big\||\cR|_{\!\cJ}^{1/2} (\hcH-\cH) \cP_{\!\cJ} \big\|_{\HS,\mbH},
\end{equation}
where
\begin{equation}
	\cT_{\!\cJ}^{1/2} = \frac{1}{\sqrt{\gammaJ}} \cP_{\!\cJ} + |\cR|_{\!\cJ}^{1/2}, \quad
	|\cR|_{\!\cJ}^{1/2} = \sum_{k \notin \cJ} \frac{\cQ_k}{\sqrt{\min_{m \in \cJ} |\lambda_k - \lambda_m|}}.
\end{equation}
$|\cR|_{\!\cJ}^{1/2}$ is called the square-root of the reduced resolvent of $\cH$ with respect to the eigenvalues $\{\lambda_k\}_{k\in\cJ}$.

These quantities generalize related quantities from the literature, such as those in \cite{jirak2024quantitative}, \cite{wahl2026higher} and \cite{stankewitz2026higher}, to our more general case.
For instance, if $\cJ$ corresponds to a spectral band, i.e., $\cJ=\{j_1, j_1 +1, \dots,j_2-1, j_2\}$, 
$|\cR|_{\!\cJ}^{1/2}$ is usually expressed as
\begin{equation}
	|\cR|_{\!\cJ}^{1/2} = \sum_{k:k<j_1} \frac{\cQ_k}{\sqrt{\lambda_k - \lambda_{j_1}}} + \sum_{k:k>j_2} \frac{\cQ_k}{\sqrt{\lambda_{j_2}-\lambda_k}}.
\end{equation}

Our strategy in this paper is to quantify, in terms of $\halpha_{\!\cJ}$ and $\hbeta_{\!\cJ}$, the remainder terms of the first-order expansions arising from spectral convergence.
This strategy subsumes the classical way of controlling the perturbation quantity $\gammaJ^{-1}\|\hcH-\cH\|_{\op,\mbH}$ by means of:
\begin{equation}
	\halpha_{\!\cJ} \leq \gammaJ^{-1} \|\hcH-\cH\|_{\op,\mbH}, \quad
	\hbeta_{\!\cJ} \leq\sqrt{|\cJ|} \halpha_{\!\cJ} \leq \sqrt{|\cJ|} (\gammaJ^{-1} \|\hcH-\cH\|_{\op,\mbH}).
	\label{spec:perturb-eq1}
\end{equation}

An alternative framework that fits within our setup is the so-called relative form boundedness condition.

\begin{example}[Relative form boundedness (RFB)] \label{spec:perturb:ex1}
	Consider the following RFB assumption:
	\begin{equation}
		|\langle (\hcH-\cH)v, v \rangle_\mbH | \leq \langle (\e_1 \cH + \e_2 \cI_\mbH)v, v \rangle_\mbH
		\quad \text{for all $v \in \mbH$}, \hquad\text{where $\e_1, \e_2 \ge 0$.}
		\label{spec:perturb:ex1-eq1}
	\end{equation}
	Such a condition arises, for example, in the study of graph Laplacians; see Theorem~4.1 of \cite{belkin2006convergence} or Theorem~4 of \cite{wahl2024kernelbased}.
	See also \cite{ipsen1998relative} or Chapter~6 of \cite{teschl2014mathematical} in matrix or operator theory.
	
	It turns out that if RFB holds, 
	we obtain useful bounds for  $\halpha_{\!\cJ}$ and $\hbeta_{\!\cJ}$ of the form
	\begin{equation}
\halpha_{\!\cJ} \leq \gammaJ^{-1} \big( \e_1 (\theta_\tmax+\gammaJ) +\e_2\big), 
		\quad
		\hbeta_{\!\cJ} \leq \gammaJ^{-1} \sqrt{\big(\e_1 \sum_{k\in\cJ} \lambda_k + \e_2 |\cJ|\big)\big(\e_1 (\theta_\tmax+\gammaJ)+ \e_2\big)}.%
\label{spec:perturb:ex1-eq2}
	\end{equation}
	
	In the following, we provide some more discussion about the connection of the various approaches for bounding spectral expansions. The RFB condition \eqref{spec:perturb:ex1-eq1} implies $\|\hcH-\cH\|_{\op,\mbH} \leq \e_1 \|\cH\|_{\op,\mbH} + \e_2$.
	Combining this with the classical perturbation approach \eqref{spec:perturb-eq1} yields inflated bounds for $\halpha_{\!\cJ}$ and $\hbeta_{\!\cJ}$ as compared to \eqref{spec:perturb:ex1-eq2} in cases where $\theta_\tmax + \gammaJ \ll \|\cH\|_{\op,\mbH}$, 
    for example, when the eigenvalues of $\cH$ exhibit exponential decay.
	
	There are also cases where the weighted perturbation-based argument can be applied while other approaches cannot.
	For instance, \cite{jirak2020perturbation} and \cite{jirak2023relative} study spectral convergence based on the \emph{relative noise} (RN) condition:
	\begin{equation}
		|\langle (\hcH-\cH) \psi_k, \psi_\ell \rangle_\mbH | \leq \e \sqrt{\lambda_k \lambda_\ell}
		\quad\text{for all $k,\ell \in \N$,\, where $\e\ge0$.}
		\label{spec:perturb:ex1-eq3}
	\end{equation}	
	Observe that this condition in particular implies that $|\langle (\hcH-\cH) \psi_k, \psi_\ell \rangle_\mbH |$ tends to zero as $k$ and $\ell$ increase. However, by plugging in $v=\psi_k$ in \eqref{spec:perturb:ex1-eq1}, we see that RFB does not imply RN since the term $\e_2$ precludes the required decay in $k$. On the other hand, as we have seen above already, the RFB does imply bounds for $\halpha_{\!\cJ}$ and $\hbeta_{\!\cJ}$ and thus for our expansions.
	This indicates that the weighted perturbation-based approach appears to have broader applicability. 
    Moreover, despite being derived under a more general setting, our results are compatible with the existing results under the RN condition alone. 
	We remark that when $\e_2=0$, RFB implies RN. In general, however, neither RFB nor RN implies the other.
	
    The details of this example are provided in \seesupp. 
\end{example}

\subsection{Eigenprojections of compact self-adjoint operators}

In this subsection, we present the first-order asymptotic expansion for the eigenprojections of $\hcH$ associated with $\cJ$ under weighted perturbation.
Specifically, we consider the following operators:
\begin{equation*}
	\hcS_{\!\cJ}
	= \sum_{k \in \cJ} \sum_{\ell \notin \cJ} \frac{\cQ_k (\hcH - \cH) \cQ_\ell + \cQ_\ell (\hcH - \cH) \cQ_k}{\lambda_k - \lambda_\ell}.	
\end{equation*}
By definition, $\hcS_{\!\cJ}$ satisfies
\begin{equation*}
	\langle \hcS_{\!\cJ} \psi_k, \psi_\ell \rangle_\mbH
	= \begin{cases}
		\displaystyle\frac{ \langle (\hcH-\cH) \psi_k, \psi_\ell \rangle_\mbH }{\lambda_k - \lambda_\ell},
		& \text{if } k \in \cJ, \ell \notin \cJ, \\
		\displaystyle\frac{ \langle (\hcH-\cH) \psi_k, \psi_\ell \rangle_\mbH }{\lambda_\ell - \lambda_k},
		& \text{if } k \notin \cJ, \ell \in \cJ, \\
		0, & \text{otherwise.}
	\end{cases}
\end{equation*}

We provide the bounds for $\hcP_{\!\cJ}-\cP_{\!\cJ}- \hcS_{\!\cJ}$ under weighted perturbations as follows.

\begin{theorem} \label{spec:eigproj}
	Under Assumptions~\ref{spec:assump1} and \ref{spec:assump2}, we have
	\begin{equation}
		\| \hcP_{\!\cJ} - \cP_{\!\cJ} - \hcS_{\!\cJ} \|_{\op,\mbH} 
        \leq \bigg(8 \min\Big\{D, 1+\frac{2\theta_\tmax-2\theta_\tmin}{\pi\gammaJ} \Big\} +32\bigg) \halpha^2_\cJ.
		\label{spec:eigproj-res1}
	\end{equation}
	Moreover, if $\halpha_{\!\cJ} \le 1/4$, then
	\begin{equation}
		\| \hcP_{\!\cJ} - \cP_{\!\cJ} - \hcS_{\!\cJ} \|_{\HS,\mbH} 
		\leq \bigg( \frac{112\sqrt{D}}{9} + 16\bigg) \halpha_{\!\cJ} \hbeta_{\!\cJ}.
		\label{spec:eigproj-res2}
	\end{equation}
\end{theorem}

\begin{remark}
	Our results extend and refine existing bounds under similar conditions.	
	For a single eigenvalue ($D=1$), our bound in the HS norm coincides with the existing results up to universal constants; see Theorem~1 of \cite{wahl2026higher} for $|\cJ|=1$ and Theorem~5.1 of \cite{stankewitz2026higher} for $|\cJ|>1$.
	OP norm-based bounds are not considered in these works.
	
	In case of a spectral band, Lemma~2 of \cite{koltchinskii2016asymptotics} and Proposition~2 of \cite{jirak2024quantitative} provide 
    $\| \hcP_{\!\cJ} - \cP_{\!\cJ} - \hcS_{\!\cJ} \|_{\op,\mbH} \lesssim (1+\gammaJ^{-1}(\theta_\tmax-\theta_\tmin) )(\gammaJ^{-1}\|\hcH - \cH\|_{\op,\mbH})^2$ 
    and $\| \hcP_{\!\cJ} - \cP_{\!\cJ} - \hcS_{\!\cJ} \|_{\HS,\mbH} \lesssim |\cJ| \halpha_{\!\cJ}^2$, respectively.
	In contrast, our results achieve sharper bounds, considering $\halpha_{\!\cJ} \leq \gammaJ^{-1}\|\hcH-\cH\|_{\op,\mbH}$ and $\hbeta_{\!\cJ} \leq \sqrt{|\cJ|} \halpha_{\!\cJ}$.
\end{remark}

\begin{remark}
	For simplicity, we adopt the threshold $1/4$ in the condition $\halpha_{\!\cJ}\le 1/4$.
	In fact, this condition can be relaxed to $\halpha_{\!\cJ} < 1/2$ at the expense of an additional factor $(1-2\halpha_{\!\cJ})^{-1}$ in the resulting bounds.
	This convention applies throughout the paper.
\end{remark}

\subsection{Eigenvectors of compact self-adjoint operators}

In this subsection, we study the first-order asymptotic expansion of the eigenvectors $\{\hpsi_k\}_{k \in \cJ}$ to  $\{\psi_k\}_{k \in \cJ}$.
Intuitively, as $\hcP_{\!\cJ}$ is close to $\cP_{\!\cJ} +\hcS_{\!\cJ}$, we expect 
$\hpsi_k = \hcP_{\!\cJ}\hpsi_k \approx \cP_{\!\cJ} \hpsi_k + \hcS_{\!\cJ} \hpsi_k \approx \cP_\cJ \hpsi_k + \hcS_{\!\cJ} \cP_{\!\cJ} \hpsi_k = \sum_{\ell \in \cJ} \langle \hpsi_k, \psi_\ell \rangle_\mbH \psi_\ell + \sum_{\ell \in \cJ} \langle \hpsi_k, \psi_\ell \rangle_\mbH \hcS_{\!\cJ} \psi_\ell$.
We thus may regard the term 
\begin{equation*}
	\sum_{\ell \in \cJ} \langle \hpsi_k, \psi_\ell \rangle_\mbH \hcS_{\!\cJ} \psi_\ell
	= \sum_{\ell \in \cJ} \langle \hpsi_k, \psi_\ell \rangle_\mbH \sum_{m \notin \cJ} \frac{ \langle (\hcH-\cH) \psi_\ell, \psi_m \rangle_\mbH}{\lambda_\ell - \lambda_m} \psi_m
\end{equation*}
as the leading first-order expansion of $\hpsi_k$ for $k \in \cJ$.
The following theorem formalizes this intuition.

\begin{theorem} \label{spec:eigvec}
	Suppose Assumptions \ref{spec:assump1} and \ref{spec:assump2} hold.
	If $\halpha_{\!\cJ} \leq 1/4$, we have
	\begin{equation}
		\sqrt{ \sum_{k \in \cJ} \bigg\| \hpsi_k - \sum_{\ell \in \cJ} \langle \hpsi_k, \psi_\ell \rangle_\mbH 
		\bigg( \psi_\ell + \sum_{m \notin \cJ} \frac{ \langle (\hcH-\cH) \psi_\ell, \psi_m \rangle_\mbH}{\lambda_\ell - \lambda_m}  \psi_m \bigg) \bigg\|_\mbH^2 }
		\leq \bigg( \frac{74\sqrt{D}}{9} + 8 \bigg) \halpha_{\!\cJ} \hbeta_{\!\cJ}. 
		\label{spec:eigvec-res1}		
	\end{equation}
\end{theorem}

In case of a simple eigenvalue, i.e., $\cJ=\{j\}$, the coefficients $\langle \hpsi_k, \psi_\ell \rangle_\mbH$ reduce to the single term $\langle \hpsi_j, \psi_j \rangle_\mbH$, which is expected to be close to one.
Consequently, a direct comparison between $\hpsi_j$ and $\psi_j$ becomes available, as stated below.

\begin{theorem} \label{spec:eigvec:simple}
	Under the same conditions as in Theorem~\ref{spec:eigvec}, if $\cJ = \{j\}$ and $\langle \hpsi_j, \psi_j \rangle \ge 0$, we have
	\begin{gather}
		\bigg\| \hpsi_j - \psi_j - \sum_{m: m\ne j} \frac{\langle (\hcH-\cH)\psi_j, \psi_m \rangle_\mbH}{\lambda_j-\lambda_m}\psi_m \bigg\|_\mbH
		\leq \frac{164}{9} \halpha_{\!\cJ} \hbeta_{\!\cJ}. 
		\label{spec:eigvec:simple-res1}
	\end{gather}
\end{theorem}

\begin{remark}
	Theorem~\ref{spec:eigvec:simple} provides sharper bounds than existing results based on classical perturbation bounds; see, e.g., Theorem~5.1.8 of \cite{hsing2015theoretical}.
\end{remark}

\subsection{Eigenvalues of compact self-adjoint operators}

In this subsection, we derive the first-order asymptotic expansions for the eigenvalues of $\hcH$ about the eigenvalues of $\cH$ under conditions analogous to those considered in the previous subsection.
We present two types of approximations, whose usefulness depends on whether we have 
well-separated {\em distinct} eigenvalues or clustered {\em distinct} eigenvalues, respectively.

More precisely, the first scenario addresses the case when the distinct eigenvalues $\{\theta_d\}_{d \in [D]}$ are sufficiently separated from each other, i.e., the inner spectral gaps $\{\gammaJd\}_{d\in[D]}$, defined by
\begin{equation*}
	\gammaJd = \min_{\ell \notin \cJ_d} |\theta_d - \lambda_\ell|, \quad d=1,\dots,D,
\end{equation*}
are not too small.
In this case, we investigate the vector of the differences $(\hlambda_k - \lambda_k)_{k \in \cJ}$ and show the norm of the remainder term can be bounded in terms of $\halpha_{\!\cJ_d}$ and $\hbeta_{\!\cJ_d}$, $d=1,\dots, D$, defined as in \eqref{spec:def-alpha} with $\cJ$ replaced by $\cJ_d$.

When $\{\theta_d\}_{d\in[D]}$ are close, that is, $\{\gammaJd\}_{d\in[D]}$ are small, the result from the former scenario may fail to be applicable due to the reciprocal dependency of $\halpha_{\!\cJ_d}$ and $\hbeta_{\!\cJ_d}$ on $\gammaJd$.
This issue is, for instance, reflected in the equality $\gammaJd^{-1} \|\cP_{\!\cJ_d} (\hcH-\cH) \cP_{\!\cJ_d} \|_{\HS,\mbH} \leq \halpha_{\!\cJ_d}$, which indicates that controlling $\halpha_{\!\cJ_d}$ would be difficult when $\gammaJd$ is too small.
See also Example~\ref{spec:perturb:ex1} how $\gammaJd$ can affect $\halpha_{\!\cJ_d}$ and $\hbeta_{\!\cJ_d}$.

In this case, we consider two alternative forms. 
The first form focuses on the difference of the sums of the eigenvalues $\sum_{k\in\cJ} \hlambda_k - \sum_{k\in\cJ} \lambda_k$, and show that it can be well approximated by the corresponding sum of the approximations, with dependency on $\gammaJ$ rather than on $\{\gammaJd\}_{d\in[D]}$.
The restriction to the sum reflects the fact that when eigenvalues form a cluster, the fluctuation of individual eigenvalues tends to cancel out, yielding a quantity that is more stable and can be appropriately approximated.

The second form is based on the belief that if the eigenvalues $\{\lambda_k\}_{k \in \cJ}$ form a cluster, then there exists a center point $c$ around which $\{\lambda_k\}_{k \in \cJ}$ are clustered.
In this case, we investigate the vector of the differences $(\hlambda_k - c)_{k \in \cJ}$ instead of $(\hlambda_k - \lambda_k)_{k \in \cJ}$.

The clustered eigenvalue scenario typically occurs when $\cH$ is a perturbed operator of some underlying operator, and $\cJ$ corresponds to a single eigenvalue with multiplicity of that original operator.
For instance, consider the convergence of the eigenvalues $\{\hlambda_{k,n,\e}\}_{k \in \cJ}$ of the kernel graph Laplacian $\hDelta_{n,\e}$ with kernel bandwidth $\e>0$ and sample size $n$, to the eigenvalues $\{\lambda_{k,\e}\}_{k \in \cJ}$ of the expected kernel graph Laplacian $\Delta_{\e} = \E (\hDelta_{n,\e})$.
If $\cJ$ with $|\cJ|>1$ corresponds to a single eigenvalue $\mu$ of the Laplace operator $\Delta = \lim_{\e \to 0} \Delta_\e$, then $\{\lambda_{k,\e}\}_{k \in \cJ}$ remain close to $\mu$ for small $\e$, forming a cluster.
Note that $\varepsilon$ usually depends on $n$, and it will tend to zero as $n$ increases.
This case will be treated in an accompanying paper that at this point is in preparation.

\subsubsection{Case 1: Well-separated eigenvalues}

In this regime, for each $d = 1,\ldots, D$,
the vector $(\hlambda_k - \lambda_k)_{k\in\cJ_d} = (\hlambda_k - \theta_d)_{k\in\cJ_d}$ 
can be well approximated by the vector of the eigenvalues of the matrix
\begin{equation}
	\left( \langle (\hcH-\cH) \psi_k, \psi_\ell \rangle_\mbH \right)_{k,\ell \in \cJ_d}.
\end{equation}
Concatenating these expansions across all $d \in [D]$ 
yields the first-order asymptotic expansion for the vector $(\hlambda_k - \lambda_k)_{k \in \cJ}$, as stated in the following theorem.

\begin{theorem} \label{spec:eigval-separate}
	Suppose Assumptions \ref{spec:assump1} and \ref{spec:assump2} hold.
	If $\max\limits_{d \in [D]} \halpha_{\!\cJ_d} \le 1/4$, we have
	\begin{multline}
		\Bigg\| (\hlambda_k - \lambda_k)_{k \in \cJ} - \bigoplus_{d=1}^D \specd \Big( \langle (\hcH-\cH) \psi_k, \psi_\ell \rangle_\mbH \Big)_{k,\ell \in \cJ_d} \Bigg\|_2
		\\ \leq \Bigg\| \bigg( \min \bigg\{ \frac{36+14\sqrt{2}}{9}\gammaJd\halpha_{\!\cJ_d}\hbeta_{\!\cJ_d} ,\,\, (3+\sqrt{2}) \gammaJd \hbeta_{\!\cJ_d}^2\bigg\} \bigg)_{d \in [D]} \Bigg\|_2, 
	\end{multline}
	where $\bigoplus$ denotes the direct sum of vectors.
\end{theorem}

\begin{remark}
    Our results extend and improve existing results for the case of a single eigenvalue ($D=1$). 
	For $|\cJ|=1$ with $\cJ=\{j\}$, Theorem~2 of \cite{wahl2026higher} provides
	$ \big|\, \hlambda_j - \lambda_j - \langle (\hcH-\cH) \psi_j, \psi_j \rangle_\mbH \,\big| \lesssim \gammaJ\hbeta_{\!\cJ}^2.$
	For $|\cJ|>1$, Theorem~5.2 of \cite{stankewitz2026higher} provides
	$$ \bigg\| (\hlambda_k - \lambda_k)_{k \in \cJ} - \specd \Big( \langle (\hcH-\cH) \psi_k, \psi_\ell \rangle_\mbH \Big)_{k,\ell \in \cJ} \bigg\|_2 \lesssim \|\cP_{\!\cJ} (\hcH-\cH)\cP_{\!\cJ}\|_{\HS,\mbH}+\gammaJ\halpha_{\!\cJ}\hbeta_{\!\cJ}. $$
    In contrast, we obtain the bound $\min\{ \gammaJ \halpha_{\!\cJ}\hbeta_{\!\cJ}, \gammaJ \hbeta_{\!\cJ}^2\}$, which coincides with the case $|\cJ|=1$ (since $\hbeta_{\!\cJ} \leq \halpha_{\!\cJ}$ in this case) and is sharper when $|\cJ|>1$.
	
	We note that $\halpha_{\!\cJ}$ is based on the OP norm, whereas $\hbeta_{\!\cJ}$ is based on the HS norm. 
	Since the HS norm usually dominates the OP norm, one may expect $\hbeta_{\!\cJ}$ to be larger than $\halpha_{\!\cJ}$ in many settings. However, the relative magnitudes of these two quantities depend on the specific situation.
\end{remark}

\subsubsection{Case 2: Clustered eigenvalues}

We consider the matrix
\begin{equation}
	\left( \langle (\hcH-\cH) \psi_k, \psi_\ell \rangle_\mbH \right)_{k,\ell \in \cJ}.
	\label{spec:eignval-cluster-matrix}
\end{equation}
The difference of the sum of eigenvalues is approximated by the trace of the matrix \eqref{spec:eignval-cluster-matrix} as follows.

\begin{theorem} \label{spec:eigval-cluster}
	Suppose Assumptions~\ref{spec:assump1} and \ref{spec:assump2} hold. If $\halpha_{\cJ} \le 1/4$, we have
	\begin{equation}
		\bigg| \sum_{k\in\cJ} \hlambda_k - \sum_{k\in\cJ} \lambda_k
		- \sum_{k \in \cJ} \langle (\hcH-\cH) \psi_k, \psi_k \rangle_\mbH \bigg|
		\\ \leq
        \bigg( \frac{5}{2} \gammaJ + 4 \theta_\tmax - 4 \theta_\tmin  \bigg)  \hbeta_{\!\cJ}^2. 
	\end{equation}
\end{theorem}

The following theorem considers cases when the eigenvalues $\{\lambda_k\}_{k \in \cJ}$ are tightly clustered near $c$.
In this context, we are able to directly compare the vector of the differences $(\hlambda_k - c)_{k \in \cJ}$ with the vector of eigenvalues of the matrix \eqref{spec:eignval-cluster-matrix}, rather than considering the sum of eigenvalues as in the previous theorem.

\begin{theorem} \label{spec:eigval-cluster2}
	Under the same conditions as in Theorem~\ref{spec:eigval-cluster}, we have for any $c \in \R$,
	\begin{multline}
		\bigg\| (\hlambda_k - c)_{k \in \cJ} - \specd \left( \langle (\hcH-\cH) \psi_k, \psi_\ell \rangle_\mbH \right)_{k,\ell \in \cJ} \bigg\|_2 \\ 
		\leq \min \bigg\{ \bigg( \Big(\frac{56\sqrt{2D}}{9} \halpha_\cJ + 4 \Big)\gammaJ + \frac{56(\sqrt{2}+1)\sqrt{D}}{9} \max_{d \in [D]} |\theta_d -c| \bigg) \halpha_{\!\cJ}\hbeta_{\!\cJ},  \\ 
        \bigg( (3+\sqrt{2})\gammaJ + (4\sqrt{2}+4) \max_{d \in [D]} |\theta_d - c| \bigg) \hbeta_{\!\cJ}^2 \bigg\} 
		+ \|(\lambda_k - c)_{k \in \cJ}\|_2.
		\label{spec:eigval-cluster-res}
	\end{multline}
\end{theorem}

We note that when $|\cJ|=1$ with $\cJ= \{j\}$ and $c=\lambda_j$, the above three theorems result in the same bounds up to universal constants, considering $\halpha_{\!\cJ} \le 1/4$ 
and $\hbeta_{\!\cJ} \leq \sqrt{|\cJ|}\halpha_{\!\cJ}$.

\section{First-order asymptotic expansions for spectral convergence of kernel Gram matrices and weak convergence} \label{mercer}

As an application of the results from Section \ref{spec}, we study the first-order asymptotic expansion describing the spectral convergence of kernel Gram matrices toward that of the corresponding integral operators.

A key technical difficulty arises from the fact that the two objects of interest lie in different spaces: one is a discrete matrix, and the other is a possibly continuous operator.
Following the approach of \cite{wahl2019perturbation}, we address this issue by restricting our attention to kernels satisfying Mercer's conditions and considering the associated reproducing kernel Hilbert space (RKHS).
This allows us to apply the spectral results developed in the previous section.

\subsection{Setting}
We begin by stating the main assumptions and notation used throughout this section.
\begin{assumption} \label{mercer:assump1}
	We assume the following:
	\begin{enumerate}[label=(\alph*)]
		\item $\M$ is a compact metric space equipped with a probability measure $\P$.
		\item $X, X_1, \dots, X_n$ are i.i.d. samples drawn from $\P$.
		\item The kernel function $h : \M \times \M \rightarrow \R$ is symmetric, continuous and positive semi-definite.
	\end{enumerate}
\end{assumption}

Under these assumptions, Mercer's theorem guarantees the existence of $\LP$-orthonormal eigenpairs $\{(\lambda_k, \phi_k)\}_{k \in \N}$ of the integral operator $\bfH: \LP \rightarrow \LP$, defined by
\begin{equation*}
	\bfH f(x) = \int_\M h(x,y) f(y) d\P(y), \quad f \in \LP.
\end{equation*}
Note that $\bfH$ is compact and self-adjoint. 
Moreover, the eigenvalues are non-negative and can be arranged in non-increasing order:
$\lambda_\tmax = \lambda_1 \geq \lambda_2 \geq \dots \geq 0.$
Each eigenvalue has finite multiplicity and is repeated accordingly in the sequence.

The kernel Gram matrix $\hbfH_n \in \R^{n \times n}$ corresponding to $h$ and $X_1,\dots,X_n$ is defined by
\begin{equation*}
	\hbfH_n = \left( \frac{h (X_i, X_{i'})}{n} \right)_{i,i' \in [n]}.
\end{equation*}
We denote by $ \{(\hlambda_{k,n}, \hvarphi_{k,n})\}_{k \in [n]}$ the eigenpairs of $\hbfH_n$ with eigenvalues arranged in non-decreasing order, repeated according to their multiplicities, and orthonormal eigenvectors normalized in $\LPn$-norm.

Analogously to Assumption \ref{spec:assump2}, we fix an index set $\cJ \subset \N$ as follows.

\begin{assumption} \label{mercer:assump2}
	$\cJ \subset [n]$ and the spectral gap $\gammaJ = \min\limits_{k \in \cJ, \ell \notin \cJ} |\lambda_k - \lambda_\ell|$ is strictly positive.
	Additionally, assume $\gammaJ\le\theta_\tmin$.
\end{assumption}

The assumption $\cJ \subset [n]$ is simply because $\hbfH_{n}$ has only $n$ eigenpairs.
Strict positivity of $\gammaJ$ in particular implies that the eigenvalues $\{\lambda_k\}_{k \in \cJ}$ are nonzero, and hence, part (a) of Lemma~\ref{mercer:duality} below can be applied to our analysis.
The additional condition $\gammaJ \leq \theta_\tmin$ is imposed to simplify the results; it is not restrictive and holds in almost all practical cases.
Indeed, since $\cJ$ is a finite set, there typically exists an eigenvalue $\lambda_{\ell}$ of $\bfH$ strictly smaller than any eigenvalues in $\{\lambda_k\}_{k \in \cJ}$, which subsequently implies $\gammaJ \leq \theta_\tmin$.

We again denote $\theta_\tmax = \theta_1 > \dots > \theta_D = \theta_\tmin$ the distinct eigenvalues of $\bfH$ in $\{\lambda_k\}_{k \in \cJ},$ and let $\cJ_d = \{k \in \cJ: \lambda_k = \theta_d \}$, $d=1,\dots,D$.

\subsection{Covariance operators in Reproducing kernel Hilbert spaces}

The RKHS $\mbH$ associated with the kernel $h$ is
\begin{equation*}
	\mbH = \bigg\{f = \sum_{k=1}^{\rank(\bfH)} a_k \phi_k : \sum_{k=1}^{\rank(\bfH)} \frac{a^2_k}{\lambda_k} < \infty \bigg\}
\end{equation*}
with inner product
\begin{equation*}
	\langle f, g \rangle_\mbH
	= \sum_{k=1}^{\rank(\bfH)} \frac{\langle f, \phi_k \rangle_\LP \langle g, \phi_k \rangle_\LP}{\lambda_k},
\end{equation*}
where $\rank(\bfH)$ denotes the number of non-zero eigenvalues of $\bfH$, possibly infinity. 
For $\rank(\bfH) = \infty$, the convergence of the sums is understood in the RKHS norm.
From the construction, it is clear that the dimension of $\mbH$ coincides with $\rank(\bfH)$.

As a set, $\mbH$ is a subset of $\LP$ and contains all eigenfunctions of $\bfH$ as well as $h(x,\cdot)$ for all $x \in \M$. 
Moreover, $\mbH$ satisfies the reproducing property: $\langle h (x, \cdot), f \rangle_\mbH = f(x)$ for any $f \in \mbH$ and $x \in \M$.

We now introduce the (uncentered) covariance operator and its empirical version $\cH, \hcH_n: \mbH \rightarrow \mbH$ defined by
\begin{equation*}
	\cH = \E h (X, \cdot) \otimes_\mbH h (X, \cdot), \quad
	\hcH_n = \frac{1}{n} \sum_{i=1}^{n} h (X_i, \cdot) \otimes_\mbH h (X_i, \cdot).
\end{equation*}
It can be easily shown that both $\cH$ and $\hcH_n$ are compact and self-adjoint.

The importance of the covariance operators stems from their connection to $\bfH$ and $\hbfH_n$. 
Indeed, by the reproducing property, it holds that $\cH f = \bfH f$ and $\sn{\hcH_n f} = \hbfH_n \sn{f}$ for any $f \in \mbH$.
This implies that studying the spectral properties of $\cH$ and $\hcH_n$ provides insight into those of $\bfH$ and $\hbfH_n$.
However, there is a subtle issue: $\mbH$ has a different inner product structure compared to $\LP$ and $\LPn$.
Taking this into account, the following lemma formulates the similarities.
This is a well-known result; for completeness, we provide the proof in \seesupp.
\begin{lemma} \label{mercer:duality}
	Under Assumption \ref{mercer:assump1}, the following statements hold:
	\begin{enumerate}[label=(\alph*)]
		\item For any $f, g \in \mbH$, we have $\langle \cH f, g \rangle_\mbH = \langle f, g \rangle_\LP$ and $\langle \hcH_n f, g \rangle_\mbH = \langle \sn{f}, \sn{g} \rangle_\LPn$. 
		In particular, $\langle (\hcH_n-\cH) f, g \rangle_\mbH = (\P_n-\P)(fg)$.
		
		\item
		For $1\leq k \leq \rank(\bfH)$, define  
		$\psi_k = \sqrt{\lambda_k} \phi_k.$
		Then, $\set{(\lambda_k, \psi_k): 1 \leq k \leq \rank(\bfH)}$ forms an $\mbH$-orthonormal family of eigenpairs of $\cH$.
		
		\item
		For $1\leq k \leq \rank(\hbfH_n)$ and $x \in \M$, define 
		$\hpsi_{k,n} (x) = \hlambda_{k,n}^{-1/2} \cdot \frac{1}{n} \sum_{i=1}^n h(x,X_i) \hvarphi_{k,n}(i),$
		where $\hvarphi_{k,n}(i)$ denotes the $i$th component of $\hvarphi_{k,n}$.
		Then, for such $k$, we have $\sn{\hpsi_{k,n}} = \sqrt{\hlambda_{k,n}} \hvarphi_{k,n}.$
		Moreover, $\{(\hlambda_{k,n}, \hpsi_{k,n}): 1 \leq k \leq \rank(\hbfH_n)\}$ forms an $\mbH$-orthonormal family of eigenpairs of $\hcH_n$.
	\end{enumerate}
\end{lemma}

Lemma \ref{mercer:duality} tells us that $\cH$ (resp. $\hcH_n$) shares the same eigenvalues with $\bfH$ (resp. $\hbfH_n$), 
and analogous arguments apply to the eigenfunctions of $\bfH$ and the eigenvectors of $\hbfH_n$ after scaling.
Therefore, in the subsequent subsections, we focus on the spectral convergence of $\hcH_n$ to $\cH$.

The advantage of studying the spectral convergence of $\hcH_n$ to $\cH$, rather than that of $\hbfH_n$ to $\bfH$, is that
$\hcH_n$ and $\cH$ act on the same space $\mbH$.
This makes direct comparison more natural.
Moreover, since both operators are compact and self-adjoint, the spectral results developed in the previous section can be directly applied.

\subsection{Convergence of weighted perturbation quantities}

We define $\halpha_{\!\cJ,n}$ and $\hbeta_{\!\cJ,n}$ as in Section~\ref{spec} with $\hcH$ replaced by $\hcH_n$.
To control these quantities, we formulate Bernstein-type inequalities as follows.

\begin{proposition} \label{mercer:perturb}
	Suppose Assumptions~\ref{mercer:assump1} and \ref{mercer:assump2} hold. Define
	\begin{align*}
		\bfkappa_{\!\cJ}
		&= \sup_{x \in \M} \bigg( \frac{1}{\gammaJ}\sum_{k \in \cJ} \lambda_k \phi_k^2(x) + \sum_{k \notin \cJ} \frac{ \lambda_k \phi_k^2 (x) }{\min_{m \in \cJ} |\lambda_k - \lambda_m|} \bigg), \\
		\bfsigma_{\!\cJ}
		&= \frac{1}{\gammaJ}\sum_{k \in \cJ} \lambda_k + \sum_{k \notin \cJ} \frac{ \lambda_k }{\min_{m \in \cJ} |\lambda_k - \lambda_m|}, \\
		\bfvarkappa_{\!\cJ}
		&= \sup_{x \in \M} \sqrt{ \frac{1}{\gammaJ}\sum_{k \in \cJ} \lambda_k \phi_k^2(x) \sum_{\ell \notin \cJ} \frac{ \lambda_\ell \phi_\ell^2 (x) }{\min_{m \in \cJ} |\lambda_\ell - \lambda_m|} }, \\
		\bfvarsigma_{\!\!\cJ}
		&= \E \Bigg[ \frac{1}{\gammaJ}\sum_{k \in \cJ} \lambda_k \phi_k^2(X) \sum_{\ell \notin \cJ} \frac{ \lambda_\ell \phi_\ell^2 (X) }{\min_{m \in \cJ} |\lambda_\ell - \lambda_m|} \Bigg].
	\end{align*}	
	For $\,0<\tau<1$, we have
	\begin{gather} 
		\P \bigg( \halpha_{\!\cJ,n} \le \sqrt{\frac{4\theta_\tmax \bfkappa_{\!\cJ} \log ( 2\bfsigma_{\!\!\cJ} /\tau ) }{n \gammaJ}} + \frac{2\bfkappa_{\!\cJ} \log ( 2\bfsigma_{\!\!\cJ} /\tau ) }{3n} \bigg) \geq 1- \tau,
		\label{mercer:perturb-res1} \\
		\P \bigg( \hbeta_{\!\cJ,n} 
		\le 
        \sqrt{\frac{2\bfvarsigma_\cJ \log(2/\tau)}{n}} +  \frac{\bfvarkappa_{\!\cJ} \log(2/\tau)}{3n}
        \bigg) \geq 1 -\tau. 
		\label{mercer:perturb-res2}
	\end{gather}	
	In particular, we have for $n \ge \frac{38 \theta_\tmax \bfkappa_{\!\cJ}}{\gammaJ}$,
	\begin{equation}
		\P \Big( \halpha_{\!\cJ,n} > 1/4 \Big) 
		\leq 2 \bfsigma_{\!\!\cJ} \exp \bigg( -\frac{n \gammaJ}{70 \theta_\tmax \bfkappa_{\!\cJ}} \bigg).
		\label{mercer:perturb-res3}
	\end{equation}	
\end{proposition}

We note that similar concentration inequalities for weighted perturbation quantities under single eigenvalues and additional assumptions on $\|\phi_k\|_\infty$ can be found in Proposition~2.8 of \cite{stankewitz2026higher}. 

$\bfsigma_{\!\!\cJ}$ is called the relative rank of the eigenvalues $\{\lambda_k\}_{k\in\cJ}$.
By construction, it holds that $\bfsigma_{\!\!\cJ} \leq \bfkappa_{\!\cJ}$, $\bfvarsigma_{\!\cJ} \leq \bfvarkappa_{\!\cJ}^2$ and $2 \bfvarkappa_{\!\cJ} \leq \bfkappa_{\!\cJ}$.
Also, using the classical quantity $\rho := \sup_{x \in \M} h(x,x) = \sup_{x \in \M} \sum_{k=1}^\infty \lambda_k \phi_k^2(x),$
we can bound $\bfsigma_{\!\!\cJ}, \bfkappa_{\!\cJ}, 2\bfvarkappa_{\!\cJ} \leq \gammaJ^{-1} \rho$ and $\bfvarsigma_{\!\cJ} \leq \gammaJ^{-2} \rho \sum_{k \in \cJ} \lambda_k$.

Below, we consider Gaussian kernels as a concrete example.

\begin{example}[Gaussian kernels] \label{mercer:perturb:ex1}
	Consider a Gaussian kernel $h(x,y) = \exp(-\|x-y\|_2^2)$ on the $m$-dimensional cube $\M = [0,1]^m$.
	According to Section~3.3 of \cite{dommel2025approximation}, with mild conditions on $\P$, the eigenvalues and eigenfunctions satisfy $\lambda_k \gtrsim \exp( -c(k+m)^{2/m})$ and $\sup_{x \in \M} |\phi_k (x)| \lesssim k^2$ for some constant $c$. 
	If we simply assume $m=2$, $\cJ=\{j\}$ and $\lambda_k \sim \exp(-k)$, it can be shown that
	$\bfkappa_{\!\cJ} \lesssim j^5$, $\bfsigma_{\!\!\cJ} \lesssim j$, $\bfvarkappa_{\!\cJ} \lesssim j^{9/2}$, $\bfvarsigma_{\!\cJ} \lesssim j^5$.
	These are strictly sharper than $\rho$-based bounds $\gammaJ^{-1} \rho$, $\gammaJ^{-2} \rho \sum_{k \in \cJ} \lambda_k \sim \exp(j)$. Detailed derivations are presented in \seesupp.
\end{example}

In what follows, we write for simplicity,
\begin{align*}
	\bfa_{\!\cJ,n}(\tau) &= \sqrt{\frac{4\theta_\tmax \bfkappa_{\!\cJ} \log ( 2\bfsigma_{\!\!\cJ} /\tau ) }{n\gammaJ}} + \frac{2\bfkappa_{\!\cJ} \log ( 2\bfsigma_{\!\!\cJ} /\tau)}{3n}, \\ 
	\bfb_{\!\cJ,n} (\tau) &= 
    \sqrt{\frac{2\bfvarsigma_\cJ \log(2/\tau)}{n}} +  \frac{\bfvarkappa_{\!\cJ} \log(2/\tau)}{3n}, 
\end{align*}
so $\halpha_{\!\cJ,n} \leq \bfa_{\!\cJ,n} (\tau) $ and $\hbeta_{\!\cJ,n} \leq \bfb_{\!\cJ,n}(\tau)$ with probability at least $1-\tau$.
We also use the similar notation $\bfa_{\!\cJ_d,n}(\tau)$ and $\bfb_{\!\cJ_d,n} (\tau)$, obtained by simply replacing $\cJ$ and $\theta_\tmax$ with $\cJ_d$ and $\theta_d$, respectively.

\subsection{Eigenprojections of kernel Gram matrices}

We define the eigenprojections
\begin{gather*}
	\hbfP_{\!\cJ,n} = \sum_{k \in \cJ} \hbfQ_{k,n}, \quad
	\bfP_{\!\cJ} = \sum_{k \in \cJ} \bfQ_k,		
	\quad\text{where}\hquad 
	\hbfQ_{k,n} =\hvarphi_{k,n} \otimes_\LPn \hvarphi_{k,n}, \hquad	
	\bfQ_k =\phi_k \otimes_\LP \phi_k.		
\end{gather*}
Our goal is to quantify the closeness $\hbfP_{\!\cJ,n}$ to $\bfP_{\!\cJ}$.
However, there are two main difficulties in this comparison.

The first difficulty is that $\hbfP_{\!\cJ,n}$ is an $n\times n$ matrix, whereas $\bfP_{\!\cJ}$ is an operator acting on $\LP$, and hence they cannot be compared directly.
This discrepancy can be resolved by considering their associated bilinear forms:
\begin{equation*}
	\langle \hbfP_{\!\cJ,n} \sn{f}, \sn{g} \rangle_\LPn \quad\text{and}\quad \langle \bfP_{\!\cJ} f, g \rangle_\LP.
\end{equation*}
This approach was introduced by \cite{koltchinskii1998asymptotics}.

The second difficulty is that the admissible choices of $f$ and $g$ in these bilinear forms are restricted.
Specifically, $f$ and $g$ must belong to $\mbH$ in order to relate the eigenprojections with the RKHS.
Moreover, they must be drawn from a bounded subset of $\mbH$.

Considering these issues, our first step is to connect the bilinear forms with the covariance operators $\cH$ and $\hcH_n$.
This representation is crucial because it allows us to reformulate the comparison problem in terms of perturbations of compact self-adjoint operators in $\mbH$.

\begin{lemma} \label{mercer:eigproj-duality}
	Suppose Assumptions \ref{mercer:assump1} and \ref{mercer:assump2} hold.
	For any $f,g \in \mbH$, we have
	\begin{equation}
		\langle \hbfP_{\!\cJ,n} \sn{f}, \sn{g} \rangle_\LPn - \langle \bfP_{\!\cJ} f, g \rangle_\LP
		= \langle ( \hcH_n \hcP_{\!\cJ,n} - \cH \cP_{\!\cJ}) f,g \rangle_\mbH.
	\end{equation}
	provided the eigenvalues $\{\hlambda_{k,n}\}_{k\in\cJ}$ are nonzero.
	Here, $\hcP_{\!\cJ,n}$ and $\cP_{\!\cJ}$ are the eigenprojections of $\hcH_n$ and $\cH$ corresponding to $\cJ$, respectively, defined by
	\begin{equation}
		\hcP_{\!\cJ,n} = \sum_{k \in \cJ} \hpsi_{k,n} \otimes_\mbH \hpsi_{k,n},
		\quad 
		\cP_{\!\cJ} = \sum_{k \in \cJ} \psi_k \otimes_\mbH \psi_k.
	\end{equation}
\end{lemma}

The lemma shows that the comparison of the bilinear forms reduces to controlling the operator difference 
$ \hcH_n \hcP_{\!\cJ,n} - \cH \cP_{\!\cJ}. $
See Lemma~\refsupp{mercer:eigproj-step1} in \seesupp\, how the difference can be approximated.

As a next step, we restrict attention to bounded function classes in $\mbH$.
Let $\msF \subset \mbH$ be a class of functions in $\mbH$ with 
$M_\msF : = \sup_{f \in \msF} \|f\|_\mbH < \infty.$
This restriction allows us to obtain uniform control of the deviations of the bilinear forms over $f,g \in \msF$.

Equipped with this setup, we state our main concentration result for the empirical eigenprojections.
The following theorem provides a high-probability bound on the deviation of the bilinear forms.

\begin{theorem} \label{mercer:eigproj-concen}
	Suppose Assumptions \ref{mercer:assump1} and \ref{mercer:assump2} hold and $n \ge \frac{38 \theta_\tmax \bfkappa_{\!\cJ}}{\gammaJ}.$
	For $\tau \in (0,1)$, we have with probability at least 
	$ 1-\tau - 2 \bfsigma_{\!\!\cJ} \exp \Big( -\frac{n \gammaJ}{70\theta_\tmax \bfkappa_{\!\cJ}} \Big),$
	\begin{multline}
		\sup_{(f,g) \in \msF\times\msF} \Big| \langle \hbfP_{\!\cJ,n} \sn{f}, \sn{g} \rangle_\LPn- \langle \bfP_{\!\cJ} f, g \rangle_\LP 
		- \langle \hbfUpsilon_{\!\cJ,n} f, g \rangle_\LP \Big|
		\\ 
        \leq 12 M_\msF^2 \theta_\tmax \bigg(\min\bigg\{D,\, 1+\frac{2\theta_\tmax-2\theta_\tmin}{\pi\gammaJ} \bigg\} \bigg) \bfa_{\!\cJ,n}^2(\tau), 
	\end{multline}
	where $ \hbfUpsilon_{\!\cJ,n} : \LP \rightarrow \LP$ is the random linear operator defined by
	\begin{equation}
		\langle \hbfUpsilon_{\!\cJ,n} \phi_k, \phi_\ell \rangle_\LP =
		\begin{cases}
			(\P_n - \P) (\phi_k \phi_\ell),
			& \text{ if } k, \ell \in \cJ,
			\\[.75em]
			\displaystyle\frac{\lambda_k (\P_n - \P) (\phi_k \phi_\ell)}{\lambda_k-\lambda_\ell},
			& \text{ if } k \in \cJ, \ell \notin \cJ,
			\\[1em]
			\displaystyle\frac{\lambda_\ell (\P_n - \P) (\phi_k \phi_\ell)}{\lambda_\ell-\lambda_k},
			& \text{ if } k \notin \cJ, \ell \in \cJ,
			\\[.75em]
			0,
			& \text{otherwise}.
		\end{cases}
	\end{equation}
\end{theorem}

We now turn to the weak convergence.
From the above theorem, we see that
$$ \sup_{(f,g) \in \msF\times\msF} \big| \langle \hbfP_{\!\cJ,n} \sn{f}, \sn{g} \rangle_\LPn- \langle \bfP_{\!\cJ} f, g \rangle_\LP 
- \langle \hbfUpsilon_{\!\cJ,n} f, g \rangle_\LP \big| = \OP (n^{-1}). $$
If $\sqrt{n} \, \hbfUpsilon_{\!\cJ,n}$ were to converge to some linear operator $\bfUpsilon_{\!\cJ}$,
then one expects
$$ \sqrt{n} \left( \langle \hbfP_{\!\cJ,n} \sn{f}, \sn{g} \rangle_\LPn- \langle \bfP_{\!\cJ} f, g \rangle_\LP \right) $$ 
to converge weakly to $\langle \bfUpsilon_{\!\cJ} f, g \rangle$ for $(f,g)  \in \msF\times\msF$.
To formalize this intuition, we introduce the necessary notions below.

$\mathscr{G}_\P$ denotes the generalized $\P$-Brownian bridge,
which is a centered Gaussian process indexed by functions in $\LP$ with covariance
$ \cov (\mathscr{G}_\P(f), \mathscr{G}_\P (g)) = \P(fg) - \P f \, \P g. $
Equivalently, $\mathscr{G}_\P$ arises as the weak limit of the empirical process $\sqrt{n} (\Pn - \P)$.

Let $\fB_{\msF \times \msF}$ be the collection of all functionals $F : \msF \times \msF \rightarrow \R$ such that
\begin{equation*}
	\|F\|_{\fB_{\msF \times \msF}} := \sup_{(f,g) \in \msF\times\msF} |F(f,g)| < \infty.
\end{equation*}
Note that ${\fB_{\msF\times\msF}}$ is a metric space under the metric $d_{\fB_{\msF\times\msF}} (F_1, F_2)= \|F_1 - F_2 \|_{\fB_{\msF\times\msF}}$.

Now, applying the weak convergence of bounded functionals (see Section 1.5 of \cite{vaart2023Weak}), we obtain the following result.

\begin{corollary} \label{mercer:eigproj-weak}
	Suppose Assumptions \ref{mercer:assump1} and \ref{mercer:assump2} hold.
	As $n \rightarrow \infty$, the sequence of random bilinear forms
	$ \sqrt{n} \big( \langle \hbfP_{\!\cJ,n} \sn{f}, \sn{g} \rangle_\LPn- \langle \bfP_{\!\cJ} f, g \rangle_\LP \big)$, $(f,g) \in \msF \times \msF$, converges weakly in ${\fB_{\msF \times \msF}}$ to the Gaussian process
	$ \langle \bfUpsilon_{\!\cJ} f, g \rangle_\LP$, $(f,g) \in \msF\times\msF,$
	where $ \bfUpsilon_{\!\cJ} : \LP \rightarrow \LP$ is the random linear operator defined by
	\begin{equation}
		\langle \bfUpsilon_{\!\cJ} \phi_k, \phi_\ell \rangle_\LP =
		\begin{cases}
			\msG_\P (\phi_k \phi_\ell),
			& \text{ if } k, \ell \in \cJ,
			\\[.75em]
			\displaystyle\frac{\lambda_k \msG_\P (\phi_k \phi_\ell)}{\lambda_k-\lambda_\ell},
			& \text{ if } k \in \cJ, \ell \notin \cJ,
			\\[1em]
			\displaystyle\frac{\lambda_\ell \msG_\P (\phi_k \phi_\ell)}{\lambda_\ell-\lambda_k},
			& \text{ if } k \notin \cJ, \ell \in \cJ,
			\\[.75em]
			0,
			& \text{otherwise}.
		\end{cases}
	\end{equation}
\end{corollary}

\begin{remark} \label{remark-KoltchinskiGine}
	The above weak convergence result is consistent with Theorem~2.2 of \cite{koltchinskii1998asymptotics}, up to a minor typo in formula (2.4) therein. 
	Their result requires strong control over the tail behavior of the spectral structure,
	such as $\sum_{i>R_n} \lambda_i^2 = o(n^{-1})$ or
	$ \big( \sum_{k,\ell \in [R_n]} \P(\phi_k^2 \phi_\ell^2) \big) \big( \sum_{k,\ell \in [R_n]} (\lambda_k^2 + \lambda_\ell^2) \P(\phi_k^2 \phi_\ell^2) \big) = o(n)$ 
	for some sequence of natural numbers $R_n \to \infty$.
	On the other hand, our result relies only on Mercer's conditions.
	
	A similar comparison holds for Corollary~\ref{mercer:eigval-weak} below (dealing with eigenvalues) with Theorem~5.1 of \cite{koltchinskii2000random}.
\end{remark}

\subsection{Eigenvectors of kernel Gram matrices}

In this subsection, we derive the first-order expansion regarding eigenvectors of kernel Gram matrices.
As seen in the previous subsection, the eigenvectors $\{\hvarphi_{k,n}\}_{k \in \cJ}$ are $n$-dimensional vectors; 
whereas the corresponding parts $\{\phi_{k,n}\}_{k \in \cJ}$ are functions in $\LP$.
Considering this discrepancy, we focus on the sum of the squared inner products
$$\sum_{k \in \cJ} \sum_{\ell \in \cJ} \langle \hvarphi_{k,n}, \sn{\phi_\ell} \rangle^2_\LP.$$

With this, we formulate the first-order expansion as follows.

\begin{theorem} \label{mercer:eigvec-concen}
	Suppose Assumptions \ref{mercer:assump1} and \ref{mercer:assump2} hold and $n \ge \frac{38 \theta_\tmax \bfkappa_{\!\cJ}}{\gammaJ}.$
	For $\tau \in (0,1)$, we have with probability at least 
	$ 1-\tau - 2 \bfsigma_{\!\!\cJ} \exp \Big( -\frac{n \gammaJ}{70\theta_\tmax \bfkappa_{\!\cJ}} \Big),$
	\begin{equation}
		\Bigg| \sum_{k \in \cJ} \sum_{\ell \in \cJ} \langle \hvarphi_{k,n}, \sn{\phi_\ell} \rangle_\LP^2 - |\cJ| - (\P_n-\P) \Big( \sum_{k \in \cJ} \phi_k^2 \Big) \Bigg|
		\\
		\leq 
        \frac{11}{2}\bfb_{\!\cJ,n}^2 (\tau). 
	\end{equation}
\end{theorem}

The following weak convergence result is immediate.

\begin{corollary} \label{mercer:eigvec-weak}
	Suppose Assumptions \ref{mercer:assump1} and \ref{mercer:assump2} hold.
	As $n \rightarrow \infty$, the sequence of random variables
	$ \sqrt{n} \bigg( \sum\limits_{k \in \cJ} \sum\limits_{\ell \in \cJ} \langle \hvarphi_{k,n}, \sn{\phi_\ell} \rangle_\LP^2 - |\cJ| \bigg)$ converges weakly to $\msG_\P \Big(\sum\limits_{k \in \cJ} \phi_k^2 \Big)$.
\end{corollary}

\begin{remark}
	The above weak convergence result can also be derived from Corollary~\ref{mercer:eigproj-weak} by taking $\msF = \{\phi_k\}_{k \in \cJ}$, $f=g=\phi_\ell$ and summing over $\ell \in \cJ$.
	The same argument can be applied to Theorem~\ref{mercer:eigvec-concen} and Theorem~\ref{mercer:eigproj-concen}. 
	However, in this case, the obtained bound from Theorem~\ref{mercer:eigproj-concen} may not be sharp, considering $\hbeta_{\!\cJ}\leq \sqrt{|\cJ|} \halpha_{\!\cJ}$.
\end{remark}

\subsection{Eigenvalues of kernel Gram matrices}

As shown in Lemma \ref{mercer:duality}, $\cH$ (resp. $\hcH_n$) shares the same eigenvalues with $\bfH$ (resp. $\hbfH_n$).
Consequently, the convergence of the eigenvalues of $\hcH_n$ to those of $\cH$ immediately implies the corresponding convergence of $\hbfH_n$ to those of $\bfH$.
Therefore, by directly applying Theorems~\ref{spec:eigval-separate} and \ref{spec:eigval-cluster} together with Proposition~\ref{mercer:perturb}, we obtain the following results for the cases of well-separated and clustered eigenvalues.

\begin{theorem} \label{mercer:eigval-concen}
	Under Assumptions \ref{mercer:assump1} and \ref{mercer:assump2}, the following hold:
	\begin{enumerate}[leftmargin=1em, label=(\alph*)]
		\item
		If $n \ge \max\limits_{d \in [D]} \frac{38 \theta_d \bfkappa_{\!\cJ_d}}{\gammaJd}$,
		we have with probability at least 
		$1-\tau-\sum\limits_{d=1}^D 2 \bfsigma_{\!\!\cJ_d} \exp \Big( -\frac{n\gammaJd}{70 \theta_d \bfkappa_{\!\cJ_d}} \Big),$
		\begin{multline}
			\bigg\| (\hlambda_{k,n} - \lambda_k)_{k \in \cJ} - \bigoplus_{d=1}^D \specd \Big( \theta_d (\P_n-\P) (\phi_k \phi_\ell) \Big)_{k,\ell \in \cJ_j} \bigg\|_2
			\\ \leq 
			\Bigg\| \bigg( \min \bigg\{ \frac{36+14\sqrt{2}}{9}\gammaJd \bfa_{\!\cJ_d,n} \Big(\frac{\tau}{2D}\Big) \bfb_{\!\cJ_d,n} \Big(\frac{\tau}{2D}\Big) ,\,\, (3+\sqrt{2}) \gammaJd \bfb_{\!\cJ_d,n}^2 \Big(\frac{\tau}{2D}\Big) \bigg\} \, \bigg)_{d \in [D]} \Bigg\|_2 .
		\end{multline}
		
		\item
		If $n \ge \frac{38 \theta_\tmax \bfkappa_{\cJ}}{\gammaJ}$, we have
		with probability at least $1-\tau- 2 \bfsigma_{\!\!\cJ} \exp \Big( -\frac{n\gammaJ}{70 \theta_\tmax \bfkappa_{\!\cJ_d}} \Big),$
		\begin{equation}
			\bigg| \sum_{k\in\cJ} \hlambda_{k,n} - \sum_{k\in\cJ} \lambda_k
			- (\Pn-\P) \Big( \sum_{k \in \cJ} \lambda_k \phi_k^2 \Big) \bigg| 
            \bigg( \frac{5}{2} \gammaJ + 4 \theta_\tmax - 4 \theta_\tmin  \bigg) \bfb_{\!\cJ,n}^2. 
		\end{equation}
	\end{enumerate}
\end{theorem}

Similar to eigenvectors, the weak convergence of eigenvalues follows directly.

\begin{corollary} \label{mercer:eigval-weak}
	Under Assumptions \ref{mercer:assump1} and \ref{mercer:assump2}, the following hold as $n\to\infty$:
	\begin{enumerate}[leftmargin=1em, label=(\alph*)]
		\item The sequence of vectors
		$\sqrt{n} (\hlambda_{k,n} - \lambda_k)_{k \in \cJ}$
		converges weakly to
		$\bigoplus\limits_{d=1}^D \specd \Big( \theta_d \msG_\P \left( \phi_k \phi_\ell \right) \Big)_{k,\ell \in \cJ_d}.$
		
		\item The sequence
		$\sqrt{n} \sum\limits_{k\in\cJ} (\hlambda_{k,n} - \lambda_k)$ converges weakly to $\msG_\P \Big( \sum\limits_{k \in \cJ} \lambda_k \phi_k^2 \Big).$
	\end{enumerate}	
\end{corollary}

We refer to Remark~\ref{remark-KoltchinskiGine} for connecting this result to the results in \cite{koltchinskii2000random}.

\subsection{Possible extension to random index sets} \label{mercer:rdidx}

Throughout the previous results, the index set $\cJ$ has been treated as fixed.
In practice, however, $\cJ$ may need to be estimated from the data. 
For instance, suppose we are interested in estimating $\theta_2$, the second largest distinct eigenvalue, but we do not know whether the largest eigenvalue $\theta_1$ has multiplicity one or not, and we also might similarly not know the multiplicity of $\theta_2$. 
If we would know both of these multiplicities to be, say $J_1$ and $J_2$, then we estimate $\theta_2$ by using $\hlambda_j$ for $j = J_1+1,\ldots,J_1+J_2$.
Not knowing $J_2$ prevents us from even using a single estimate for $\theta_2.$
In this case, we could attempt to estimate both $\cJ_1$ and $\cJ_2.$

We can still obtain finite sample approximation bounds when using an estimator of the unknown target index set.
This is formalized in the following.

Let $\hcJ_n$ be a consistent estimator of $\cJ$, in the sense that
\begin{equation} 
	\P(\hcJ_n \ne \cJ) = p_n \to 0 \quad\text{as}\quad n \to \infty.
	\label{mercer:rdidx-assump}
\end{equation}
For methods related to consistent index set estimation, see \cite{hall2009tierespecting}.
See also Section 6.1.4 of \cite{jolliffe2002principal} for related sequential testing procedures.
Under \eqref{mercer:rdidx-assump}, our asymptotic results continue to hold with $\cJ$ replaced by $\hcJ_n$.
As an example, we illustrate this with the case of the sum of the eigenvalues; analogous extensions follow in the same way.

\begin{corollary}
	Under condition \eqref{mercer:rdidx-assump}, part (b) of Theorem~\ref{mercer:eigval-concen} with $\sum_{k\in\cJ} (\hlambda_{k,n} - \lambda_k)$ replaced by $\sum_{k\in\hcJ_n} (\hlambda_{k,n} - \lambda_k)$ holds with probability at least
	\begin{equation*}
		1-\tau- 2 \bfsigma_{\!\!\cJ} \exp \Big( -\frac{n\gammaJ}{70 \theta_\tmax \bfkappa_{\!\cJ_d}} \Big) - p_n.
	\end{equation*}	
	Similarly, part (b) of Corollary~\ref{mercer:eigval-weak} holds when $\sqrt{n} \sum_{k\in\cJ} (\hlambda_{k,n} - \lambda_k)$ is replaced by $\sqrt{n} \sum_{k\in\hcJ_n} (\hlambda_{k,n} - \lambda_k)$.
\end{corollary}

\section{Proofs of Section~\ref{spec}} \label{spec:proof}

\subsection{First-order expansion of generalized compression operators} \label{spec:proof:cmp}

To facilitate the proofs of the main results in Section~\ref{spec}, we introduce generalized compression operators. 
Given a complex-valued function $f$, we define
\begin{equation}
	f_{\cJ} (\cH) = \sum_{k \in \cJ} f(\lambda_k) \cQ_k.
\end{equation}
This definition generalizes usual spectral compression operators in two important cases:
\begin{itemize}
	\item For $f \equiv 1$, $f_\cJ (\cH) = \cP_{\!\cJ}$.
	\item For $f(z) = z$, $ f_\cJ (\cH) = \sum_{k \in \cJ} \lambda_k \cQ_k = \cH \cP_{\!\cJ}, $
	which coincides with the canonical compression of $\cH$ onto the eigenspace associated with $\{ \lambda_k\}_{k \in \cJ}$.	
\end{itemize}
Analogously, we define $ f_\cJ (\hcH) = \sum_{k \in \cJ} f(\hlambda_k) \hcQ_k. $
Our goal is to derive the first-order asymptotic expansion of the difference $f_\cJ (\hcH)- f_\cJ (\cH)$ under weighted perturbation.
To this end, we define $\nabla_{\cH} f_\cJ(\hcH) = \sum_{k=1}^3 \nabla_{\cH,k} f_\cJ(\hcH)$ with
\begin{align*}
	\nabla_{\cH,1} f_\cJ(\hcH)
	&= \sum_{d=1}^D f'(\theta_d) \cP_{\!\cJ_d} (\hcH-\cH) \cP_{\!\cJ_d}, 
	\\ \nabla_{\cH,2} f_\cJ(\hcH)
	&= \sum_{\substack{1\leq d_1,d_2 \leq D \\ d_1 \not= d_2}} \frac{f(\theta_{d_2}) - f(\theta_{d_1})}{\theta_{d_2}-\theta_{d_1}}  \cP_{\!\cJ_{d_1}} (\hcH - \cH) \cP_{\!\cJ_{d_2}}, 
	\\ \nabla_{\cH,3} f_\cJ(\hcH)
	&= \sum_{k \in \cJ} f(\lambda_k) \sum_{\ell \notin \cJ} \frac{\cQ_k (\hcH-\cH) \cQ_\ell+ \cQ_\ell (\hcH-\cH) \cQ_k}{\lambda_k - \lambda_\ell}.
\end{align*}
$\nabla_{\cH} f_\cJ(\hcH)$ can be viewed as the derivative of $f_\cJ (\hcH)$ in the direction $\cH$.
With this, we establish an upper bound for the OP norm of the remainder term as follows. 
The detailed proof is deferred to \seesupp.

\begin{proposition} \label{spec:comp1}
	Suppose Assumptions~\ref{spec:assump1} and \ref{spec:assump2} hold.
	Let $\partial\Gamma_{\!\cJ} \subset \C$ denote the positively oriented boundary of the set
	$\Gamma_{\!\cJ} := \bigcup_{d=1}^D \{z \in \C: |z-\theta_d| < \frac{\gammaJ}{2} \}$.
	If $f:\C \to \C$ is holomorphic on a neighborhood of $\Gamma_{\!\cJ}$ and $\halpha_{\!\cJ} \le 1/4$,
	we have 
	\begin{equation}
		\| f_\cJ(\hcH) - f_\cJ(\cH) - \nabla_\cH f_\cJ(\hcH) \|_{\op,\mbH}
		\le 8 \bigg(\sup_{z \in \partial\Gamma_{\!\cJ}} |f(z)| \bigg) \bigg(\min\bigg\{D,\, 1+\frac{2\theta_\tmax-2\theta_\tmin}{\pi\gammaJ} \bigg\} \bigg) \halpha_{\!\cJ}^2.
		\label{spec:comp1-res1}
	\end{equation}
	In particular, the following holds without the condition $\halpha_{\!\cJ} \le 1/4$:
	\begin{equation}
		\|\nabla_\cH f_\cJ(\hcH) \|_{\op,\mbH}
		\leq 2 \bigg(\sup_{z \in \partial\Gamma_{\!\cJ}} |f(z)| \bigg) \bigg(\min\bigg\{D,\, 1+\frac{2\theta_\tmax-2\theta_\tmin}{\pi\gammaJ} \bigg\} \bigg) \halpha_{\!\cJ}.
		\label{spec:comp1-res2}
	\end{equation}
\end{proposition}

While the previous proposition provides a general bound in terms of the OP norm, 
it is also necessary to control the HS norm and trace of the remainder term.
To avoid unnecessary technical complexities, and given that our primary applications only require specific choices of $f$, 
we focus on these tractable cases in the following propositions. 
The proofs are deferred to \seesupp.

\begin{proposition} \label{spec:comp2}
	Under Assumptions~\ref{spec:assump1} and \ref{spec:assump2}, if $f \equiv 1$ and $\halpha_{\!\cJ} \le 1/4$, we have
	\begin{align}
		\big\|[f_\cJ(\hcH) - f_\cJ(\cH)- \nabla_\cH f_\cJ(\hcH)]\cP_{\!\cJ} \big\|_{\HS,\mbH} 
		&\leq \bigg( \frac{56\sqrt{D}}{9} + 8\bigg) \halpha_{\!\cJ} \hbeta_{\!\cJ},
		\\
		\big\| f_\cJ(\hcH) - f_\cJ(\cH)- \nabla_\cH f_\cJ(\hcH) \big\|_{\HS,\mbH} 
		&\leq \bigg( \frac{112\sqrt{D}}{9} + 16\bigg) \halpha_{\!\cJ} \hbeta_{\!\cJ},
	\end{align}
\end{proposition}

\begin{proposition} \label{spec:comp3}
	Under Assumptions~\ref{spec:assump1} and \ref{spec:assump2}, if $f(z)= \xi_1 z + \xi_2$ with $\xi_1,\xi_2 \in \C$ and $\halpha_{\!\cJ} \le 1/4$, we have
	\begin{multline}
		\big\| \cP_{\!\cJ} [f_\cJ(\hcH) - f_\cJ(\cH)- \nabla_\cH f_\cJ (\hcH)] \cP_{\!\cJ} \big\|_{\HS,\mbH}
		\\ \le 
		\min \bigg\{\, \bigg( 4 |\xi_1| \gammaJ + \frac{56 \sqrt{D}}{9} \max_{d \in [D]} |\xi_1 \theta_d + \xi_2| \bigg) \halpha_{\!\cJ} \hbeta_{\!\cJ}
		,\,\,
		\Big( 3 |\xi_1| \gammaJ + 4 \max_{d \in [D]} |\xi_1 \theta_d + \xi_2| \Big) \hbeta_{\!\cJ}^2 \bigg\},
		\label{spec:comp3-res1}
	\end{multline}
	\begin{align}
		\big| \tr \big( f_\cJ(\hcH) - f_\cJ(\cH)- \nabla_\cH f_\cJ (\hcH) \big) \big|
		&\leq \bigg( \frac{5}{2}|\xi_1| \gammaJ + 8\max_{d \in [D]} |\xi_1 \theta_d + \xi_2| \bigg)  \hbeta_{\!\cJ}^2.
		\label{spec:comp3-res2}
	\end{align}
	In particular, for $\xi_1 =1$, $\xi_2 =0$ and $\Lambda_{\!\cJ}^{-1/2}:= \diag(\lambda_k^{-1/2})_{k \in \cJ}$, we have
	\begin{align}
        \big| \tr\big( \Lambda_{\!\cJ}^{-1/2} \cP_{\!\cJ} [f_\cJ(\hcH) - f_\cJ(\cH)- \nabla_\cH f_\cJ(\hcH)] \cP_{\!\cJ} \Lambda_{\!\cJ}^{-1/2} \big) \big|
		\leq \Big( \frac{3\gammaJ}{2\theta_\tmin} + 4 \Big) \hbeta_{\!\cJ}^2. 
		\label{spec:comp3-res3}		
	\end{align}
\end{proposition}

\subsection{Proof of Theorem~\ref{spec:eigproj}}

\begin{proof}[Proof of Theorem~\ref{spec:eigproj}]
	We first show \eqref{spec:eigproj-res1}.
	For $f \equiv 1$, it can be easily shown that
	\begin{equation}
		f_\cJ(\hcH) = \hcP_{\!\cJ}, \quad f_\cJ(\cH) = \cP_{\!\cJ}, \quad
		\nabla_{\cH} f_\cJ(\hcH)
		= \sum_{k=1}^3 \nabla_{\cH,k} f_\cJ (\hcH)
		= 0 + 0 + \hcS_{\!\cJ} = \hcS_{\!\cJ}.
		\label{spec:eigproj-eq1}
	\end{equation}
	Now, by applying Proposition~\ref{spec:comp1}, we see that \eqref{spec:eigproj-res1} holds for $\halpha_{\!\cJ} \le 1/4$.
	
	Note that $\| \hcP_{\!\cJ} \|_{\op,\mbH} = \|\cP_{\!\cJ} \|_{\op,\mbH} = 1$.
	Hence, by a simple calculation, we see that \eqref{spec:eigproj-res1} also holds for $\halpha_{\!\cJ} > 1/4$.
	
	\eqref{spec:eigproj-res2} follows from \eqref{spec:eigproj-eq1} and Proposition~\ref{spec:comp2}.
	This completes the proof.
\end{proof}

\subsection{Proof of Theorems~\ref{spec:eigvec} and \ref{spec:eigvec:simple}}

Before proceeding, we state a lemma to prove Theorem~\ref{spec:eigvec}.
The proof is provided in \seesupp.

\begin{lemma} \label{spec:eigproj:aux}
	Under Assumptions~\ref{spec:assump1} and \ref{spec:assump2},
	if $\halpha_{\!\cJ} \le 1/4$, then
	\begin{equation}
		\|\hcS_{\!\cJ} (\cI_\mbH- \cP_{\!\cJ}) \hcP_{\!\cJ}\|_{\HS,\mbH}
		\le 2 \sqrt{D} \halpha_{\!\cJ} \hbeta_{\!\cJ}.
	\end{equation}
\end{lemma}

With this lemma, we prove Theorem~\ref{spec:eigvec} as follows.

\begin{proof}[Proof of Theorem~\ref{spec:eigvec}]
	Note that for all $k \in \cJ$, we have $\hcP_{\!\cJ} \hpsi_k = \hpsi_k$, $\cP_{\!\cJ} \hpsi_k = \sum_{\ell \in \cJ} \langle \hpsi_k, \psi_\ell \rangle_\mbH \psi_\ell$ and 
	\begin{align*}
		\hcS_{\!\cJ} \cP_{\!\cJ} \hpsi_k
		= \sum_{\ell \in \cJ} \langle \hpsi_k, \psi_\ell \rangle_\mbH \hcS_{\!\cJ} \psi_\ell
		= \sum_{\ell \in \cJ} \langle \hpsi_k, \psi_\ell \rangle_\mbH \sum_{m \notin \cJ} \frac{ \langle (\hcH-\cH) \psi_\ell, \psi_m \rangle_\mbH}{\lambda_\ell - \lambda_m}  \psi_m.
	\end{align*}
	Consequently, we obtain
	\begin{multline}
		\sum_{k \in \cJ} \bigg\| \hpsi_k - \sum_{\ell \in \cJ} \langle \hpsi_k, \psi_\ell \rangle_\mbH \bigg(\psi_\ell
		+ \sum_{m \notin \cJ} \frac{ \langle (\hcH-\cH) \psi_\ell, \psi_m \rangle_\mbH}{\lambda_\ell - \lambda_m}  \psi_m \bigg) \bigg\|_\mbH^2
		\\= \sum_{k \in \cJ} \| (\hcP_{\!\cJ} - \cP_{\!\cJ} - \hcS_{\!\cJ} \cP_{\!\cJ} ) \hpsi_k \|_\mbH^2
		= \| (\hcP_{\!\cJ} - \cP_{\!\cJ} - \hcS_{\!\cJ} \cP_{\!\cJ} ) \hcP_{\!\cJ} \|_{\HS,\mbH}^2.
		\label{spec:eigvec-eq1}
	\end{multline}
	Hence, to complete the proof, it suffices to bound $\| (\hcP_{\!\cJ} - \cP_{\!\cJ} - \hcS_{\!\cJ} \cP_{\!\cJ} ) \hcP_{\!\cJ} \|_{\HS,\mbH}$.
	
	Using $\cP_{\!\cJ} = \cI_\mbH - (\cI_\mbH- \cP_{\!\cJ})$, 
	we separate and bound
	\begin{multline*}
		\|(\hcP_{\!\cJ} - \cP_{\!\cJ} - \hcS_{\!\cJ} \cP_{\!\cJ}) \hcP_{\!\cJ}\|_{\HS,\mbH}
		\leq \|(\hcP_{\!\cJ} - \cP_{\!\cJ} - \hcS_{\!\cJ}) \hcP_{\!\cJ}\|_{\HS,\mbH} 
		+ \|\hcS_{\!\cJ} (\cI_\mbH- \cP_{\!\cJ}) \hcP_{\!\cJ}\|_{\HS,\mbH}.
	\end{multline*}
	Applying Proposition~\ref{spec:comp2} and \eqref{spec:eigproj-eq1}, we have
	\begin{equation}
		\|(\hcP_{\!\cJ} - \cP_{\!\cJ} - \hcS_{\!\cJ})\cP_{\!\cJ} \|_{\HS,\mbH} 
		\leq \bigg( \frac{56\sqrt{D}}{9}+8\bigg) \halpha_{\!\cJ}\hbeta_{\!\cJ}.
		\label{spec:eigvec-eq3}
	\end{equation}
	Therefore, additionally applying Lemma~\ref{spec:eigproj:aux}, we obtain
	\begin{align}
		\| (\hcP_{\!\cJ} - \cP_{\!\cJ} - \hcS_{\!\cJ} \cP_{\!\cJ} ) \hcP_{\!\cJ} \|_{\HS,\mbH} 
		\leq \bigg( \frac{74\sqrt{D}}{9} + 8 \bigg) \halpha_{\!\cJ} \hbeta_{\!\cJ}.
		\label{spec:eigvec-eq2}
	\end{align}
	The proof is complete by combining \eqref{spec:eigvec-eq1} and \eqref{spec:eigvec-eq2}.
\end{proof}

To prove Theorem~\ref{spec:eigvec:simple}, we need to control the inner product $\langle \hpsi_j, \psi_j \rangle_\mbH$.
For this purpose, we state the following lemma.
The result is stronger than needed here but will be used later for eigenvalues.

\begin{lemma} \label{spec:eigvec:aux1}
	Suppose Assumptions~\ref{spec:assump1} and \ref{spec:assump2} hold. 
	Let $\hPsi_{\!\cJ} = \big( \langle \psi_k, \hpsi_\ell \rangle_\mbH \big)_{k,\ell \in \cJ}$.
	If $\halpha_{\!\cJ} \le 1/4$, we have
	\begin{equation}
		\big\| \hPsi_{\!\cJ} \hPsi_{\!\cJ}^\top - I_{|\cJ|} \big\|_\F
		\leq \min \bigg\{ \frac{56\sqrt{D}}{9} \halpha_{\!\cJ}\hbeta_{\!\cJ} ,\,\, 4 \hbeta_{\!\cJ}^2\bigg\}.
	\end{equation}
	In particular, for $\cJ=\{j\}$,
	\begin{equation}
		|\langle \hpsi_j, \psi_j \rangle_\mbH^2 -1 | 
		\leq 4 \hbeta_{\!\cJ}^2
		\leq 4 \halpha_{\!\cJ} \hbeta_{\!\cJ}. 
	\end{equation}
\end{lemma}

\begin{proof}[Proof of Lemma~\ref{spec:eigvec:aux1}]
	The second claim follows from the first, together with $\hPsi_{\!\cJ} = (\langle\hpsi_j, \psi_j\rangle_\mbH) \in \R^{1 \times 1}$ and $\hbeta_{\!\cJ} \leq \halpha_{\!\cJ}$ when $\cJ=\{j\}$.
	
	To show the first claim, consider $\langle \hcP_{\!\cJ} \psi_k, \psi_\ell \rangle_\mbH$ for $k,\ell \in \cJ$.
	Using $\hcP_{\!\cJ} \psi_k = \sum_{m \in \cJ} \langle \psi_k, \hpsi_m \rangle_\mbH \hpsi_m,$
	we express $\langle \hcP_{\!\cJ} \psi_k, \psi_\ell \rangle_\mbH$ as
	\begin{equation} 
		\langle \hcP_{\!\cJ} \psi_k, \psi_\ell \rangle_\mbH
		= \sum_{m \in \cJ} \langle \psi_k, \hpsi_m \rangle_\mbH 
		\langle \hpsi_m, \psi_\ell \rangle_\mbH,
		\label{spec:eigvec:aux1-eq1}
	\end{equation}
	which is the $(k,\ell)$-entry of $\hPsi_{\!\cJ} \hPsi_{\!\cJ}^\top$.
	On the other hand, using the identities
	\begin{equation*}
		\hcP_{\!\cJ} = \cP_{\!\cJ} + \hcS_{\!\cJ} + (\hcP_{\!\cJ} - \cP_{\!\cJ} - \hcS_{\!\cJ}), \quad 
		\langle \cP_{\!\cJ} \psi_k, \psi_\ell \rangle_\mbH = \delta_{k \ell}, \quad
		\langle \hcS_{\!\cJ} \psi_k, \psi_\ell \rangle_\mbH =0,
	\end{equation*}
	we also express $\langle \hcP_{\!\cJ} \psi_k, \psi_\ell \rangle_\mbH$ as
	\begin{align}
		\langle \hcP_{\!\cJ} \psi_k, \psi_\ell \rangle_\mbH
		&= \langle \cP_{\!\cJ} \psi_k, \psi_\ell \rangle_\mbH
		+ \langle \hcS_{\!\cJ} \psi_k, \psi_\ell \rangle_\mbH
		+ \langle (\hcP_{\!\cJ} - \cP_{\!\cJ} - \hcS_{\!\cJ}) \psi_k, \psi_\ell \rangle_\mbH
		\notag\\&= \delta_{k\ell}
		+ \langle (\hcP_{\!\cJ} - \cP_{\!\cJ} - \hcS_{\!\cJ}) \psi_k, \psi_\ell \rangle_\mbH.
		\label{spec:eigvec:aux1-eq2}
	\end{align}
	By combining \eqref{spec:eigvec:aux1-eq1} and \eqref{spec:eigvec:aux1-eq2}, we obtain
	\begin{equation*}
		\hPsi_{\!\cJ} \hPsi_{\!\cJ}^\top - I_{|\cJ|}
		= \big( \langle (\hcP_{\!\cJ} - \cP_{\!\cJ} - \hcS_{\!\cJ}) \psi_k, \psi_\ell \rangle_\mbH \big)_{k,\ell \in \cJ}.
	\end{equation*}
	Therefore, the claim follows by Proposition~\ref{spec:comp3} and \eqref{spec:eigproj-eq1} by letting $ f\equiv 1$.
\end{proof}

With this lemma, we prove Theorem~\ref{spec:eigvec:simple} below.

\begin{proof}[Proof of Theorem~\ref{spec:eigvec:simple}]
	We separate and bound
	\begin{multline}
		\bigg\| \hpsi_j - \psi_j -\sum_{m:m \not= j} \frac{\langle (\hcH-\cH) \psi_j, \psi_m \rangle_\mbH}{\lambda_j - \lambda_m} \psi_m \bigg\|_\mbH
		\leq E_1 + E_2 
		\\:= \| \hpsi_j - \langle \psi_j, \hpsi_j \rangle_\mbH \hpsi_j \|_\mbH 
		+ \bigg\| \langle \psi_j, \hpsi_j \rangle_\mbH \hpsi_j - \psi_j -\sum_{m:m \not= j} \frac{\langle (\hcH-\cH) \psi_j, \psi_m \rangle_\mbH}{\lambda_j - \lambda_m} \psi_m \bigg\|_\mbH.
		\label{spec:eigvec:simple-eq1}
	\end{multline}
	Using Lemma~\ref{spec:eigvec:aux1} and $\langle \hpsi_j, \psi_j \rangle_\mbH \ge 0$, we have
	\begin{align}
		E_1 
		= |\langle \hpsi_j, \psi_j \rangle_\mbH - 1 | 
		= \frac{|\langle \hpsi_j, \psi_j \rangle_\mbH^2 - 1 |}{\langle \hpsi_j, \psi_j \rangle_\mbH + 1}
		\leq 4 \halpha_{\!\cJ} \hbeta_{\!\cJ}
		\label{spec:eigvec:simple-eq2}
	\end{align}
	On the other hand, it can be easily shown that	
	\begin{align*}
		E_2
		&= \|(\hcP_{\!\cJ} - \cP_{\!\cJ} - \hcS_{\!\cJ})\psi_j\|_\mbH
		= \| [\hcP_{\!\cJ} - \cP_{\!\cJ} - \hcS_{\!\cJ}] \cP_{\!\cJ} \|_{\HS,\mbH}.
	\end{align*}
	Applying \eqref{spec:eigvec-eq3} with $D=1$, we obtain
	\begin{align}
		E_2 
		\leq \bigg( \frac{56}{9} + 8 \bigg) \halpha_{\!\cJ} \hbeta_{\!\cJ}.
		\label{spec:eigvec:simple-eq3}
	\end{align}
	Therefore, the proof is complete by combining \eqref{spec:eigvec:simple-eq1}-\eqref{spec:eigvec:simple-eq3}.
\end{proof}

\subsection{Proof of Theorems~\ref{spec:eigval-separate}, \ref{spec:eigval-cluster} and \ref{spec:eigval-cluster2}}

We begin by stating some lemmas.

\begin{lemma} \label{spec:eigval:aux1}
	Suppose Assumptions~\ref{spec:assump1} and \ref{spec:assump2} hold. If $\halpha_{\!\cJ} \le 1/4$ and $c \in \R$, we have 
	\begin{multline*}
		\Big\| \hPsi_{\!\cJ} \diag(\hlambda_k-c)_{k \in \cJ} \hPsi_{\!\cJ}^\top - \diag(\lambda_k-c)_{k \in \cJ}
		- \big( \langle (\hcH-\cH) \psi_k, \psi_\ell \rangle_\mbH \big)_{k,\ell \in \cJ} \Big\|_{\F}
		\\ \le
		\min \bigg\{\, \bigg( 4 \gammaJ + \frac{56\sqrt{D}}{9} \max_{d \in [D]}|\theta_d-c| \bigg) \halpha_{\!\cJ}\hbeta_{\!\cJ}
		,\,\,
		\Big( 3 \gammaJ + 4 \max_{d \in [D]}|\theta_d-c| \Big) \hbeta_{\!\cJ}^2 \bigg\},
	\end{multline*}
	where $\hPsi_{\!\cJ} = \big( \langle \psi_k, \hpsi_\ell \rangle_\mbH \big)_{k,\ell \in \cJ}$, which is defined in Lemma~\ref{spec:eigvec:aux1}.
\end{lemma}

\begin{proof}[Proof of Lemma~\ref{spec:eigval:aux1}]
	For $f(z) = z-c$, we have
	\begin{equation}
		f_\cJ (\hcH) = \sum_{k \in \cJ} (\hlambda_k-c) \hcQ_k, \quad
		f_\cJ (\cH) = \sum_{k \in \cJ} (\lambda_k-c) \cQ_k,
		\label{spec:eigval-step1-eq1}
	\end{equation}\\[-3.5em]
	\begin{align*}
		\nabla_{\cH,1} f_\cJ (\hcH)
		&= \sum_{d=1}^D \cP_{\!\cJ_d} (\hcH-\cH) \cP_{\!\cJ_d},	\quad
		\nabla_{\cH,2} f_\cJ (\hcH)
		= \sum_{\substack{1\leq d_1,d_2 \leq D \\ d_1 \not= d_2}} \cP_{\!\cJ_{d_1}} (\hcH - \cH) \cP_{\!\cJ_{d_2}}, \\[-.5em]
		\nabla_{\cH,3} f_\cJ (\hcH)
		&= \sum_{k \in \cJ} (\lambda_k-c) \sum_{\ell \notin \cJ} \frac{\cQ_k (\hcH-\cH) \cQ_\ell+ \cQ_\ell (\hcH-\cH) \cQ_k}{\lambda_k - \lambda_\ell}.
	\end{align*}
	We see that
	\begin{align*}
		\sum_{k=1}^2 \nabla_{\cH,k} f_\cJ (\hcH)
		= \bigg( \sum_{d=1}^D \cP_{\!\cJ_d} \bigg) (\hcH - \cH) \bigg( \sum_{d=1}^D \cP_{\!\cJ_d} \bigg)
		= \cP_{\!\cJ} (\hcH - \cH) \cP_{\!\cJ},
	\end{align*}
	which yields
	\begin{equation}
		\nabla_\cH f_\cJ (\hcH) = \cP_{\!\!\cJ} (\hcH - \cH) \cP_{\!\!\cJ} + \sum_{k \in \cJ} (\lambda_k-c) \sum_{\ell \notin \cJ} \frac{\cQ_k (\hcH-\cH) \cQ_\ell+ \cQ_\ell (\hcH-\cH) \cQ_k}{\lambda_k - \lambda_\ell}
		\label{spec:eigval-step1-eq2}
	\end{equation}
	
	Noting $\langle \hcQ_m \psi_k, \psi_\ell \rangle_\mbH =\langle \psi_k, \hpsi_m \rangle_\mbH \langle \psi_\ell, \hpsi_m \rangle_\mbH $, $\langle \cQ_m \psi_k, \psi_\ell \rangle_\mbH = \delta_{k\ell}$ , $\cP_{\!\cJ} \psi_k= \psi_k$ and $\cP_{\!\cJ} \psi_\ell= \psi_\ell$ for $k, \ell,m \in \cJ$, we obtain for $k,\ell \in \cJ$,
	\begin{multline*}
		\langle (f_\cJ (\hcH) - f_\cJ (\cH) - \nabla_\cH f_\cJ (\hcH)) \psi_k, \psi_\ell \rangle_\mbH
		\\= \sum_{m \in \cJ} (\hlambda_m -c) \langle \psi_k, \hpsi_m \rangle_\mbH \langle \psi_\ell, \hpsi_m \rangle_\mbH - (\lambda_k-c)\delta_{k\ell} - \langle (\hcH-\cH) \psi_k, \psi_\ell \rangle_\mbH.
	\end{multline*}
	This gives
	\begin{multline*}
		\hPsi_{\!\cJ} \diag(\hlambda_k-c)_{k \in \cJ} \hPsi_{\!\cJ}^\top - \diag(\lambda_k-c)_{k \in \cJ}
		- \big( \langle (\hcH-\cH) \psi_k, \psi_\ell \rangle_\mbH \big)_{k,\ell \in \cJ}
		\\= \Big( \langle (f_\cJ (\hcH) - f_\cJ (\cH) - \nabla_\cH f_\cJ (\hcH)) \psi_k, \psi_\ell \rangle_\mbH \Big)_{k,\ell \in \cJ}.
	\end{multline*}
	Hence, the claim follows by Proposition~\ref{spec:comp3}.
\end{proof}

Lemma~\ref{spec:eigval:aux1} provides an indirect comparison between $\{\hlambda_k\}_{k \in \cJ}$ and $\{\lambda_k\}_{k \in \cJ}$ due to the presence of $\hPsi_{\!\cJ}$.
For the direct comparison, note that $\hPsi_{\!\cJ}$ consists of the inner products between the orthonormal eigenvectors $\{ \hpsi_k\}_{k \in \cJ}$ and $\{ \psi_k \}_{k\in\cJ}$ associated with the eigenprojections $\hcP_{\!\cJ}$ and $\cP_{\!\cJ}$, respectively.
If $\hcP_{\!\cJ}=\cP_{\!\cJ}$, then $\hPsi_{\!\cJ}$ is exactly a change-of-basis matrix between two orthonormal bases on the same space, and hence it is orthogonal.
This suggests that $\hPsi_{\!\cJ}$ behaves approximately like an orthogonal matrix as $\hcP_{\!\cJ}$ is close to $\cP_{\!\cJ}$.
In particular, the eigenvalues of a matrix $A$ would be close to those of the matrix $\hPsi_{\!\cJ} A \hPsi_{\!\cJ}^\top$.

The next lemma formalizes this intuition.
The proof follows directly from Lemma~\ref{spec:eigvec:aux1} and Lemma~\refsupp{sup:asymp.ortho} in \seesupp.

\begin{lemma} \label{spec:eigval:aux2}
	Suppose Assumptions~\ref{spec:assump1} and \ref{spec:assump2} hold. 
	If $\halpha_{\!\cJ} \le 1/4$ and $A \in \R^{|\cJ| \times |\cJ|}$ is symmetric, we have
	\begin{gather*}
		\| \specd (\hPsi_{\!\cJ} A \hPsi^\top_\cJ) - \specd (A) \|_2
		\leq \sqrt{2} \|A\|_\op \cdot \min \bigg\{ \frac{56\sqrt{D}}{9} \halpha_{\!\cJ}\hbeta_{\!\cJ} ,\,\, 4 \hbeta_{\!\cJ}^2\bigg\}.
	\end{gather*}
\end{lemma}

The following lemma provides a Weyl-type inequality for eigenvalues.
For the proof, see Lemma~\refsupp{spec:comp1:aux2} in \seesupp.

\begin{lemma} \label{spec:eigval:aux3}
	Suppose Assumptions~\ref{spec:assump1} and \ref{spec:assump2} hold. If $\halpha_{\!\cJ} < 1/2$, we have
	\begin{equation}
		\max_{k \in \cJ} |\hlambda_k - \lambda_k| \leq \gammaJ \halpha_{\!\cJ}.
	\end{equation}
\end{lemma}

Using the above lemmas and the Hoffman-Wielandt inequality, we prove the theorems as follows.

\begin{proof}[Proof of Theorem \ref{spec:eigval-separate}]
	We consider the case $D=1$, i.e., $\theta = \lambda_k$ for all $k \in \cJ$.
	For $D>1$, the result follows by concatenating the corresponding vectors for each $\cJ_d$, $d \in [D]$.
	
	Note that $(\hlambda_k - \lambda_k)_{k \in \cJ}$ is arranged in non-increasing order since $\lambda_k = \theta$ is fixed. This implies
	$(\hlambda_k - \lambda_k)_{k \in \cJ}
	= \specd \left( \diag(\hlambda_k - \lambda_k)_{k \in \cJ} \right) $ and we have
	\begin{equation}
		\Big\| (\hlambda_k - \lambda_k)_{k \in \cJ} - \specd \Big( \langle (\hcH-\cH) \psi_k, \psi_\ell \rangle_\mbH \Big)_{k,\ell \in \cJ} \Big\|_2
		\leq E_1 + E_2,
		\label{spec:eigval-separate-eq1}
	\end{equation}
	where
	\begin{align*}
		E_1 &= \Big\| \specd \Big( \hPsi_{\!\cJ} \diag(\hlambda_k - \lambda_k)_{k \in \cJ} \hPsi_{\!\cJ}^\top \Big) - \specd \Big( \langle (\hcH-\cH) \psi_k, \psi_\ell \rangle_\mbH \Big)_{k,\ell \in \cJ} \Big\|_2,
		\\
		E_2 &= \Big\| \specd \Big( \hPsi_{\!\cJ} \diag(\hlambda_k - \lambda_k)_{k \in \cJ} \hPsi_{\!\cJ}^\top \Big) - \specd \Big( \diag(\hlambda_k - \lambda_k)_{k \in \cJ} \Big) \Big\|_2.
	\end{align*}
	To bound $E_1$, applying the Hoffman-Wielandt inequality yields
	\begin{align*}
		E_1 &\leq \Big\| \hPsi_{\!\cJ} \diag(\hlambda_k - \lambda_k)_{k \in \cJ} \hPsi_{\!\cJ}^\top - \Big( \langle (\hcH-\cH) \psi_k, \psi_\ell \rangle_\mbH \Big)_{k,\ell \in \cJ} \Big\|_\F.
	\end{align*}
	Noting $\max_{d \in [D]} |\theta_d -c| = |\theta-c|$, plugging in $c = \theta$ in Lemma~\ref{spec:eigval:aux1} gives
	\begin{equation}
		E_1 
		\leq \min \big\{ 4 \gammaJ \halpha_{\!\cJ} \hbeta_{\!\cJ},\, 3\gammaJ \hbeta_{\!\cJ}^2 \big\}.
		\label{spec:eigval-separate-eq2}
	\end{equation}	
	Next, by applying Lemma~\ref{spec:eigval:aux2}, we bound $E_2$ as
	\begin{align*}
		E_2	&\leq \sqrt{2} \| \diag(\hlambda_k - \lambda_k)_{k \in \cJ} \|_\op 
		\cdot \min \bigg\{ \frac{56}{9} \halpha_{\!\cJ} \hbeta_{\!\cJ} ,\, 4 \hbeta_{\!\cJ}^2 \bigg\}
		= \sqrt{2} \max_{k\in\cJ} |\hlambda_k-\lambda_k|
		\cdot \min \bigg\{ \frac{56}{9} \halpha_{\!\cJ} \hbeta_{\!\cJ} ,\, 4 \hbeta_{\!\cJ}^2 \bigg\}.
	\end{align*}
	Using Lemma~\ref{spec:eigval:aux3} and the condition $\halpha_{\!\cJ} \le 1/4$, it follows that
	\begin{equation}
		E_2 \leq \sqrt{2} \gammaJ \halpha_{\!\cJ} \cdot \min \bigg\{ \frac{56}{9} \halpha_{\!\cJ} \hbeta_{\!\cJ} ,\, 4 \hbeta_{\!\cJ}^2 \bigg\}
		\leq \min \bigg\{ \frac{14\sqrt{2}}{9} \gammaJ \halpha_{\!\cJ} \hbeta_{\!\cJ} ,\, \sqrt{2}\gammaJ \hbeta_{\!\cJ}^2 \bigg\}.
		\label{spec:eigval-separate-eq3}
	\end{equation}
	Adding the bounds for $E_1$ and $E_2$, we obtain
	\begin{align}
		E_1 + E_2 
		&\leq \min \big\{ 4 \gammaJ \halpha_{\!\cJ} \hbeta_{\!\cJ},\, 3\gammaJ \hbeta_{\!\cJ}^2 \big\} 
		+ \min \bigg\{ \frac{14\sqrt{2}}{9} \gammaJ \halpha_{\!\cJ} \hbeta_{\!\cJ} ,\, \sqrt{2} \gammaJ \hbeta_{\!\cJ}^2 \bigg\}
		\notag\\&\leq \min \bigg\{ \frac{36+14\sqrt{2}}{9}\gammaJ\halpha_{\!\cJ}\hbeta_{\!\cJ} ,\,\, (3+\sqrt{2}) \gammaJ \hbeta_{\!\cJ}^2\bigg\}.
		\label{spec:eigval-separate-eq4}
	\end{align}
	Therefore, the proof is completed by combining \eqref{spec:eigval-separate-eq1} and \eqref{spec:eigval-separate-eq4}.
\end{proof}

\begin{proof}[Proof of Theorem~\ref{spec:eigval-cluster}]
	Set $c = (\theta_\tmax+\theta_\tmin)/2$ and $f(z) = z-c$.
	By \eqref{spec:eigval-step1-eq1} and \eqref{spec:eigval-step1-eq2}, we have
	\begin{align*}
		\tr\Big( f_\cJ (\hcH) - f_\cJ (\cH) - \nabla_\cH f_\cJ (\hcH) \Big) 
		&= \sum_{k \in \cJ} (\hlambda_k -c) - \sum_{k \in \cJ} (\lambda_k - c) - \sum_{k \in \cJ} \langle (\hcH-\cH) \psi_k, \psi_k \rangle_\mbH
		\\&= \sum_{k \in \cJ} \hlambda_k - \sum_{k \in \cJ} \lambda_k - \sum_{k \in \cJ} \langle (\hcH-\cH) \psi_k, \psi_k \rangle_\mbH.
	\end{align*}
	Hence, applying Proposition~\ref{spec:comp3} yields
	\begin{equation*}
		\bigg| \sum_{k \in \cJ} \hlambda_k - \sum_{k \in \cJ} \lambda_k - \sum_{k \in \cJ} \langle (\hcH-\cH) \psi_k, \psi_k \rangle_\mbH \bigg|
		\leq \bigg( \frac{5}{2} \gammaJ + 8\max_{d \in [D]} |\theta_d - c| \bigg)  \hbeta_{\!\cJ}^2.
	\end{equation*}
	This completes the proof by using $\max_{d \in [D]} |\theta_d -c| = (\theta_\tmax-\theta_\tmin)/2$.
\end{proof}

\begin{proof}[Proof of Theorem~\ref{spec:eigval-cluster2}]	
	Noting $(\hlambda_k - c)_{k \in \cJ} = \specd \left( \diag(\hlambda_k - c)_{k \in \cJ} \right)$,
	we have
	\begin{equation}
		\Big\| (\hlambda_k - c)_{k \in \cJ} - \specd \left( \langle (\hcH-\cH) \psi_k, \psi_\ell \rangle_\mbH \right)_{k,\ell \in \cJ} \Big\|_2
		\leq E_1 + E_2,
	\end{equation}
	where
	\begin{align*}
		E_1 &= \Big\| \specd \Big( \hPsi_{\!\cJ} \diag(\hlambda_k-c)_{k \in \cJ} \hPsi_{\!\cJ}^\top \Big) 
		- \specd \Big( \langle (\hcH-\cH) \psi_k, \psi_\ell \rangle_\mbH \Big)_{k,\ell \in \cJ} \Big\|_2,
		\\
		E_2 &= \Big\| \specd \Big( \hPsi_{\!\cJ} \diag(\hlambda_k - c)_{k \in \cJ} \hPsi_{\!\cJ}^\top \Big) - \specd \Big( \diag(\hlambda_k - c)_{k \in \cJ} \Big) \Big\|_2.
	\end{align*}
	By the Hoffman-Wielandt inequality and Lemma~\ref{spec:eigval:aux1}, we bound $E_1$ as
	\begin{align}
		E_1 &\leq \Big\| \hPsi_{\!\cJ} \diag(\hlambda_k - c)_{k \in \cJ} \hPsi_{\!\cJ}^\top - \Big( \langle (\hcH-\cH) \psi_k, \psi_\ell \rangle_\mbH \Big)_{k,\ell \in \cJ} \Big\|_\F.
		\notag\\&\leq \Big\| \hPsi_{\!\cJ} \diag(\hlambda_k - c)_{k \in \cJ} \hPsi_{\!\cJ}^\top 
		- \diag(\lambda_k - c)_{k \in \cJ} 
		- \Big( \langle (\hcH-\cH) \psi_k, \psi_\ell \rangle_\mbH \Big)_{k,\ell \in \cJ} \Big\|_\F
		\notag\\&\quad + \Big\| \diag(\lambda_k - c)_{k \in \cJ} \Big\|_\F
		\notag\\&\leq \min \bigg\{\, \bigg( 4 \gammaJ + \frac{56\sqrt{D}}{9} \max_{d \in [D]}|\theta_d-c| \bigg) \halpha_{\!\cJ}\hbeta_{\!\cJ} ,\,\, \Big( 3 \gammaJ + 4 \max_{d \in [D]}|\theta_d-c| \Big) \hbeta_{\!\cJ}^2 \bigg\}
		+ \|(\lambda_k - c)_{k \in \cJ}\|_2.
	\end{align}
	Using $\max_{k \in \cJ}|\hlambda_k -c| \leq \max_{k \in \cJ}|\hlambda_k -\lambda_k| + \max_{k \in \cJ} |\lambda_k-c|$,
	and applying Lemmas~\ref{spec:eigval:aux2} and \ref{spec:eigval:aux3}, we bound $E_2$ as
	\begin{align}
		E_2 &\leq \sqrt{2} \|\diag(\hlambda_k-c)_{k\in\cJ} \|_\op \cdot \min \bigg\{ \frac{56\sqrt{D}}{9} \halpha_{\!\cJ}\hbeta_{\!\cJ} ,\,\, 4 \hbeta_{\!\cJ}^2\bigg\}.
		\notag\\&\leq \Big(\gammaJ\halpha_{\!\cJ} + \max_{d \in [D]} |\theta_d - c|\Big) 
		\min \bigg\{ \frac{56\sqrt{2D}}{9} \halpha_{\!\cJ}\hbeta_{\!\cJ} ,\,\, 4\sqrt{2} \hbeta_{\!\cJ}^2\bigg\}.
		\label{spec:eigval-cluster-eq3}
	\end{align}
	Adding the bounds for $E_1$ and $E_2$ yields
	\begin{multline}
		E_1 + E_2 \leq 
		\min \bigg\{ \bigg( \Big(\frac{56\sqrt{2D}}{9} \halpha_\cJ + 4 \Big)\gammaJ + \frac{56(\sqrt{2}+1)\sqrt{D}}{9} \max_{d \in [D]} |\theta_d -c| \bigg) \halpha_{\!\cJ}\hbeta_{\!\cJ} , \\
		\bigg( (4 \sqrt{2}\halpha_\cJ +3) \gammaJ + (4\sqrt{2}+4) \max_{d \in [D]} |\theta_d - c| \bigg) \hbeta_{\!\cJ}^2 \bigg\}
		+ \|(\lambda_k - c)_{k \in \cJ}\|_2.
	\end{multline}
	Therefore, the proof is complete by using $\halpha_{\!\cJ} \leq 1/4$.
\end{proof}


\printbibliography[heading=bibintoc]

\appendix

\section{Proofs of Section~\refsupp{mercer}} \label{mercer:proof}

\subsection{Proof of Lemma~\refsupp{mercer:duality}}

\begin{proof}[Proof of Lemma \refsupp{mercer:duality}]
	To show (a), by the definition of the covariance operators and the reproducing property, we have
	\begin{align*}
		\langle \cH f, g \rangle_\mbH
		&= \langle \E h (X, \cdot) \otimes_\mbH h (X, \cdot) f, g \rangle_\mbH
		= \E \langle h (X, \cdot), f \rangle_\mbH \langle h (X, \cdot),  g \rangle_\mbH
		\\&= \E f(X) g(X)
		= \langle f, g \rangle_\LP.
	\end{align*}
	Similarly, we also have
	\begin{align*}
		\langle \hcH_n f, g \rangle_\mbH
		&= \frac{1}{n} \sum_{i=1}^n \langle h (X_i, \cdot) \otimes_\mbH h (X_i, \cdot) f, g \rangle_\mbH
		= \frac{1}{n} \sum_{i=1}^n \langle h (X_i, \cdot), f \rangle_\mbH \langle h(X_i, \cdot), g \rangle_\mbH
		\\&= \frac{1}{n} \sum_{i=1}^n f(X_i) g(X_i)
		= \langle \sn{f}, \sn{g} \rangle_\LPn.
	\end{align*}
	
	For (b), note that for $\psi_k = \sqrt{\lambda_k} \phi_k$ with $\lambda_k \neq 0$, we have
	$$ \cH \psi_k = \bfH \psi_k = \sqrt{\lambda_k} \, \bfH \phi_k = \lambda_k \sqrt{\lambda_k} \phi_k = \lambda_k \psi_k. $$
	Thus, $\lambda_k$ is an eigenvalue of $\cH$ corresponding to $\psi_k$.
	Moreover, for $\psi_k$ and $\psi_\ell$ with $\lambda_k, \lambda_\ell \neq 0$, using (a), we have
	\begin{align*}
		\langle \psi_k, \psi_\ell \rangle_\mbH
		= \sqrt{\lambda_k \lambda_\ell} \langle \phi_k, \phi_\ell \rangle_\mbH
		= \frac{\sqrt{\lambda_k \lambda_\ell}}{\lambda_k} \langle \cH \phi_k, \phi_\ell \rangle_\mbH
		= \frac{\sqrt{\lambda_k \lambda_\ell}}{\lambda_k} \langle \phi_k, \phi_\ell \rangle_\LP
		= \delta_{k \ell}.
	\end{align*}
	Hence, the family $\{\psi_k:1\leq k\leq\rank(\bfH)\}$ is orthonormal in $\mbH$.
	
	Lastly, we prove (c). Choose an eigenpair $(\hlambda_{k,n}, \hvarphi_{k,n})$ of $\hbfH_n$ with $\hlambda_{k,n} \neq 0$.
	Note that the $i'$th component of $\hbfH_n \hvarphi_{k,n}$ is 
	$ \frac{1}{n} \sum_{i=1}^n h(X_{i'},X_i) \hvarphi_{k,n}(i),$
	and since $\hbfH_n \hvarphi_{k,n}= \hlambda_{k,n} \hvarphi_{k,n},$ we obtain
	\begin{align*}
		\sn{\hpsi_{k,n}} 
		= \frac{1}{\sqrt{\hlambda_{k,n}}} \left( \frac{1}{n} \sum_{i=1}^n h(X_{i'},X_i) \hvarphi_{k,n}(i) \right)_{i' \in \{1,\dots,n\}}
		= \frac{1}{\sqrt{\hlambda_{k,n}}} \hlambda_{k,n} \hvarphi_{k,n}
		= \sqrt{\hlambda_{k,n}} \hvarphi_{k,n}.
	\end{align*}
	Therefore,
	$\hpsi_{k,n} (X_i) = \sqrt{\hlambda_{k,n}} \hvarphi_{k,n} (i),$ $i \in [n]$,
	and this yields
	\begin{align*}
		\hcH_n \hpsi_{k,n} 
		= \frac{1}{n} \sum_{i=1}^n h(X_i,\cdot) \hpsi_{k,n} (X_i)
		= \hlambda_{k,n} \cdot \frac{1}{\sqrt{\hlambda_{k,n}}} \cdot \frac{1}{n} \sum_{i=1}^n h(X_i,\cdot) \hvarphi_{k,n} (i)
		= \hlambda_{k,n} \hpsi_{k,n}.
	\end{align*}
	Hence, $\hlambda_{k,n}$ is the eigenvalue of $\hcH_n$ corresponding to $\hpsi_{k,n}$.
	
	Additionally, for $\hpsi_{k,n}$ and $\hpsi_{\ell,n}$ with $\hlambda_{k,n}, \hlambda_{\ell,n} \neq 0$,
	part (a) gives
	\begin{align*}
		\langle \hpsi_{k,n}, \hpsi_{\ell,n} \rangle_\mbH
		= \frac{1}{\hlambda_{k,n}} \langle \hcH_n \hpsi_{k,n}, \hpsi_{\ell,n} \rangle_\mbH
		= \frac{\sqrt{\hlambda_{k,n} \hlambda_{\ell,n}}}{\hlambda_{k,n}} \langle \hvarphi_{k,n}, \hvarphi_{\ell,n} \rangle_\LPn
		= \delta_{k \ell}.
	\end{align*}
	
	Hence, the family $\{\hpsi_{k,n}:1\leq k \leq \rank(\hbfH_n) \}$ is orthonormal in $\mbH$.
\end{proof}

\subsection{Proof of Proposition~\refsupp{mercer:perturb}}

\begin{lemma} \label{mercer:perturb:aux1}
	Under Assumptions~\refsupp{mercer:assump1} and \refsupp{mercer:assump2}, we have
	\begin{equation}
		\tr (\cT_{\!\cJ}^{1/2} \cH \cT_{\!\cJ}^{1/2}) = \bfsigma_{\!\cJ},
		\quad 
		\| \cT_{\!\cJ}^{1/2} \cH \cT_{\!\cJ}^{1/2} \|_{\op,\mbH} \leq \frac{2\theta_\tmax}{\gammaJ}.
	\end{equation}
\end{lemma}

\begin{proof}[Proof of Lemma~\ref{mercer:perturb:aux1}]
	By the definition of $\cT_{\!\cJ}^{1/2}$, it is clear that
	\begin{align*}
		\cT_{\!\cJ}^{1/2} \cH \cT_{\!\cJ}^{1/2} = \frac{1}{\gammaJ}\sum_{k \in \cJ} \lambda_k \cQ_k + \sum_{k \notin \cJ} \frac{\lambda_k \cQ_k}{\min_{m \in \cJ} |\lambda_k-\lambda_m|}.
	\end{align*}
	This gives $\tr (\cT_{\!\cJ}^{1/2} \cH \cT_{\!\cJ}^{1/2}) = \bfsigma_{\!\cJ}$ and
	\begin{align*}
		\| \cT_{\!\cJ}^{1/2} \cH \cT_{\!\cJ}^{1/2} \|_{\op,\mbH} = \max \bigg\{ \frac{\theta_\tmax}{\gammaJ}, \sup_{k \notin \cJ} \frac{\lambda_k}{\min_{m \in \cJ} |\lambda_k-\lambda_m|}  \bigg\}.
	\end{align*}
	For any $k \notin \cJ$, we bound
	\begin{align*}
		\frac{\lambda_k}{\min_{m \in \cJ} |\lambda_k-\lambda_m|}
		= \max_{m \in \cJ} \bigg( \frac{\lambda_m}{|\lambda_m-\lambda_k|} + \frac{\lambda_k-\lambda_m}{|\lambda_k-\lambda_m|} \bigg)
		\leq \frac{\theta_\tmax}{\gammaJ} + 1.
	\end{align*}
	Applying the assumption $\gammaJ \le \theta_\tmax$, the claim follows.
\end{proof}

\begin{proof}[Proof of \eqrefsupp{mercer:perturb-res1} and \eqrefsupp{mercer:perturb-res3} in Proposition~\refsupp{mercer:perturb}]
	We apply the Bernstein-type bound stated in Lemma~\ref{sup:Bernstein1}, which is a modification of Lemma~5 of \citesupp{dicker2017Kernel}.
	To this end, we derive suitable bounds for $\|\cY_i\|_{\op,\mbH}$, $\|\E \cY_i^2\|_{\op,\mbH}$ and $\tr(\E \cY_i^2)$,
	where $\cY_i$ is defined by
	\begin{align*}
		\cY_i &= \cT_{\cJ}^{1/2} [h(X_i, \cdot) \otimes_\mbH h(X_i, \cdot) - \E h(X, \cdot) \otimes_\mbH h(X, \cdot)] \cT_{\cJ}^{1/2},
	\end{align*}
	so that $\cT_{\cJ}^{1/2} (\hcH_n - \cH) \cT_{\cJ}^{1/2} = \frac{1}{n} \sum_{i=1}^n \cY_i$.
	Our argument closely follows that of Lemma~4 in \citesupp{wahl2024kernelbased}, which establishes concentration inequalities under similar situation.
	
	To further simplify notation, we introduce the function $f_i$ defined by
	\begin{align*}
		f_i = \cT_{\cJ}^{1/2} h(X_i, \cdot)
		&= \frac{1}{\sqrt{\gammaJ}} \sum_{k \in \cJ} \psi_k(X_i) \psi_k
		+ \sum_{k \notin \cJ} \frac{\psi_k(X_i)}{\sqrt{\min_{m \in \cJ} |\lambda_k-\lambda_m|}} \psi_k
		\\&= \frac{1}{\sqrt{\gammaJ}} \sum_{k \in \cJ} \sqrt{\lambda_k} \phi_k(X_i) \psi_k
		+ \sum_{k \notin \cJ} \frac{\sqrt{\lambda_k}\phi_k(X_i)}{\sqrt{\min_{m \in \cJ} |\lambda_k-\lambda_m|}} \psi_k.
	\end{align*}
	Then, the following can be shown easily:
	\begin{align*}
		\cY_i &= f_i \otimes_\mbH f_i - \E f_i \otimes_\mbH f_i, \\
		\E f_i \otimes_\mbH f_i &= \cT_{\cJ}^{1/2} \cH \cT_{\cJ}^{1/2}
		= \frac{1}{\gammaJ} \sum_{k \in \cJ} \lambda_k \cQ_k
		+ \sum_{k \notin \cJ} \frac{\lambda_k \cQ_k}{\min_{\ell \in \cJ} |\lambda_k-\lambda_\ell|} 
		\\
		\|f_i\|^2_\mbH &= \frac{1}{\gammaJ} \sum_{k \in \cJ} \lambda_k \phi_k^2(X_i) 
		+ \sum_{k \notin \cJ} \frac{\lambda_k\phi_k^2(X_i)}{\min_{m \in \cJ} |\lambda_k-\lambda_m|}
		\leq \bfkappa_{\!\cJ}, 
		\\
		(f_i \otimes f_i)^2 &= \|f_i\|_\mbH^2 (f_i \otimes f_i).
	\end{align*}
	
	First, we see that	
	\begin{equation} 
		\|\cY_i\|_{\op, \mbH}
		\leq \| f_i \otimes_\mbH f_i \|_{\op, \mbH} + \E \| f_i \otimes_\mbH f_i \|_{\op, \mbH}
		\leq 2 \bfkappa_{\!\cJ}.
		\label{mercer:perturb-eq1}
	\end{equation}
	Next, observe that
	\begin{equation*}
		0 \preceq\E \cY_i^2
		= \E \left[ f_i \otimes_\mbH f_i \right]^2 - \left[ \E f_i \otimes_\mbH f_i \right]^2
		\preceq	\E \|f_i\|^2_\mbH (f_i \otimes_\mbH f_i)
		\preceq \bfkappa_{\!\cJ} \cT_{\!\cJ}^{1/2} \cH \cT_{\!\cJ}^{1/2},
	\end{equation*}
	where $A \preceq B$ means $B-A$ is positive semi-definite.
	Using Lemma~\ref{mercer:perturb:aux1}, we obtain
	\begin{equation}
		\| \E \cY_i^2 \|_{\op,\mbH} 
		\leq \bfkappa_{\!\cJ} \|\cT_{\!\cJ}^{1/2} \cH \cT_{\!\cJ}^{1/2} \|_{\op,\mbH}
		\leq \frac{2\theta_\tmax \bfkappa_{\!\cJ}}{\gammaJ},
		\label{mercer:perturb-eq2}
	\end{equation}
	and
	\begin{equation}
		\tr( \E \cY_i^2 ) 
		\leq \bfkappa_{\!\cJ} \tr \big( \cT_{\!\cJ}^{1/2} \cH \cT_{\!\cJ}^{1/2} \big)
		\leq \bfkappa_{\!\cJ} \bfsigma_{\!\!\cJ} 
		= \frac{\gammaJ \bfsigma_{\!\!\cJ}}{2\theta_\tmax} \cdot \frac{2\theta_\tmax \bfkappa_{\!\cJ}}{\gammaJ}.
		\label{mercer:perturb-eq3}
	\end{equation}
	From \eqref{mercer:perturb-eq1}-\eqref{mercer:perturb-eq3}, 
	Lemma~\ref{sup:Bernstein1} can be applied with $R= 2\bfkappa_{\!\cJ}$, $V=\frac{2\theta_\tmax \bfkappa_{\!\cJ}}{\gammaJ}$ and $D=\frac{\gammaJ \bfsigma_{\!\!\cJ}}{2\theta_\tmax}$, yielding the following results.
	
	\mycase{Deriving \eqrefsupp{mercer:perturb-res1}}
	It can be easily shown that $\frac{\gammaJ \bfsigma_{\!\!\cJ}}{2\theta_\tmax} \ge \frac{1}{2}$.
	Hence, for $0<\tau<1$, we have with probability at least $1-\tau$,
	\begin{align*}
		\halpha_{\!\cJ} &\leq \sqrt{\frac{2 \cdot \frac{2\theta_\tmax \bfkappa_{\!\cJ}}{\gammaJ} \cdot \log (4\cdot \frac{\gammaJ\bfsigma_{\!\!\cJ}}{2\theta_\tmax}/\tau) }{n}} + \frac{2\bfkappa_{\!\cJ} \cdot \log ( 4\cdot \frac{\gammaJ\bfsigma_{\!\!\cJ}}{2\theta_\tmax}/\tau) }{3n}
		\\&= \sqrt{\frac{4\theta_\tmax \bfkappa_{\!\cJ} \log ( \frac{2\gammaJ\bfsigma_{\!\!\cJ}}{\theta_\tmax}/\tau) }{\gammaJ n}} + \frac{2\bfkappa_{\!\cJ} \log ( \frac{2\gammaJ\bfsigma_{\!\!\cJ}}{\theta_\tmax}/\tau) }{3n}
		\\&\leq \sqrt{\frac{4\theta_\tmax \bfkappa_{\!\cJ} \log ( 2\bfsigma_{\!\!\cJ} /\tau ) }{\gammaJ n}} + \frac{2\bfkappa_{\!\cJ} \log ( 2\bfsigma_{\!\!\cJ} /\tau ) }{3n}.
	\end{align*}
	The last inequality is obtained from the assumption $\gammaJ \le \theta_\tmax$.
	
	\mycase{Deriving \eqrefsupp{mercer:perturb-res3}}
	We have
	\begin{equation}
		\P( \halpha_{\!\cJ} > 1/4 ) \leq 4 \cdot \frac{\gammaJ \bfsigma_{\!\!\cJ}}{2\theta_\tmax} \cdot \exp\bigg( - \frac{3n\cdot(1/4)^2}{6 \cdot \frac{2\theta_\tmax \bfkappa_{\!\cJ}}{\gammaJ} + 2 \cdot 2\bfkappa_{\!\cJ} \cdot (1/4)} \bigg)
		\label{mercer:perturb-eq4}
	\end{equation}
	provided that
	\begin{equation}
		\frac{1}{4} \ge \sqrt{\frac{2\theta_\tmax \bfkappa_{\!\cJ}}{\gammaJ n}} + \frac{2\bfkappa_{\!\cJ}}{3n}
		= \sqrt{\frac{2\theta_\tmax \bfkappa_{\!\cJ}}{\gammaJ n}} + \frac{2\theta_\tmax \bfkappa_{\!\cJ}}{\gammaJ n} \cdot \frac{\gammaJ}{3\theta_\tmax}.
		\label{mercer:perturb-eq5}
	\end{equation}
	
	Using the assumption $1 \le \theta_\tmax/\gammaJ$, the right-hand side of \eqref{mercer:perturb-eq5} is bounded by
	\begin{align}
		2 \bfsigma_{\!\!\cJ} \exp\bigg( - \frac{ \frac{3n}{16} }{\frac{12\theta_\tmax \bfkappa_{\!\cJ}}{\gammaJ} + \frac{\theta_\tmax}{\gammaJ} \cdot \bfkappa_{\!\cJ}} \bigg)
		= 2 \bfsigma_{\!\!\cJ} \exp\bigg( -\frac{3n \gammaJ }{208 \theta_\tmax \bfkappa_{\!\cJ}} \bigg)
		\leq 2 \bfsigma_{\!\!\cJ} \exp\bigg( -\frac{n \gammaJ}{70 \theta_\tmax \bfkappa_{\!\cJ}} \bigg).
		\label{mercer:perturb-eq6}
	\end{align}
	On the other hand, since $\gammaJ/\theta_\tmax \le 1$, \eqref{mercer:perturb-eq5} holds if $\frac{2\theta_\tmax\bfkappa_{\!\cJ}}{\gammaJ n} \leq \frac{1}{19}$, equivalently, 
	\begin{equation}
		n \ge \frac{38 \theta_\tmax \bfkappa_{\!\cJ}}{\gammaJ}.
		\label{mercer:perturb-eq7}
	\end{equation}
	
	Therefore, $\eqrefsupp{mercer:perturb-res3}$ is obtained from \eqref{mercer:perturb-eq6} and \eqref{mercer:perturb-eq7}, concluding the proof.
\end{proof}

\begin{proof}[Proof of \eqrefsupp{mercer:perturb-res2} in Proposition~\refsupp{mercer:perturb}]
	Let
	$$\cY_i = \gammaJ^{-1/2} |\cR|_{\!\cJ}^{1/2} [h(X_i,\cdot) \otimes_\mbH h(X_i,\cdot)] \cP_{\!\cJ} - \gammaJ^{-1/2} \E  |\cR|_{\!\cJ}^{1/2} [h(X,\cdot) \otimes_\mbH h(X,\cdot)] \cP_{\!\cJ},$$
	so that $ \sum_{i=1}^n \cY_i = \gammaJ^{-1/2} |\cR|_{\!\cJ}^{1/2} (\hcH_n-\cH) \cP_{\!\cJ}. $
	
	Since $	\gammaJ^{-1/2} \E  |\cR|_{\!\cJ}^{1/2} [h(X,\cdot) \otimes_\mbH h(X,\cdot)] \cP_{\!\cJ}	= \gammaJ^{-1/2} |\cR|_{\!\cJ}^{1/2} \cH \cP_{\!\cJ} = 0$, we have
	\begin{align*}
		\| \cY_i \|_{\HS,\mbH}^2
		&= \frac{1}{\gamma_\cJ} \sum_{k \in \cJ} \sum_{\ell \notin \cJ} \frac{\langle [h(X_i, \cdot) \otimes_\mbH h(X_i,\cdot)] \psi_k, \psi_\ell \rangle_\mbH^2}{\min_{m \in \cJ} |\lambda_\ell - \lambda_m|}
		\\&= \frac{1}{\gamma_\cJ} \sum_{k \in \cJ} \sum_{\ell \notin \cJ} \frac{ \psi_k^2 (X_i) \psi_\ell^2 (X_i) }{\min_{m \in \cJ} |\lambda_\ell - \lambda_m|}
		\\&= \frac{1}{\gamma_\cJ} \sum_{k \in \cJ} \lambda_k \phi_k^2(X_i) \sum_{\ell \notin \cJ} \frac{ \lambda_\ell \phi_\ell^2 (X_i) }{\min_{m \in \cJ} |\lambda_\ell - \lambda_m|}.
	\end{align*}
	Consequently, we obtain $\| \cY_i \|_{\HS,\mbH} \leq \bfvarkappa_{\!\cJ}$ and $ \E \| \cY_i \|^2_{\HS,\mbH} = \bfvarsigma_\cJ$.
	Hence, the claim follows by the Bernstein inequality (see Lemma~\ref{sup:Bernstein2}).
\end{proof}

\subsection{Proof of Lemma~\refsupp{mercer:eigproj-duality}}

\begin{proof}[Proof of Lemma~\refsupp{mercer:eigproj-duality}]
	First, from the definition of $\hcP_{\!\cJ,n}$, we see that
	\begin{align*}
		\langle \hcH_n\hcP_{\!\cJ,n} f, g \rangle_\mbH
		&= \langle \hcH_n\hcP_{\!\cJ,n} f, \hcP_{\!\cJ,n} g \rangle_\mbH
		\\&= \sum_{k \in \cJ} \sum_{\ell \in \cJ} \langle f, \hpsi_{k,n} \rangle_\mbH \langle g, \hpsi_{\ell,n} \rangle_\mbH \langle \hcH_n \hpsi_{k,n}, \hpsi_{\ell,n} \rangle_\mbH
		\\&= \sum_{k \in \cJ} \sum_{\ell \in \cJ} \langle f, \hpsi_{k,n} \rangle_\mbH \langle g, \hpsi_{\ell,n} \rangle_\mbH \hlambda_{k,n} \delta_{k \ell}
		\\&= \sum_{k \in \cJ} \hlambda_{k,n} \langle f, \hpsi_{k,n} \rangle_\mbH \langle g, \hpsi_{k,n} \rangle_\mbH
		\\&= \sum_{k \in \cJ} \frac{\langle f, \hcH_n \hpsi_{k,n} \rangle_\mbH \langle g, \hcH_n \hpsi_{k,n} \rangle_\mbH}{\hlambda_{k,n}}.
	\end{align*}
	Applying Lemma~\refsupp{mercer:duality} yields
	\begin{align*}
		\sum_{k \in \cJ} \frac{\langle f, \hcH_n \hpsi_{k,n} \rangle_\mbH \langle g, \hcH_n \hpsi_{k,n} \rangle_\mbH}{\hlambda_{k,n}}
		&= \sum_{k \in \cJ} \frac{\langle \sn{f}, \sn{\hpsi_{k,n}} \rangle_\LPn \langle \sn{g}, \sn{\hpsi_{k,n}} \rangle_\LPn}{\hlambda_{k,n}}
		\\&= \sum_{k \in \cJ} \langle \sn{f}, \hvarphi_{k,n} \rangle_\LPn \langle \sn{g}, \hvarphi_{k,n} \rangle_\LPn
		\\&= \langle \hbfP_{\!\cJ,n} \sn{f}, \sn{g} \rangle_\LPn.
	\end{align*}
	Thus, we conclude
	$\langle \hcH_n\hcP_{\!\cJ,n} f, g \rangle_\mbH = \langle \hbfP_{\!\cJ,n} \sn{f}, \sn{g} \rangle_\LPn.$
	
	Similarly, it can be shown that
	$\langle \cH\cP_{\!\cJ} f, g \rangle_\mbH = \langle \bfP_{\!\cJ} f, g \rangle_\LP.$
	This completes the proof.
\end{proof}

\subsection{Proof of Theorem~\refsupp{mercer:eigproj-concen} and Corollary~\refsupp{mercer:eigproj-weak}}

We start with stating and proving some auxiliary lemmas.

\begin{lemma} \label{mercer:eigproj-step1}
	Suppose Assumptions~\refsupp{mercer:assump1} and \refsupp{mercer:assump2} hold.
	If $\halpha_{\!\cJ,n} \leq 1/4$, we have
	\begin{equation}
		\|\hcH_n \hcP_{\!\cJ,n} - \cH \cP_{\!\cJ} - \hcA_{\!\cJ,n}\|_{\op,\mbH}
		\leq 12\theta_\tmax \bigg(\min\bigg\{D,\, 1+\frac{2\theta_\tmax-2\theta_\tmin}{\pi\gammaJ} \bigg\} \bigg) \halpha^2_{\!\cJ,n}.
	\end{equation}
	where 
	\begin{equation}
		\hcA_{\!\cJ,n}
		= \cP_{\!\cJ} (\hcH_n - \cH) \cP_{\!\cJ} + \sum_{k \in \cJ} \lambda_k \sum_{\ell \notin \cJ} \frac{\cQ_k (\hcH_n-\cH) \cQ_\ell+ \cQ_\ell (\hcH_n-\cH) \cQ_k}{\lambda_k - \lambda_\ell}.
	\end{equation}
	In particular, $\hcA_{\cJ,n}$ satisfies without the condition $\halpha_{\!\cJ} \leq 1/4$:
	\begin{equation}
		\|\hcA_{\cJ,n}\|_{\op,\mbH} \le 3 \theta_\tmax \bigg(\min\bigg\{D,\, 1+\frac{2\theta_\tmax-2\theta_\tmin}{\pi\gammaJ} \bigg\} \bigg) \halpha_{\!\cJ,n}.
	\end{equation}
\end{lemma}

\begin{proof}[Proof of Lemma~\ref{mercer:eigproj-step1}]
	By letting $f(z) = z$, applying \eqrefsupp{spec:eigval-step1-eq1} and \eqrefsupp{spec:eigval-step1-eq2} in the proof of Lemma~\refsupp{spec:eigval:aux1} and Proposition~\refsupp{spec:comp1} yields
	\begin{gather*}
		\|\hcH_n \hcP_{\!\cJ,n} - \cH \cP_{\!\cJ} - \hcA_{\cJ,n}\|_{\op,\mbH} 
		\leq 8 \Big(\sup_{z \in \partial\Gamma_{\!\cJ}} |z| \Big) \bigg(\min\bigg\{D,\, 1+\frac{2\theta_\tmax-2\theta_\tmin}{\pi\gammaJ} \bigg\} \bigg) \halpha_{\!\cJ,n}^2, \\
		\| \hcA_{\cJ,n} \|_{\op,\mbH}
		\leq 2 \Big(\sup_{z \in \partial\Gamma_{\!\cJ}} |z| \Big) \bigg(\min\bigg\{D,\, 1+\frac{2\theta_\tmax-2\theta_\tmin}{\pi\gammaJ} \bigg\} \bigg) \halpha_{\!\cJ,n},
	\end{gather*}
	where the latter bound holds without the condition $\halpha_{\!\cJ} \leq 1/4$.
	Since $\sup_{z \in \partial\Gamma_{\!\cJ}} |z| = \theta_\tmax + \frac{\gammaJ}{2} \leq \frac{3}{2} \theta_\tmax$ by the assumption $\gammaJ\leq \theta_\tmin$, the claim follows.
\end{proof}

\begin{lemma} \label{mercer:eigproj-step2}
	Under Assumptions~\refsupp{mercer:assump1} and \refsupp{mercer:assump2}, we have
	\begin{align}
		\langle \cP_{\!\cJ} (\hcH_n - \cH) \cP_{\!\cJ} \psi_k, \psi_\ell \rangle_\mbH
		&= \begin{cases}
			\sqrt{\lambda_k \lambda_\ell} (\P_n - \P) (\phi_k \phi_\ell), & \text{if } k,\ell \in \cJ, 
			\\[.5em]
			0, & \text{otherwise},
		\end{cases}
		\label{mercer:eigproj-step2-eq1} \\
		\langle \hcA_{\cJ,n} \psi_k, \psi_\ell \rangle_\mbH
		&= \begin{cases}
			\sqrt{\lambda_k \lambda_\ell} (\P_n - \P) (\phi_k \phi_\ell), & \text{if } k,\ell \in \cJ, 
			\\[.5em]
			\displaystyle\frac{\lambda_k \sqrt{\lambda_k \lambda_\ell}(\P_n - \P) (\phi_k \phi_\ell)}{\lambda_k-\lambda_\ell}, & \text{if } k \in \cJ,\ell \notin \cJ, 
			\\[.5em]
			\displaystyle\frac{\lambda_\ell\sqrt{\lambda_k \lambda_\ell} (\P_n - \P) (\phi_k \phi_\ell)}{\lambda_\ell-\lambda_k}, & \text{if } k \notin \cJ, \ell \in \cJ, 
			\\[.5em]
			0, & \text{otherwise}.
		\end{cases}
	\end{align}
\end{lemma}

\begin{proof}[Proof of Lemma \ref{mercer:eigproj-step2}]
	Since $\cP_{\!\cJ} \psi_k = \psi_k$ if $k \in \cJ$ and $0$ otherwise, we have
	$$ \langle \cP_{\!\cJ} (\hcH_n - \cH) \cP_{\!\cJ} \psi_k, \psi_\ell \rangle_\mbH 
	= \langle (\hcH_n - \cH) \cP_{\!\cJ} \psi_k, \cP_{\!\cJ} \psi_\ell \rangle_\mbH =  0,$$
	unless both $k$ and $\ell$ belong to $\cJ$.
	Now, for $k,\ell \in \cJ$,
	\begin{align*}
		\langle \cP_{\!\cJ} (\hcH_n - \cH) \cP_{\!\cJ} \psi_k, \psi_\ell \rangle_\mbH
		= \langle (\hcH_n - \cH) \psi_k, \psi_\ell \rangle_\mbH
		= \sqrt{\lambda_k \lambda_\ell} \langle (\hcH_n - \cH) \phi_k, \phi_\ell \rangle_\mbH.
	\end{align*}
	By Lemma~\refsupp{mercer:duality}, we have
	\begin{align*}
		\langle (\hcH_n - \cH) \phi_k, \phi_\ell \rangle_\mbH
		&= \langle \hcH_n \phi_k, \phi_\ell \rangle_\mbH - \langle \cH \phi_k, \phi_\ell \rangle_\mbH
		\\&= \langle \sn{\phi_k}, \sn{\phi_\ell} \rangle_\LPn - \langle \phi_k, \phi_\ell \rangle_\LP
		\\&= (\Pn-\P) (\phi_k \phi_\ell).
	\end{align*}
	Hence, we obtain \eqref{mercer:eigproj-step2-eq1}.
	
	Next, it can be easily shown that
	$$ \bigg\langle \bigg(\sum_{k \in \cJ} \lambda_k \sum_{\ell \notin \cJ} \frac{\cQ_k (\hcH_n-\cH) \cQ_\ell+ \cQ_\ell (\hcH_n-\cH) \cQ_k}{\lambda_k - \lambda_\ell}\bigg) \psi_{m_1}, \psi_{m_2} \bigg\rangle_\mbH = 0,$$
	unless either $m_1 \in \cJ$ and $m_2 \notin \cJ$, or $m_1 \notin \cJ$ and $m_2 \in \cJ$.
	Now, for $m_1 \in \cJ$ and $m_2 \notin \cJ$, we have
	\begin{align}
		&\quad \left\langle \left(\sum_{k \in \cJ} \lambda_k \sum_{\ell \notin \cJ} \frac{\cQ_k (\hcH_n-\cH) \cQ_\ell+ \cQ_\ell (\hcH_n-\cH) \cQ_k}{\lambda_k - \lambda_\ell}\right) \psi_{m_1}, \psi_{m_2} \right\rangle_\mbH
		\nonumber\\&= \left\langle \left(\sum_{k \in \cJ} \lambda_k \sum_{\ell \notin \cJ} \frac{\cQ_\ell (\hcH_n-\cH) \cQ_k}{\lambda_k - \lambda_\ell}\right) \psi_{m_1}, \psi_{m_2} \right\rangle_\mbH
		\nonumber\\&= \frac{\lambda_{m_1} \langle (\hcH_n - \cH) \psi_{m_1}, \psi_{m_2} \rangle_\mbH}{\lambda_{m_1} - \lambda_{m_2}}
		\nonumber\\&= \frac{\lambda_{m_1} \sqrt{\lambda_{m_1} \lambda_{m_2}} (\P_n-\P)(\phi_{m_1} \phi_{m_2})}{\lambda_{m_1} - \lambda_{m_2}}.
		\label{mercer:eigproj-step2-eq2}
	\end{align}
	By symmetry, we also obtain the result for $m_1 \notin \cJ$ and $m_2 \in \cJ$.
	This completes the proof.
\end{proof}

Next, we state the non-stochastic asymptotic expansion as follows.
\begin{proposition} \label{mercer:eigproj-step3}
	Suppose Assumptions~\refsupp{mercer:assump1} and \refsupp{mercer:assump2} hold.
	If $\halpha_{\!\cJ,n} \le 1/4$ and $f, g \in \mbH$, we have
	\begin{multline}
		\bigg| \langle \hbfP_{\!\cJ,n} \sn{f}, \sn{g} \rangle_\LPn
		- \langle \bfP_{\!\cJ} f, g \rangle_\LP 
		- \langle \hbfUpsilon_{\!\cJ,n} f, g \rangle_\LP \bigg|
		\\ \leq \|f\|_\mbH \|g\|_\mbH \cdot 12 \theta_\tmax \bigg(\min\bigg\{D,\, 1+\frac{2\theta_\tmax-2\theta_\tmin}{\pi\gammaJ} \bigg\} \bigg) \halpha_{\!\cJ,n}^2
	\end{multline}
	where $ \hbfUpsilon_{\!\cJ,n} : \LP \rightarrow \LP$ is defined in Theorem \refsupp{mercer:eigproj-concen}.
\end{proposition}

\begin{proof}[Proof of Proposition~\ref{mercer:eigproj-step3}]
	By Lemma~\ref{spec:comp1:aux2}, the condition $\halpha_{\!\cJ,n} \le 1/4$ implies $\{\hlambda_{k,n}\}_{k \in \cJ}$ are nonzero.
	Hence, Lemma~\refsupp{mercer:eigproj-duality} can be applied and it gives
	\begin{equation}
		\langle (\hcH_n\hcP_{\!\cJ,n} - \cH\cP_{\!\cJ}) f, g \rangle_\mbH
		= \langle \hbfP_{\!\cJ,n} \sn{f}, \sn{g} \rangle_\LPn - \langle \bfP_{\!\cJ} f, g \rangle_\LP .
		\label{mercer:eigproj-step3-eq2}
	\end{equation}
	For $\langle \hcA_{\cJ,n} f, g \rangle_\mbH$, note that by Lemma~\ref{mercer:eigproj-step2},
	\begin{equation*}
		\langle \hcA_{\cJ,n} \phi_k, \phi_\ell \rangle_\mbH
		= \frac{\langle \hcA_{\cJ,n} \psi_k, \psi_\ell \rangle_\mbH}{\sqrt{\lambda_k} \sqrt{\lambda_\ell}}
		= \langle \hbfUpsilon_{\!\cJ,n} \phi_k, \phi_\ell \rangle_\LP
		\quad\text{for all $k,\ell \in \N$.}
	\end{equation*}
	Write
	$ f = \sum_{k=1}^\infty a_k \phi_k$, and $g = \sum_{\ell=1}^\infty b_\ell \phi_k, $
	where $a_k = \langle f , \phi_k \rangle_\LP$ and $b_\ell = \langle f , \phi_\ell \rangle_\LP.$	
	Then,
	\begin{align}
		\langle \hcA_{\cJ,n} f, g \rangle_\mbH
		&= \sum_{k=1}^\infty \sum_{\ell=1}^\infty a_k b_\ell \langle \hcA_{\cJ,n} \phi_k, \phi_\ell \rangle_\mbH
		\notag\\&= \sum_{k=1}^\infty \sum_{\ell=1}^\infty a_k b_\ell \langle \hbfUpsilon_{\!\cJ,n} \phi_k, \phi_\ell \rangle_\LP
		= \langle \hbfUpsilon_{\!\cJ,n} f, g \rangle_\LP.
		\label{mercer:eigproj-step3-eq3}
	\end{align}
	Combining \eqref{mercer:eigproj-step3-eq2} and \eqref{mercer:eigproj-step3-eq3} yields
	\begin{multline*}
		\bigg| \langle \hbfP_{\!\cJ,n} \sn{f}, \sn{g} \rangle_\LPn - \langle \bfP_{\!\cJ} f, g \rangle_\LP - \langle \hbfUpsilon_{\!\cJ,n} f, g \rangle_\LP \bigg|
		\\ = \Big| \langle (\hcH_n\hcP_{\!\cJ,n} - \cH\cP_{\!\cJ} - \hcA_{\cJ,n})f, g \rangle_\mbH \Big|
		\\ \leq \|f\|_\mbH \|g\|_\mbH \Big\| \hcH_n\hcP_{\!\cJ,n}-\cH\cP_{\!\cJ} - \hcA_{\cJ,n} \Big\|_{\op,\mbH}.
	\end{multline*}
	Therefore, the proof is complete by applying Lemma~\ref{mercer:eigproj-step1}.
\end{proof}

We are now ready to prove Theorem~\refsupp{mercer:eigproj-concen} and Corollary~\refsupp{mercer:eigproj-weak}.

\begin{proof}[Proof of Theorem~\refsupp{mercer:eigproj-concen}]
	Set the events $E_1$ and $E_2$ as
	\begin{align*}
		E_1 &= \bigg( \, \sup_{(f,g) \in \msF\times\msF} \big| \langle \hbfP_{\!\cJ,n} \sn{f}, \sn{g} \rangle_\LPn- \langle \bfP_{\!\cJ} f, g \rangle_\LP 
		- \langle \hbfUpsilon_{\!\cJ,n} f, g \rangle_\LP \big| > \xi \bigg), \\
		E_2 &= \Big( \halpha_{\!\cJ,n} \le 1/4 \Big),
	\end{align*}
	where 
	\begin{equation*}
		\xi = 12 M_\msF^2 \theta_\tmax \bigg(\min\bigg\{D,\, 1+\frac{2\theta_\tmax-2\theta_\tmin}{\pi\gammaJ} \bigg\} \bigg) \bfa_{\!\cJ,n}^2(\tau),
	\end{equation*}
	On $E_1 \cap E_2$, Proposition~\ref{mercer:eigproj-step3} implies $\halpha_{\!\cJ,n} > \bfa_{\!\cJ,n}(\tau)$.
	Hence, applying Proposition~\refsupp{mercer:perturb} yields
	\begin{align*}
		\P (E_1) \leq \P(E_1 \cap E_2) + \P(E_2^c)
		\leq \tau + 2 \bfsigma_{\!\!\cJ} \exp \bigg( -\frac{n\,\gammaJ}{70\theta_\tmax \bfkappa_{\!\cJ}} \bigg).
	\end{align*}
	This concludes the proof.
\end{proof}

\begin{proof}[Proof of Corollary~\refsupp{mercer:eigproj-weak}]
	For $f,g \in \msF$, set
	\begin{gather*}
		V_n (f,g) = \sqrt{n} \left( \langle \hbfP_{\!\cJ,n} \sn{f}, \sn{g} \rangle_\LPn- \langle \bfP_{\!\cJ} f, g \rangle_\LP \right),\\
		Z_n (f,g) = \sqrt{n} \langle \hbfUpsilon_{\!\cJ,n} f, g \rangle_\LP,
		\quad
		Z (f,g) = \langle \bfUpsilon_{\!\cJ} f, g \rangle_\LP.
	\end{gather*}
	By the central limit theorem, it is clear that $Z_n$ converges to $Z$ in the sense of finite-dimensional distributions.
	
	Note that by \eqref{mercer:eigproj-step3-eq3} and Lemma~\refsupp{mercer:eigproj-duality},
	\begin{align*}
		\left| Z_n (f,g) \right| 
		= \sqrt{n} \left| \langle \hcA_{\cJ,n} f , g \rangle_\LP \right|
		= \sqrt{n} \left| \langle \cH \hcA_{\cJ,n} f , g \rangle_\mbH \right|
		\leq \sqrt{n} \|\cH\|_{\op,\mbH} \| \hcA_{\cJ,n} \|_{\op,\mbH} \|f\|_\mbH \|g\|_\mbH,
	\end{align*}
	and by Lemma~\ref{mercer:eigproj-step1} and Proposition~\refsupp{mercer:perturb},
	\begin{align*}
		\| \hcA_{\cJ,n} \|_{\op,\mbH} 
		= \OP(n^{-1/2}).
	\end{align*}
	Hence, we have
	\begin{align*}
		\| Z_n \|_{{\fB_{\msF\times\msF}}} 
		= \sup_{(f,g) \in \msF\times\msF} \left| \sqrt{n} \langle \hcA_{\cJ,n} f , g \rangle_\LP \right|
		= \OP(1),
	\end{align*}
	which implies $Z_n$ is asymptotically tight in $\fB_\msF$.
	Therefore, by Theorem~1.5.4 in \citesupp{vaart2023Weak}, $Z_n$ converges weakly to $Z$ in $\fB_\msF$.
	
	Finally, from Theorem~\refsupp{mercer:eigproj-concen}, it is clear that
	\begin{gather}
		\sup_{(f,g) \in \msF\times\msF} \big| \langle \hbfP_{\!\cJ,n} \sn{f}, \sn{g} \rangle_\LPn- \langle \bfP_{\!\cJ} f, g \rangle_\LP 
		- \langle \hbfUpsilon_{\!\cJ,n} f, g \rangle_\LP \big| = \OP (n^{-1}),
	\end{gather}
	and this implies $\| V_n - Z_n \|_{\fB_{\msF\times\msF}}=\OP (n^{-1/2})$.
	Therefore, by Slutsky's theorem, we conclude that $V_n$ converges weakly to $Z$ in $\fB_\msF$, completing the proof.
\end{proof}

\subsection{Proof of Theorem~\refsupp{mercer:eigvec-concen} and Corollary~\refsupp{mercer:eigvec-weak}}

We start by stating and proving non-stochastic bounds for eigenvectors.

\begin{proposition} \label{mercer:eigvec:aux}
	Suppose Assumptions \refsupp{mercer:assump1} and \refsupp{mercer:assump2} hold.
	If $\halpha_{\!\cJ,n} \le 1/4$, we have
	\begin{equation}
		\bigg| \sum_{k \in \cJ} \sum_{\ell \in \cJ} \langle \hvarphi_{k,n}, \sn{\phi_\ell} \rangle_\LPn^2	- |\cJ| - (\Pn-\P) \Big( \sum_{k \in \cJ} \phi_k^2 \Big) \bigg| 
		\\ \leq
		\frac{11}{2} \hbeta_{\!\cJ,n}^2.
	\end{equation}
\end{proposition}

\begin{proof}[Proof of Proposition~\ref{mercer:eigvec:aux}]
	From the proof of Proposition~\ref{mercer:eigproj-step3}, we see that for $\ell \in \cJ$,
	\begin{align*}
		&\quad \langle (\hcH_n \hcP_{\!\cJ,n} - \cH \cP_{\!\cJ} -\hcA_\cJ) \phi_\ell, \phi_\ell \rangle_\mbH
		\\&= \langle \hbfP_{\!\cJ,n} \sn{\phi_\ell}, \sn{\phi_\ell} \rangle_\LPn - \langle \bfP_{\!\cJ} \phi_\ell, \phi_\ell \rangle_\LP 
		- \langle \hbfUpsilon_{\!\cJ,n} \phi_\ell, \phi_\ell \rangle_\LP
		\\&= \sum_{k \in \cJ} \langle \hvarphi_{k,n}, \sn{\phi_\ell} \rangle_\LPn \langle \hvarphi_{k,n}, \sn{\phi_\ell} \rangle_\LPn  - 1 - (\Pn-\P)(\phi_\ell^2).
	\end{align*}
	With $\Lambda_{\!\cJ}^{-1/2} = \diag(\lambda_\ell^{-1/2})_{\ell\in\cJ}$, the above equation yields 
	\begin{multline*}
		\sum_{k \in \cJ} \sum_{\ell \in \cJ} \langle \hvarphi_{k,n}, \sn{\phi_\ell} \rangle_\LPn^2	- |\cJ| - (\Pn-\P) \Big( \sum_{k \in \cJ} \phi_k^2 \Big)
		\\= \sum_{\ell \in \cJ} \langle (\hcH_n \hcP_{\!\cJ,n} - \cH \cP_{\!\cJ} -\hcA_\cJ) \phi_\ell, \phi_\ell \rangle_\mbH
		= \sum_{\ell \in \cJ} \langle (\hcH_n \hcP_{\!\cJ,n} - \cH \cP_{\!\cJ} -\hcA_\cJ) \lambda_\ell^{-1/2} \psi_\ell, \lambda_\ell^{-1/2} \psi_\ell \rangle_\mbH
		\\= \tr \Big( \Lambda_{\!\cJ}^{-1/2} \cP_{\!\cJ} (\hcH_n \hcP_{\!\cJ,n} - \cH \cP_{\!\cJ} -\hcA_\cJ) \Lambda_{\!\cJ}^{-1/2} \Phi_{\!\cJ} \Big).
	\end{multline*}
	By applying Proposition~\refsupp{spec:comp3} with $f(z)=z$, we obtain
	\begin{equation}
		\bigg| \sum_{k \in \cJ} \sum_{\ell \in \cJ} \langle \hvarphi_{k,n}, \sn{\phi_\ell} \rangle_\LPn^2	- |\cJ| - (\Pn-\P) \Big( \sum_{k \in \cJ} \phi_k^2 \Big) \bigg|
		\leq \bigg( \frac{3\gammaJ}{2\theta_\tmin} + 4\bigg) \hbeta_{\!\cJ,n}^2.
	\end{equation}
	This completes the proof by using the assumption $\gamma_\cJ \leq \theta_\tmin$.
\end{proof}

We now prove Theorem~\refsupp{mercer:eigvec-concen} and Corollary~\refsupp{mercer:eigvec-weak} as follows.

\begin{proof}[Proof of Theorem~\refsupp{mercer:eigvec-concen}]
	Set the events $E_1$ and $E_2$ as
	\begin{align*}
		E_1 &= \bigg(\, \bigg| \sum_{k \in \cJ} \sum_{\ell \in \cJ} \langle \hvarphi_{k,n}, \sn{\phi_\ell} \rangle_\LPn^2	- |\cJ| - (\Pn-\P) \Big( \sum_{k \in \cJ} \phi_k^2 \Big) \bigg| > \xi := \frac{11}{2} \bfb_{\!\cJ,n}^2(\tau) \bigg), \\
		E_2 &= \Big( \halpha_{\!\cJ,n} \le 1/4 \Big),
	\end{align*}
	On $E_1 \cap E_2$, Proposition~\ref{mercer:eigvec:aux} implies
	$\frac{11}{2} \hbeta_{\!\cJ,n}^2 > \xi$, which is equivalent to $\hbeta_{\!\cJ,n} > \bfb_{\!\cJ,n} (\tau) $.
	Hence, applying Proposition~\ref{mercer:perturb} yields
	\begin{align*}
		\P (E_1) \leq \P(E_1 \cap E_2) + \P(E_2^c)
		\leq \tau + 2 \bfsigma_{\!\!\cJ} \exp \bigg( -\frac{n \gammaJ}{70\theta_\tmax \bfkappa_{\!\cJ}} \bigg).
	\end{align*}
	This concludes the proof.
\end{proof}

\begin{proof}[Proof of Corollary~\refsupp{mercer:eigvec-weak}]
	By Theorem~\refsupp{mercer:eigvec-concen}, it can be easily show that
	\begin{equation*}
		\sum_{k \in \cJ} \sum_{\ell \in \cJ} \langle \hvarphi_{k,n}, \sn{\phi_\ell} \rangle_\LPn^2	- |\cJ| - (\Pn-\P) \Big( \sum_{k \in \cJ} \phi_k^2 \Big)
		= \OP(n^{-1}).
	\end{equation*}
	Hence, the claim follows by the central limit theorem and Slutsky's theorem.
\end{proof}

\subsection{Proof of Theorem~\refsupp{mercer:eigval-concen} and Corollary~\refsupp{mercer:eigval-weak}}

\begin{proof}[Proof of part (a) of Theorem~\refsupp{mercer:eigval-concen}]
	Define the events $E$ and $E_1, \dots, E_D$ as
	\begin{align*}
		E &= \bigg( \,\, \bigg\| (\hlambda_{k,n} - \lambda_k)_{k \in \cJ} - \bigoplus_{d=1}^D \specd \Big( \theta_j (\P_n-\P) (\phi_k \phi_\ell) \Big)_{k,\ell \in \cJ_d} \bigg\|_2 > \xi \bigg), \\
		E_d &= \big( \halpha_{\!\cJ_d,n} \le 1/4 \big), \quad d=1,\dots,D,
	\end{align*}
	where 
	\begin{equation*}
		\xi = \Bigg\| \bigg( \min \bigg\{ \frac{36+14\sqrt{2}}{9}\gammaJd\bfa_{\!\cJ_d,n} \Big(\frac{\tau}{2D}\Big) \bfb_{\!\cJ_d,n} \Big(\frac{\tau}{2D}\Big) ,\,\, (3+\sqrt{2}) \gammaJd \bfb_{\!\cJ_d,n}^2 \Big(\frac{\tau}{2D}\Big) \bigg\} \, \bigg)_{d \in [D]} \Bigg\|_2.
	\end{equation*}	
	For $k,\ell \in \cJ_d$, applying \eqref{mercer:eigproj-step2-eq1} in Lemma~\ref{mercer:eigproj-step2} yields
	\begin{equation*}
		\langle (\hcH-\cH) \psi_k, \psi_\ell \rangle_\mbH
		= \sqrt{\lambda_k \lambda_\ell} (\P_n-\P) (\phi_k \phi_\ell)
		= \theta_d (\P_n-\P) (\phi_k \phi_\ell),
	\end{equation*}	
	and hence,
	\begin{equation*}
		\bigoplus_{d=1}^D \specd \left( \langle (\hcH-\cH) \psi_k, \psi_\ell \rangle_\mbH \right)_{k,\ell \in \cJ_d} 
		= \bigoplus_{d=1}^D \specd \Big( \theta_d (\P_n-\P) (\phi_k \phi_\ell) \Big)_{k,\ell \in \cJ_d}.
	\end{equation*}
	Thus, on $E \cap \left(\bigcap_{d=1}^D E_d \right)$, Theorem~\refsupp{spec:eigval-separate} gives
	\begin{equation*}
		\Bigg\| \bigg( \min \bigg\{ \frac{36+14\sqrt{2}}{9}\gammaJd \halpha_{\!\cJ_d,n} \hbeta_{\!\cJ_d,n} ,\,\,
		(3+\sqrt{2}) \gammaJd \hbeta_{\!\cJ_d,n}^2  \bigg\} \, \bigg)_{d \in [D]} \Bigg\|_2 > \xi,
	\end{equation*}
	which further implies
	\begin{equation*}
		\text{either}\hquad \halpha_{\!\cJ_d,n} > \bfa_{\!\cJ_d,n} \Big(\frac{\tau}{2D}\Big) 
		\hquad \text{or}\hquad \hbeta_{\!\cJ_d,n} > \bfb_{\!\cJ_d,n} \Big(\frac{\tau}{2D}\Big) 
		\quad\text{for some $d=1,\dots,D$}.
	\end{equation*}

	Therefore, under the condition $n \ge \max\limits_{d \in [D]} \frac{38 \theta_d \bfkappa_{\!\cJ_d}}{\gammaJd}$, applying Proposition \refsupp{mercer:perturb} yields
	\begin{align*}
		\P(E) &\le \P\left(E \cap \left(\bigcap_{d=1}^D E_d \right)\right) + \sum_{d=1}^D\P(E_d^c)
		\\&\le \sum_{d=1}^D \P\Big(\halpha_{\!\cJ_d,n} > \bfa_{\!\cJ,n}\Big(\frac{\tau}{2D}\Big) \Big) 
		+ \sum_{d=1}^D \P\Big(\hbeta_{\!\cJ_d,n} > \bfb_{\!\cJ,n}\Big(\frac{\tau}{2D}\Big) \Big)
		+ \sum_{d=1}^D\P(E_d^c)
		\\&\le \tau + \sum_{d=1}^D \bfsigma_{\!\!\cJ_d} \exp \bigg( -\frac{n \gammaJd}{70\theta_d \bfkappa_{\!\cJ_d}} \bigg).
	\end{align*}
	This completes the proof.
\end{proof}

\begin{proof}[Proof of part (b) of Theorem~\refsupp{mercer:eigval-concen}]
	Define the events $E_1$ and $E_2$ as
	\begin{align*}
		E_1 &= \bigg(\,\, \bigg| \sum_{k\in\cJ} (\hlambda_{k,n} - \lambda_k)
		- (\Pn-\P) \bigg( \sum_{k \in \cJ} \lambda_k \phi_k^2 \bigg) \,\bigg| > \xi \bigg), \\
		E_2 &= \big( \halpha_{\!\cJ,n} \le 1/4 \big),
	\end{align*}
	where
	\begin{multline*}
		\xi	= 
		\sqrt{|\cJ|}\cdot \min \bigg\{ \, \bigg( \Big(\frac{56\sqrt{D}}{9} \bfa_{\!\cJ,n} \Big(\frac{\tau}{2}\Big) + 4 \Big)\gammaJ + \frac{56\sqrt{D}}{9} (\theta_\tmax-\theta_\tmin) \bigg) \bfa_{\!\cJ,n} \Big(\frac{\tau}{2}\Big) \bfb_{\!\cJ,n} \Big(\frac{\tau}{2}\Big),
		\\ 4( \gammaJ + \theta_\tmax-\theta_\tmin) \bfb_{\!\cJ,n}^2 \Big(\frac{\tau}{2}\Big) \bigg\}.
	\end{multline*}
	
	On $E_1 \cap E_2$, Theorem~\refsupp{spec:eigval-cluster} gives
	\begin{equation*}
		\min \bigg\{ \, \bigg( \Big(\frac{56\sqrt{D}}{9} \halpha_{\!\cJ,n} + 4 \Big)\gammaJ + \frac{56\sqrt{D}}{9} (\theta_\tmax-\theta_\tmin) \bigg) \halpha_{\!\cJ,n} \hbeta_{\!\cJ,n}, \,\,
		4 ( \gammaJ + \theta_\tmax-\theta_\tmin) \hbeta_{\!\cJ,n}^2 \bigg\}
		> \xi,
	\end{equation*}
	which further implies either $\halpha_{\!\cJ,n} > \bfa_{\!\cJ,n} (\tau/2) $ or $\hbeta_{\!\cJ,n} > \bfb_{\!\cJ,n} (\tau/2) $.
	Hence, applying Proposition~\ref{mercer:perturb} yields
	\begin{align*}
		\P (E_1) \leq \P(E_1 \cap E_2) + \P(E_2^c)
		\leq \tau + 2 \bfsigma_{\!\!\cJ} \exp \bigg( -\frac{n \gammaJ}{70\theta_\tmax \bfkappa_{\!\cJ}} \bigg).
	\end{align*}
	This concludes the proof.
\end{proof}

\begin{proof}[Proof of Corollary \refsupp{mercer:eigval-weak}]
	From Theorem~\refsupp{mercer:eigval-concen}, it can be easily shown that
	\begin{gather*}
		\bigg\| (\hlambda_{k,n} - \lambda_k)_{k \in \cJ} - \bigoplus_{d=1}^D \specd \Big( \theta_d (\P_n-\P) (\phi_k \phi_\ell) \Big)_{k,\ell \in \cJ_d} \bigg\|_2
		= \OP(n^{-1}).
	\end{gather*}
	Hence, by the multivariate central limit theorem and Slutsky's theorem, the first claim follows.
	The second claim can be proved in the same way.
\end{proof}

\section{Technical proofs}

\subsection{Proof of Proposition~\refsupp{spec:comp1}}

Our proof is based on resolvent-based holomorphic calculus as in Theorem~5.14 of \citesupp{hsing2015theoretical} or Lemma~5.2 of \citesupp{koltchinskii2000random}.
Unlike these existing results, we aim to derive sharper and more general bounds by working with weighted perturbation quantities and general spectral subsets.
We begin with auxiliary results needed for the proof of Proposition~\refsupp{spec:comp1}.

For $z \in \C$ with $z \notin \{\lambda_k\}_{k\in\N}$, we consider the resolvent $\fR_z$ of $\cH$ and its square root $\fR_z^{1/2}$:
\begin{equation*}
	\fR_z = (z \cI_\mbH - \cH)^{-1} = \sum_{k=1}^\infty \frac{\cQ_k}{z-\lambda_k}, \quad
	\fR_z^{1/2} = (z \cI_\mbH - \cH)^{-1/2} 
	= \sum_{k=1}^\infty \frac{\cQ_k}{\sqrt{z-\lambda_k}}.
\end{equation*}
Recall from Proposition~\refsupp{spec:comp1} that $\partial\Gamma_{\!\cJ} \subset \C$ is the positively oriented boundary of 
\begin{equation*}
	\Gamma_{\!\cJ} = \bigcup_{d=1}^D \bigg\{z \in \C: |z-\theta_d| < \frac{\gammaJ}{2} \bigg\}.
\end{equation*}
Note that it can be easily shown that
\begin{equation*}
	\sup_{z \in \partial\Gamma_{\!\cJ}} \|\fR_z\|_{\op,\mbH} = \sup_{z \in \partial\Gamma_{\!\cJ}} \|\fR_z^{1/2}\|^2_{\op,\mbH} = \frac{2}{\gammaJ}, \quad
	\oint_{\partial\Gamma_{\!\cJ}} |dz| \leq \min\big\{ \pi D \gammaJ,\hquad \pi\gammaJ+ 2(\theta_\tmax - \theta_\tmin) \big\}.
\end{equation*}
With $\fR_z^{1/2}$ and $\partial\Gamma_{\!\cJ}$, we define $\hfa_{\!\cJ}$ by
\begin{equation*}
	\hfa_{\!\cJ} = \sup_{z \in \partial\Gamma_{\!\cJ}} \| \fR_z^{1/2} (\hcH-\cH) \fR_z^{1/2} \|_{\op,\mbH}.
\end{equation*}
By construction, $\hfa_{\!\cJ}$ is well-defined since $\lambda_k \notin \partial\Gamma_{\!\cJ}$ for all $k \in \N$.

$\fR_z^{1/2}$ and $\hfa_{\!\cJ}$ has been studied in the existing literature in probabilistic setting; see, Lemma~10.2 of \citesupp{hilgert2013minimax}, Section~4 of \citesupp{mas2015high} or Section~6.1 of \citesupp{cardot2007clt}.
We adapt the results from these works to our setting.

Our strategy is to prove several results based on $\hfa_{\!\cJ}$ and to bound them using $\halpha_{\!\cJ}$.
To this end, we first show $\hfa_{\!\cJ}$ is highly related to $\halpha_{\!\cJ}$ as follows.

\begin{lemma} \label{spec:comp1:aux1}
	Under Assumptions~\refsupp{spec:assump1} and \refsupp{spec:assump2}, we have $\hfa_{\!\cJ} \leq 2 \halpha_{\!\cJ}$.
\end{lemma}

\begin{proof}[Proof of Lemma~\ref{spec:comp1:aux1}]
	We define $
	\cT_\cJ^{\,-1/2} = \big(\cT_{\!\cJ}^{1/2}\big)^{-1} 
	= \sqrt{\gammaJ}\, \cP_{\!\cJ} 
	+ \sum_{k \not \in \cJ} \min_{m \in \cJ} |\lambda_k - \lambda_m| \cQ_k.
	$
	Using the identity 
	\begin{equation}
		\cT_\cJ^{\,-1/2} \fR_z^{1/2}
		= \fR_z^{1/2} \cT_\cJ^{\,-1/2}
		= \sum_{k \in \cJ} \sqrt{\frac{\gammaJ}{z-\lambda_k}} \cQ_k
		+ \sum_{k \notin \cJ} \sqrt{\frac{\min_{m \in \cJ}|\lambda_k - \lambda_m|}{z-\lambda_k}} \cQ_k,
		\label{spec:comp1:aux1-eq1}
	\end{equation}
	we bound $\hfa_{\!\cJ}$ as
	\begin{align}
		\hfa_{\!\cJ}
		&= \sup_{z \in \partial\Gamma_{\!\cJ}} \| \fR_z^{1/2}\cT_\cJ^{\,-1/2}
		\cT_{\!\cJ}^{1/2} (\hcH-\cH) \cT_{\!\cJ}^{1/2}
		\cT_\cJ^{\,-1/2} \fR_z^{1/2} \|_{\op,\mbH}
		\notag\\ &\leq \Big(\sup_{z \in \partial\Gamma_{\!\cJ}} \|\cT_\cJ^{\,-1/2} \fR_z^{1/2} \|_{\op,\mbH}\Big)^2 
		\halpha_{\!\cJ}.
		\label{spec:comp1:aux1-eq3}
	\end{align}
	To complete the proof, it suffices to show that
	$ \sup_{z \in \partial\Gamma_{\!\cJ}} \|\cT_\cJ^{\,-1/2} \fR_z^{1/2} \|_{\op,\mbH} = \sqrt{2}. $
	
	To this end, note from \eqref{spec:comp1:aux1-eq1} that $\psi_k$ is an eigenvector of $\cT_\cJ^{\,-1/2} \fR_z^{1/2}$ with corresponding eigenvalue $\sqrt{\frac{\gammaJ}{z-\lambda_k}}$ if $k \in \cJ$, and $\sqrt{\frac{\min_{m \in \cJ}|\lambda_k - \lambda_m|}{z-\lambda_k}}$ if $k \notin \cJ$.
	This gives
	\begin{align}
		&\quad \sup_{z \in \partial\Gamma_{\!\cJ}} \|\cT_\cJ^{\,-1/2} \fR_z^{1/2}\|_{\op,\mbH}
		= \sup_{z \in \partial\Gamma_{\!\cJ}} \max \Bigg\{ 
		\max_{k \in \cJ} \sqrt{\frac{\gammaJ}{|z-\lambda_k|}},
		\max_{k \notin \cJ} \sqrt{\frac{\min_{m \in \cJ} |\lambda_k - \lambda_m|}{|z-\lambda_k|}}
		\Bigg\}
		\notag\\&= \max \Bigg\{ 
		\max_{k \in \cJ} \sup_{z \in \partial\Gamma_{\!\cJ}} \sqrt{\frac{\gammaJ}{|z-\lambda_k|}},
		\max_{k \notin \cJ} \sup_{z \in \partial\Gamma_{\!\cJ}} \sqrt{\frac{\min_{m \in \cJ} |\lambda_k - \lambda_m|}{|z-\lambda_k|}}
		\Bigg\} =: \max\{ E_1, E_2 \}.
		\label{spec:comp1:aux1-eq2}
	\end{align}
	By the construction of $\Gamma_{\!\cJ}$, it is clear that 
	$$ \max_{k \in \cJ}\inf_{z \in \partial\Gamma_{\!\cJ}} |z-\lambda_k| = \frac{\gammaJ}{2}, \quad
	\inf_{z \in \partial\Gamma_{\!\cJ}} |z-\lambda_k| = \min_{m\in\cJ} |\lambda_m-\lambda_k| - \frac{\gammaJ}{2}
	\text{,\quad for $k \notin \cJ$}. $$
	These imply
	\begin{equation*}
		E_1 = \sqrt{\frac{\gammaJ}{\gammaJ/2}} = \sqrt{2}, \quad 
		E_2 = \max_{k \notin \cJ} \sqrt{\frac{\min_{m \in \cJ} |\lambda_k - \lambda_m|}{\min_{m \in \cJ} |\lambda_k - \lambda_m| - \gammaJ/2}}.
	\end{equation*}
	For $E_2$, since the function $x \mapsto \frac{x}{x-\gammaJ/2}$, $x>\gammaJ/2$, is strictly decreasing and $\min_{k \notin\cJ} \min_{m \in \cJ} |\lambda_k - \lambda_m| = \gammaJ$, we have
	\begin{align*}
		E_2 
		= \sqrt{\frac{\gammaJ}{\gammaJ - \gammaJ/2}} = \sqrt{2}.
	\end{align*}
	Therefore, we obtain $\max\{E_1,E_2\}=\sqrt{2}$, and the proof is complete by \eqref{spec:comp1:aux1-eq3} and \eqref{spec:comp1:aux1-eq2}.
\end{proof}

The next lemma provides how the eigenvalues corresponding to $\cJ$ are close and the others are separated.
This is a generalized result compared to those in existing literature; see, e.g., Lemma~2 of \cite{wahl2026higher}.

\begin{lemma}[Weyl-type inequality and separation of eigenvalues] \label{spec:comp1:aux2}
	Suppose Assumptions~\refsupp{spec:assump1} and \refsupp{spec:assump2} hold.
	If $\hfa_{\!\cJ} <1$, we have
	\begin{equation*}
		\max_{k \in \cJ} |\hlambda_k - \lambda_k| \leq \frac{\gammaJ\hfa_{\!\cJ}}{2}, \quad
		\min_{k \notin \cJ, \ell \in \cJ} |\hlambda_k - \lambda_\ell| \ge \gammaJ \Big(1-\frac{\hfa_{\!\cJ}}{2}\Big).
	\end{equation*}
	In particular, if $\halpha_{\!\cJ} < 1/2$, applying Lemma~\ref{spec:comp1:aux1} yields
	\begin{equation*}
		\max_{k \in \cJ} |\hlambda_k - \lambda_k| \leq \gammaJ \halpha_{\!\cJ}, \quad
		\min_{k \notin \cJ, \ell \in \cJ} |\hlambda_k - \lambda_\ell| \ge \gammaJ(1-\halpha_{\!\cJ}).
	\end{equation*}
\end{lemma}

\begin{proof}[Proof of Lemma~\ref{spec:comp1:aux2}]
	The second claim follows directly from the first claim; hence, it suffices to prove the first claim.
	
	For any $z \in \partial\Gamma_{\!\cJ}$ and any $w = \sum_{m=1}^\infty w_m \psi_m \in \mbH$,
	from
	\begin{align*}
		|\langle (\hcH-\cH)w, w \rangle_\mbH|
		&= |\langle \fR^{1/2}_z (\hcH-\cH) \fR^{1/2}_z \fR^{-1/2}_z w, \fR^{-1/2}_z w \rangle_\mbH|
		\\&\leq \|\fR^{1/2}_z (\hcH-\cH) \fR^{1/2}_z\|_{\op,\mbH} \|\fR^{-1/2}_z w\|_{\mbH}^2,
	\end{align*}
	where $\fR^{-1/2}_z = \sum_{k=1}^\infty \sqrt{z-\lambda_k} \cQ_k$, and
	\begin{gather*}
		\| \fR_z^{1/2} (\hcH-\cH) \fR_z^{1/2} \|_{\op,\mbH} \leq \hfa_{\!\cJ}, \quad
		\|\fR^{-1/2}_z w\|_\mbH^2 = \sum_{m=1}^\infty |z-\lambda_k| w_m^2, \quad
		\langle \cH w , w \rangle_\mbH = \sum_{m=1}^\infty \lambda_k w_m^2,
	\end{gather*}
	we obtain the inequalities
	\begin{equation}
		\sum_{m=1}^\infty \big( \lambda_m - \hfa_{\!\cJ}|z-\lambda_m| \big) w_m^2
		\leq \langle \hcH w , w \rangle_\mbH 
		\leq \sum_{m=1}^\infty \big(\lambda_m + \hfa_{\!\cJ} |z-\lambda_m| \big) w_m^2.
		\label{spec:aux1-eq1}
	\end{equation}
	
	On the other hand, according to the Courant–Fischer–Weyl min-max principle (see, e.g., Theorem~4.1 of \citesupp{teschl2014mathematical}),
	$\hlambda_k$ can be expressed in terms of the associated quadratic form as
	\begin{equation*}
		\hlambda_k
		= \sup_{S \in \mathscr{H}_k} \inf_{w\in S: \|w\|_\mbH=1} \langle \hcH w, w \rangle_\mbH
		= \inf_{S \in \mathscr{H}_{k-1}} \sup_{w\in S^\perp: \|w\|_\mbH=1} \langle \hcH w, w \rangle_\mbH
	\end{equation*}
	where $\mathscr{H}_m$ denotes the collection of all $m$-dimensional subspaces of $\mbH$.
	Let $\Phi_m$ denote the subspace of $\mbH$ generated by the eigenvectors $\{\psi_\ell\}_{\ell \in [m]}$.
	Since $\Phi_k \in \mathscr{H}_k$ and $\Phi_{k-1} \in \mathscr{H}_{k-1}$, we obtain
	\begin{equation}
		\inf_{w \in \Phi_k: \|w\|_\mbH =1} \langle \hcH w , w \rangle_\mbH 
		\leq \hlambda_k
		\leq \sup_{w \in \Phi_{\ge k}: \|w\|_\mbH =1} \langle \hcH w , w \rangle_\mbH,
		\label{spec:aux1-eq2}
	\end{equation}
	where $\Phi_{\ge k}$ denote the subspace of $\mbH$ generated by the eigenvectors $\{ \psi_\ell : \ell \ge k\}$.
	
	Combining \eqref{spec:aux1-eq1} and \eqref{spec:aux1-eq2} yields
	\begin{equation}
		\inf_{w \in \Phi_m: \|w\|_\mbH =1} \sum_{m=1}^k \big( \lambda_m - \hfa_{\!\cJ}|z-\lambda_m| \big) w_m^2
		\leq \hlambda_k \leq
		\sup_{w \in \Phi_{\ge k}: \|w\|_\mbH =1} \sum_{m=k}^\infty \big(\lambda_m + \hfa_{\!\cJ} |z-\lambda_m| \big) w_m^2.
		\label{spec:aux1-eq3}
	\end{equation}
	Now, we construct bounds by using \eqref{spec:aux1-eq3} and choosing $z$ and $k$ appropriately.	
	
	\mycase{Case 1. $k \in \cJ$}
	By construction of $\Gamma_{\!\cJ}$, we can find $z \in \partial\Gamma_{\!\cJ}$ with $|z-\lambda_k| = \gammaJ/2$.
	This gives $|z-\lambda_m| \leq |\lambda_k-\lambda_m| + |z-\lambda_k| 
	\leq \lambda_k -\lambda_m + \gammaJ/2$
	for any $m \ge k$. Hence,
	\begin{align}
		\hlambda_k
		&\leq \sup_{w \in \Phi_{\ge k}: \|w\|_\mbH =1} \sum_{m=k}^\infty \big(\lambda_m + \hfa_{\!\cJ} |z-\lambda_m| \big) w_m^2
		\notag\\&\leq \sup_{w \in \Phi_{\ge k}: \|w\|_\mbH =1} \sum_{m=k}^\infty 
		\Big( (1-\hfa_{\!\cJ})\lambda_m + \hfa_{\!\cJ}\lambda_k + \frac{\gammaJ\hfa_{\!\cJ}}{2} \Big) w_m^2
		\notag\\&= (1-\hfa_{\!\cJ})\lambda_k + \hfa_{\!\cJ}\lambda_k + \frac{\gammaJ\hfa_{\!\cJ}}{2} 
		= \lambda_k + \frac{\gammaJ\hfa_{\!\cJ}}{2} .
		\label{spec:aux1-eq4}
	\end{align}
	In the last line, we use the fact that $\{\lambda_m\}_{m\in\N}$ are arranged in decreasing order.
	Similarly, using $|z-\lambda_m| \le \lambda_m - \lambda_k + \gammaJ/2$ for any $m \le k$, we have
	\begin{align}
		\hlambda_k
		&\ge \inf_{w \in \Phi_k: \|w\|_\mbH =1} \sum_{m=1}^k \big( \lambda_m - \hfa_{\!\cJ}|z-\lambda_k| \big) w_m^2 
		\notag\\&\ge \inf_{w \in \Phi_k: \|w\|_\mbH =1} \sum_{m=1}^k \Big( (1-\hfa_{\!\cJ})\lambda_m + \lambda_k - \frac{\gammaJ}{2}\hfa  \Big) w_m^2
		\notag\\&= (1-\hfa_{\!\cJ})\lambda_k + \hfa_{\!\cJ}\lambda_k - \frac{\gammaJ\hfa_{\!\cJ}}{2}
		= \lambda_k - \frac{\gammaJ\hfa_{\!\cJ}}{2}.
		\label{spec:aux1-eq5}
	\end{align}
	Thus, combining \eqref{spec:aux1-eq4} and \eqref{spec:aux1-eq5} yields
	\begin{equation}
		|\hlambda_k - \lambda_k| \leq \frac{\gammaJ\hfa_{\!\cJ}}{2}.
		\quad \text{for all $k \in \cJ$}.
		\label{spec:aux1-eq6}
	\end{equation}
	
	\mycase{Case 2. $k \notin \cJ$}
	For such $k$, we can choose the maximal interval $\{k_\tmin,k_\tmin + 1,\ldots, k_\tmax\} \subset \N$ containing $k$ that is disjoint from $\cJ$.
	Note that either $k_{\tmin}>1$ or $k_{\tmax} < \infty$.
	So, it suffices to check the cases $k_{\tmin}>1$ and $k_{\tmax} < \infty$.
	
	For $k_\tmin>1$, it holds that $k_\tmin \not \in \cJ$ but $k_\tmin -1 \in \cJ$.
	By construction of $\Gamma_{\!\cJ}$, we can set $z = \lambda_{k_\tmin-1} - \frac{\gammaJ}{2} \in \partial\Gamma_{\!\cJ} \cap \R$ and it satisfies $|z-\lambda_m| = \lambda_{k_\tmin-1} - \lambda_m - \frac{\gammaJ}{2}$ for all $m \ge k_\tmin$.
	It follows that
	\begin{align*}
		\hlambda_k \leq \hlambda_{k_\tmin}
		&\le \sup_{w \in \Phi_{\ge k_\tmin}: \|w\|_\mbH =1} \sum_{m={k_\tmin}}^{\infty} 
		\big(\lambda_m + \hfa_{\!\cJ} |z-\lambda_m| \big) w_m^2
		\\&= \sup_{w \in \Phi_{\ge k_\tmin}: \|w\|_\mbH =1} \sum_{m=k_\tmin}^{\infty}
		\bigg((1-\hfa_{\!\cJ})\lambda_m + \hfa_{\!\cJ} \lambda_{k_\tmin-1} - \frac{\gammaJ\hfa_{\!\cJ}}{2} \bigg) w_k^2
		\\&= (1-\hfa_{\!\cJ})\lambda_{k_\tmin} + \hfa_{\!\cJ} \lambda_{k_\tmin-1} - \frac{\gammaJ\hfa_{\!\cJ}}{2},
	\end{align*}
	and this yields
	\begin{align*}
		\hlambda_k - \lambda_{k_\tmin -1}
		\leq (1-\hfa_{\!\cJ}) (\lambda_{k_\tmin} - \lambda_{k_\tmin-1}) -\frac{\gammaJ\hfa_{\!\cJ}}{2}
		\leq -(1-\hfa_{\!\cJ}) \cdot \gamma_\cJ -\frac{\gammaJ}{2}\hfa_{\!\cJ} = \frac{\gammaJ\hfa_{\!\cJ}}{2} -\gamma_\cJ.
	\end{align*}
	Hence, we obtain
	\begin{equation}
		\lambda_{k_\tmin -1} - \hlambda_k 
		\ge \gamma_\cJ\Big(1 - \frac{\hfa_{\!\cJ}}{2}\Big) >0.
		\label{spec:aux1-eq7}
	\end{equation}
	Similarly, it can be shown for $k_\tmax < \infty$,
	\begin{equation}
		\hlambda_k - \lambda_{k_\tmax+1} 
		\ge \gamma_\cJ\Big(1 - \frac{\hfa_{\!\cJ}}{2}\Big) >0.
		\label{spec:aux1-eq8}
	\end{equation}
	By combining \eqref{spec:aux1-eq7} and \eqref{spec:aux1-eq8}, and noting the eigenvalues are arranged in decreasing order, we conclude
	\begin{equation}
		\min_{\ell \in \cJ} |\hlambda_k - \lambda_\ell| \ge \gamma_\cJ\Big(1 - \frac{\hfa_{\!\cJ}}{2}\Big)
		\quad \text{for all $k \notin \cJ$.}
		\label{spec:aux1-eq9}
	\end{equation}
	
	Therefore, we conclude the proof by \eqref{spec:aux1-eq6} and \eqref{spec:aux1-eq9}.
	
\end{proof}

Similar to $\fR_z$, for $z \notin \{\hlambda_k\}_{k \in\N} $, we introduce the resolvent $\hfR_z$ of $\hcH$, defined as
\begin{equation*}
	\hfR_z = (z \cI_\mbH - \hcH)^{-1} = \sum_{k=1}^\infty \frac{\hcQ_k}{z-\hlambda_k}.
\end{equation*}
Note that under $\hfa_{\!\cJ} <1$, Lemma~\ref{spec:comp1:aux2} guarantees that $\hfR_z$ is well-defined for $z \in \Gamma_{\!\cJ}$.
The next lemma establishes a relationship between $\hfR_z$ and $\fR_z$.

\begin{lemma} \label{spec:comp1:aux3}
	Suppose Assumptions~\refsupp{spec:assump1} and \refsupp{spec:assump2} hold. 
	If $\hfa_{\!\cJ} <1$ and $z \in \partial\Gamma_{\!\cJ}$, we have
	\begin{align}
		\hfR_z - \fR_z - \fR_z (\hcH - \cH)\fR_z 
		&= \fR_z (\hcH-\cH) \fR_z^{1/2} \hfE_z \fR_z^{1/2} (\hcH-\cH) \fR_z
		\label{spec:comp1:aux3-res1}\\
		&= \fR_z (\hcH-\cH) \hfR_z (\hcH-\cH) \fR_z,
		\label{spec:comp1:aux3-res2}
	\end{align}
	where $\hfE_z = \sum_{k=0}^\infty \big[\fR_z^{1/2}(\hcH - \cH)\fR_z^{1/2} \big]^k$ satisfying 
	\begin{equation}
		\sup_{z \in \partial\Gamma_{\!\cJ}} \|\hfE_z\|_{\op,\mbH} \leq \frac{1}{1-\hfa_{\!\cJ}}.
		\label{spec:comp1:aux3-res3}
	\end{equation}
\end{lemma}

\begin{proof}[Proof of Lemma~\ref{spec:comp1:aux3}]
	From the identity,
	\begin{align*}
		\hfR_z^{-1} = z \cI_\mbH - \hcH	= \fR_z^{-1} - (\hcH - \cH)
		= \fR_z^{-1/2} \big[ \cI_\mbH - \fR_z^{1/2}(\hcH - \cH)\fR_z^{1/2} \big] \fR_z^{-1/2},
	\end{align*}
	we obtain
	\begin{equation*}
		\hfR_z = \fR_z^{1/2} \big[ \cI_\mbH - \fR_z^{1/2}(\hcH - \cH)\fR_z^{1/2} \big]^{-1} \fR_z^{1/2}.
	\end{equation*}
	Since $\|\fR_z^{1/2}(\hcH - \cH)\fR_z^{1/2}\|_{\op,\mbH} \leq \hfa_{\!\cJ}  < 1$, the Neumann expansion of $\big[ \cI_\mbH - \fR_z^{1/2}(\hcH - \cH)\fR_z^{1/2} \big]^{-1}$ yields
	\begin{equation}
		\hfR_z = \fR_z^{1/2} \sum_{k=0}^\infty \big[\fR_z^{1/2}(\hcH - \cH)\fR_z^{1/2} \big]^k \fR_z^{1/2} = \fR_z^{1/2}\hfE_z\fR_z^{1/2}.
		\label{spec:comp1:aux3-eq1}
	\end{equation}
	On the other hand, by expanding $\hfE_z$, we have
	\begin{align}
		\hfE_z &= \cI_\mbH + \hfR_z^{1/2} (\hcH-\cH) \hfR_z^{1/2} + \sum_{k=2}^\infty \big[\fR_z^{1/2}(\hcH - \cH)\fR_z^{1/2} \big]^k
		\notag\\&= \cI_\mbH + \hfR_z^{1/2} (\hcH-\cH) \hfR_z^{1/2} + 
		\hfR_z^{1/2} (\hcH-\cH) \hfR_z^{1/2} \sum_{k=0}^\infty \big[\fR_z^{1/2}(\hcH - \cH)\fR_z^{1/2} \big]^k \hfR_z^{1/2} (\hcH-\cH) \hfR_z^{1/2}
		\notag\\&= \cI_\mbH + \hfR_z^{1/2} (\hcH-\cH) \hfR_z^{1/2} + 
		\hfR_z^{1/2} (\hcH-\cH) \hfR_z^{1/2} \hfE_z \hfR_z^{1/2}(\hcH-\cH) \hfR_z^{1/2}.
		\label{spec:comp1:aux3-eq2}
	\end{align}
	Therefore, using \eqref{spec:comp1:aux3-eq1} and \eqref{spec:comp1:aux3-eq2}, we obtain
	\begin{align*}
		\hfR_z = \fR_z^{1/2} \hfE_z \fR_z^{1/2}
		&= \fR_z^{1/2} \big[\cI_\mbH + \hfR_z^{1/2} (\hcH-\cH) \hfR_z^{1/2} + 
		\hfR_z^{1/2} (\hcH-\cH) \hfR_z^{1/2} \hfE_z \hfR_z^{1/2}(\hcH-\cH) \hfR_z^{1/2}\big] \fR_z^{1/2}
		\\&= \fR_z + \hfR_z (\hcH-\cH) \hfR_z + \hfR_z (\hcH-\cH) \hfR_z^{1/2} \hfE_z \hfR_z^{1/2}(\hcH-\cH) \hfR_z.
	\end{align*}
	which implies \eqref{spec:comp1:aux3-res1}.  
	
	\eqref{spec:comp1:aux3-res2} follows from \eqref{spec:comp1:aux3-res1} and \eqref{spec:comp1:aux3-eq1}.
	
	Lastly, using $\sup_{z \in \partial\Gamma_{\!\cJ}} \|\fR_z^{1/2}(\hcH - \cH)\fR_z^{1/2}\|_{\op,\mbH} = \hfa_{\!\cJ} < 1$, we show \eqref{spec:comp1:aux3-res3} as follows:
	\begin{align*}
		\sup_{z \in \partial\Gamma_{\!\cJ}} \|\hfE_z\|_{\op,\mbH}
		\leq \sum_{k=0}^\infty \sup_{z \in \partial\Gamma_{\!\cJ}} \|\fR_z^{1/2}(\hcH-\cH) \fR_z^{1/2}\|^k_{\op,\mbH}
		= \sum_{k=0}^\infty \hfa_{\!\cJ}^k = \frac{1}{1-\hfa_{\!\cJ}}.
	\end{align*}
	This completes the proof.
\end{proof}

The following result presents how $\nabla_\cH f_\cJ$ can be expressed in terms of resolvents.

\begin{lemma} \label{spec:comp1:aux4}
	Under the same conditions as in Proposition~\refsupp{spec:comp1}, we have
	\begin{gather}
		\nabla_\cH f_\cJ  (\hcH) = \frac{1}{2 \pi i} \oint_{\partial\Gamma_{\!\cJ}} f(z) \fR_z(\hcH - \cH)\fR_z dz,
		\label{spec:aux2-res1} 
		\\
		\| \nabla_\cH f_\cJ  (\hcH) \|_{\op,\mbH} 
		\leq \bigg( \sup_{z \in \partial\Gamma_{\!\cJ}} |f(z)|\bigg) \bigg( \min\bigg\{D,\, 1+\frac{2\theta_\tmax-2\theta_\tmin}{\pi\gammaJ} \bigg\}\bigg) \hfa_{\!\cJ}.
		\label{spec:aux2-res2}
	\end{gather}
	Note that $\hfa_{\!\cJ}$ need not be less than 1.
\end{lemma}

\begin{proof}[Proof of Lemma~\ref{spec:comp1:aux4}]
	Using $\fR_z = \sum_{k=1}^\infty (z-\lambda_k)^{-1} \cQ_k$, we write
	\begin{align*}
		\frac{1}{2 \pi i} \oint_{\partial\Gamma_{\!\cJ}} f(z) \fR_z(\hcH - \cH)\fR_z dz
		&= \sum_{k=1}^\infty \sum_{\ell=1}^\infty \left( \frac{1}{2 \pi i} \oint_{\partial\Gamma_{\!\cJ}} \frac{f(z)}{(z-\lambda_k) (z-\lambda_\ell)} dz \right) \cQ_k (\hcH - \cH) \cQ_\ell.
	\end{align*}
	By splitting the double sum according to whether indices belong to the same cluster $\cJ_j$, different clusters, or one index lies in $\cJ$ and the other not, we have
	$$ \frac{1}{2 \pi i} \oint_{\partial\Gamma_{\!\cJ}} f(z) \fR_z(\hcH - \cH)\fR_z dz= \cE_1 + \cE_2 + \cE_3 + \cE_4,$$
	where
	\begin{align*}
		\cE_1 &= \sum_{d=1}^D \sum_{k\in\cJ_d} \sum_{\ell\in\cJ_d} \left( \frac{1}{2 \pi i} \oint_{\partial\Gamma_{\!\cJ}} \frac{f(z)}{(z-\lambda_k) (z-\lambda_\ell)} dz \right) \cQ_k (\hcH - \cH) \cQ_\ell,
		\\
		\cE_2 &= \sum_{\substack{1\leq d_1,d_2 \leq D \\ d_1 \not= d_2}} \sum_{k\in\cJ_{d_1}} \sum_{\ell\in\cJ_{d_2}} \left( \frac{1}{2 \pi i} \oint_{\partial\Gamma_{\!\cJ}} \frac{f(z)}{(z-\lambda_k) (z-\lambda_\ell)} dz \right) \cQ_k (\hcH - \cH) \cQ_\ell,
		\\
		\cE_3 &= \sum_{k \in\cJ} \sum_{\ell\notin\cJ} \left( \frac{1}{2 \pi i} \oint_{\partial\Gamma_{\!\cJ}} \frac{f(z)}{(z-\lambda_k) (z-\lambda_\ell)} dz \right) \cQ_k (\hcH - \cH) \cQ_\ell
		\\
		&\quad + \sum_{k \notin\cJ} \sum_{\ell\in\cJ} \left( \frac{1}{2 \pi i} \oint_{\partial\Gamma_{\!\cJ}} \frac{f(z)}{(z-\lambda_k) (z-\lambda_\ell)} dz \right) \cQ_k (\hcH - \cH) \cQ_\ell,
		\\
		\cE_4 & = \sum_{k \notin\cJ} \sum_{\ell\notin\cJ} \left( \frac{1}{2 \pi i} \oint_{\partial\Gamma_{\!\cJ}} \frac{f(z)}{(z-\lambda_k) (z-\lambda_\ell)} dz \right) \cQ_k (\hcH - \cH) \cQ_\ell.
	\end{align*}
	Using Cauchy's integral formula (see Lemma~\ref{sup:cauchy}), we evaluate the contour integrals above, yielding that $E_4=0$ and
	\begin{align*}
		\cE_1 
		&= \sum_{d=1}^D f'(\theta_d) \left( \sum_{k\in\cJ_d} \sum_{\ell\in\cJ_d} \cQ_k (\hcH-\cH) \cQ_\ell \right)
		= \sum_{d=1}^D f'(\theta_d) \cP_{\!\cJ_d} (\hcH-\cH) \cP_{\!\cJ_d}
		= \nabla_{\cH,1} f_\cJ (\hcH),
	\end{align*}
	\begin{align*}
		\cE_2 
		&= \sum_{\substack{1\leq d_1,d_2 \leq D \\ d_1 \not= d_2}} \sum_{k\in\cJ_{d_1}} \sum_{\ell\in\cJ_{d_2}} \frac{f(\theta_{d_2}) - f(\theta_{d_1})}{\theta_{d_2}-\theta_{d_1}} \cQ_k (\hcH - \cH) \cQ_\ell
		\\&= \sum_{\substack{1\leq d_1,d_2 \leq D \\ d_1 \not= d_2}} \frac{f(\theta_{d_2}) - f(\theta_{d_1})}{\theta_{d_2}-\theta_{d_1}} \bigg( \sum_{k\in\cJ_{d_1}}\cQ_k \bigg) (\hcH - \cH) \bigg( \sum_{\ell\in\cJ_{d_2}} \cQ_\ell \bigg)
		\\&= \sum_{\substack{1\leq d_1,d_2 \leq D \\ d_1 \not= d_2}} \frac{f(\theta_{d_2}) - f(\theta_{d_1})}{\theta_{d_2}-\theta_{d_1}}  \cP_{\!\cJ_{d_1}} (\hcH - \cH) \cP_{\!\cJ_{d_2}}
		= \nabla_{\cH,2} f_\cJ (\hcH),
	\end{align*}
	\begin{align*}
		\cE_3
		&= \sum_{k \in\cJ} \sum_{\ell\notin\cJ}
		\frac{f(\lambda_k)}{\lambda_k - \lambda_\ell} \cQ_k (\hcH - \cH) \cQ_\ell
		+ \sum_{k \notin\cJ} \sum_{\ell\in\cJ}
		\frac{f(\lambda_\ell)}{\lambda_\ell - \lambda_k} \cQ_k (\hcH - \cH) \cQ_\ell
		\\&= \sum_{k \in \cJ} f(\lambda_k) \sum_{\ell \notin \cJ} \frac{\cQ_k (\hcH-\cH) \cQ_\ell+ \cQ_\ell (\hcH-\cH) \cQ_k}{\lambda_k - \lambda_\ell}
		= \nabla_{\cH,3} f_\cJ (\hcH).
	\end{align*}
	Collecting the four parts yields the first claim.
	
	For the second claim, using the first claim, the definition of $\hfa_{\!\cJ}$ and 
	\begin{equation}
		\sup_{z \in \partial\Gamma_{\!\cJ}} \|\fR_z^{1/2} \| = \sqrt{\frac{2}{\gammaJ}}, \quad 
		\frac{1}{2 \pi} \oint_{\partial\Gamma_{\!\cJ}} |dz| \leq \frac{\gammaJ}{2} \cdot 
		\min\bigg\{D,\, 1+\frac{2\theta_\tmax-2\theta_\tmin}{\pi\gammaJ} \bigg\},
		\label{spec:comp1:aux4-eq1}
	\end{equation}	
	we have
	\begin{align*}
		\| \nabla_\cH f_\cJ  (\hcH) \|_{\op,\mbH} 
		&\leq \bigg(\frac{1}{2 \pi} \oint_{\partial\Gamma_{\!\cJ}} |dz|\bigg) \sup_{z\in\partial\Gamma_{\!\cJ}} 
		\Big( |f(z)|\cdot \|\fR_z (\hcH-\cH) \fR_z\|_{\op,\mbH} \Big)
		\\&\leq \bigg(\frac{1}{2 \pi} \oint_{\partial\Gamma_{\!\cJ}} |dz|\bigg) \sup_{z\in\partial\Gamma_{\!\cJ}} 
		\Big( |f(z)|\cdot \|\fR_z^{1/2}\|_{\op,\mbH}^2\cdot\|\fR_z^{1/2} (\hcH-\cH) \fR_z^{1/2}\|_{\op,\mbH} \Big)
		\\&\leq \bigg( \sup_{z \in \partial\Gamma_{\!\cJ}} |f(z)|\bigg) \bigg( \min\bigg\{D,\, 1+\frac{2\theta_\tmax-2\theta_\tmin}{\pi\gammaJ} \bigg\}\bigg) \hfa_{\!\cJ}.
	\end{align*}
	This completes the proof.
\end{proof}

Now, we are ready to prove Proposition~\refsupp{spec:comp1}.

\begin{proof}[Proof of Proposition~\refsupp{spec:comp1}]
	The second claim follows by Lemma~\ref{spec:comp1:aux4} and $\hfa_{\!\cJ} \leq 2\halpha$ from Lemma~\ref{spec:comp1:aux1}.
	
	To show the first claim, note that by construction, $\Gamma_{\!\cJ}$ encloses precisely the eigenvalues $\{\lambda_k \}_{k \in \cJ}$ of $\cH$ and excludes $\{\lambda_k\}_{k \notin\cJ}$.
	It follows from the Dunford-Taylor functional integral (see \citesupp{kato1995perturbation} or \citesupp{dunford1988linear}) that
	\begin{equation*}
		f_\cJ (\cH) = \frac{1}{2 \pi i} \oint_{\partial\Gamma_{\!\cJ}} f(z) \fR_z dz.
	\end{equation*}
	Under $\hfa_{\!\cJ} < 1$, Lemma~\ref{spec:comp1:aux2} guarantees that $\Gamma_{\!\cJ}$ encloses precisely the eigenvalues $\{\hlambda_k \}_{k \in \cJ}$ of $\hcH$ and excludes $\{\hlambda_k\}_{k \notin\cJ}$. 
	So, we also have	
	\begin{equation*}
		f_\cJ (\hcH) = \frac{1}{2 \pi i} \oint_{\partial\Gamma_{\!\cJ}} f(z) \hfR_z dz.
	\end{equation*}
	By combining these and applying Lemmas~\ref{spec:comp1:aux3} and \ref{spec:comp1:aux4}, we obtain
	\begin{align}
		f_\cJ (\hcH) - f_\cJ (\cH) - \nabla_\cH f_\cJ  (\hcH)
		&= \frac{1}{2 \pi i} \oint_{\partial\Gamma_{\!\cJ}} f(z) \big[\hfR_z - \fR_z - \fR_z (\hcH-\cH) \fR_z \big] dz
		\notag\\&= \frac{1}{2 \pi i} \oint_{\partial\Gamma_{\!\cJ}} f(z) \fR_z (\hcH-\cH) \hfR_z (\hcH-\cH) \fR_z dz
		\label{spec:comp1-eq1}
		\\&= \frac{1}{2 \pi i} \oint_{\partial\Gamma_{\!\cJ}} f(z) \fR_z (\hcH-\cH) \fR_z^{1/2} \hfE_z \fR_z^{1/2} (\hcH-\cH) \fR_z dz.
		\label{spec:comp1-eq2}
	\end{align}
	To complete the proof, we bound the OP norm of \eqref{spec:comp1-eq2}.
	\eqref{spec:comp1-eq1} will be used when proving Propositions~\ref{spec:comp2} and \ref{spec:comp3}.
	
	Using \eqref{spec:comp1:aux3-res3} and \eqref{spec:comp1:aux4-eq1},
	we bound the OP norm of \eqref{spec:comp1-eq2} as
	\begin{align*}
		&\quad \left\| \frac{1}{2 \pi i} \oint_{\partial\Gamma_{\!\cJ}} f(z) \fR_z (\hcH-\cH) \fR_z^{1/2} \hfE_z \fR_z^{1/2} (\hcH-\cH) \fR_z dz \right\|_{\op,\mbH}
		\\&\leq  \bigg( \frac{1}{2 \pi} \oint_{\partial\Gamma_{\!\cJ}} |dz| \bigg)
		\cdot \sup_{z \in \partial\Gamma_{\!\cJ}} \bigg(  |f(z)| \cdot \| \fR_z (\hcH-\cH) \fR_z^{1/2} \hfE_z \fR_z^{1/2} (\hcH-\cH) \fR_z \|_{\op,\mbH} \bigg)
		\\&\leq \bigg( \frac{1}{2 \pi} \oint_{\partial\Gamma_{\!\cJ}} |dz| \bigg)
		\cdot \sup_{z \in \partial\Gamma_{\!\cJ}} \bigg(  |f(z)| \cdot \|\fR_z^{1/2}\|_{\op,\mbH}^2 \cdot \|\fR_z^{1/2} (\hcH-\cH) \fR_z^{1/2}\|_{\op,\mbH}^2 \cdot \|\hfE_z\|_{\op,\mbH} \bigg)
		\\&\leq \frac{\gammaJ}{2} \cdot \bigg(\min\bigg\{D,\, 1+\frac{2\theta_\tmax-2\theta_\tmin}{\pi\gammaJ} \bigg\} \bigg)
		\cdot \bigg( \sup_{z \in \partial\Gamma_{\!\cJ}} |f(z)| \bigg) 
		\cdot \bigg(\sqrt{\frac{2}{\gammaJ}}\bigg)^2 \cdot \hfa_{\!\cJ}^2 \cdot \frac{1}{1-\hfa_{\!\cJ}}
		\\&\leq \bigg( \sup_{z \in \partial\Gamma_{\!\cJ}} |f(z)| \bigg) \bigg(\min\bigg\{D,\, 1+\frac{2\theta_\tmax-2\theta_\tmin}{\pi\gammaJ} \bigg\} \bigg) \frac{\hfa^2_\cJ}{1-\hfa_{\!\cJ}}.
	\end{align*}
	Therefore, the proof is complete by using $\hfa_{\!\cJ} \leq 2 \halpha_{\!\cJ} \leq 1/2$ from Lemma~\ref{spec:comp1:aux1} and the condition $\halpha_{\!\cJ}\leq1/4$.
\end{proof}

\subsection{Proof of Proposition~\refsupp{spec:comp2}}

The resolvent-based holomorphic calculus used to prove Proposition~\refsupp{spec:comp1} has the advantage of revealing the first-order leading term explicitly and providing the remainder term via the Dunford-Taylor functional integral.
However, this integral representation does not yield a straightforward method for controlling the HS norm of the remainder term.

Due to this limitation, we explicitly evaluate the Dunford-Taylor integral to expand the remainder term, thereby showing how its HS norm can be bounded.
To control the HS norm, we adopt the contraction arguments from \citesupp{jirak2024quantitative} and \citesupp{wahl2026higher}, adapted to our general spectral subset setting.

We begin with stating and proving auxiliary lemmas for this approach.

\begin{lemma} \label{spec:comp2:aux1}
	Under Assumptions~\refsupp{spec:assump1} and \refsupp{spec:assump2}, we have
	\begin{equation*}
		\hcQ_k \cQ_\ell = \frac{\hcQ_k (\hcH-\cH) \cQ_\ell}{\hlambda_k - \lambda_\ell}
		\hquad\text{and}\hquad
		\cQ_k \hcQ_\ell = \frac{\cQ_k (\hcH-\cH) \hcQ_\ell}{\hlambda_\ell - \lambda_k}
		\quad\text{for any } k, \ell \in \N.
	\end{equation*}
\end{lemma}

\begin{proof}[Proof of Lemma~\ref{spec:comp2:aux1}]
	From $\hcQ_k \hcH = \hlambda_k \hcQ_k$ and $\cH \cQ_\ell  = \lambda_\ell \cQ_\ell$, we have
	\begin{align*}
		\hcQ_k (\hcH-\cH) \cQ_\ell = \hcQ_k \hcH \cQ_\ell - \hcQ_k \cH \cQ_\ell 
		= \hlambda_k \hcQ_k \cQ_\ell - \lambda_k \hcQ_k \cQ_\ell
		= (\hlambda_k - \lambda_\ell) \hcQ_k \cQ_\ell,
	\end{align*}
	which implies the first assertion. The second follows in the same way. 
\end{proof}

\begin{lemma} \label{spec:comp2:aux2}
	Under Assumptions~\refsupp{spec:assump1} and \refsupp{spec:assump2}, we have
	\begin{equation*}
		\max \Big\{ \gammaJ^{-1} \big\| \cP_{\!\cJ} (\hcH-\cH) \cP_{\!\cJ} \big\|_{\op,\mbH}, 
		\gammaJ^{-1/2} \big\| |\cR|_{\!\cJ}^{1/2} (\hcH-\cH) \cP_{\!\cJ} \big\|_{\op,\mbH}, 
		\big\| |\cR|_{\!\cJ}^{1/2} (\hcH-\cH) |\cR|_{\!\cJ}^{1/2} \big\|_{\op,\mbH} \Big\}
		\leq \halpha_{\!\cJ}.
	\end{equation*}
\end{lemma}

\begin{proof}[Proof of Lemma~\ref{spec:comp2:aux2}]
	Using $ \cP_{\!\cJ} \cT_{\!\cJ}^{1/2} = \gammaJ^{-1} \cP_{\!\cJ}$ and $\|\cP_{\!\cJ}\|_{\op,\mbH}=1$,
	we bound
	\begin{align*}
		\gammaJ^{-1} \| \cP_{\!\cJ} (\hcH-\cH) \cP_{\!\cJ} \|_{\op,\mbH}
		&= \| \cP_{\!\cJ} \cT_{\!\cJ}^{1/2} (\hcH-\cH)  \cT_{\!\cJ}^{1/2}\cP_{\!\cJ} \|_{\op,\mbH}
		\leq \|\cT_\cJ (\hcH-\cH) \cT_\cJ\|_{\op,\mbH} = \halpha_{\!\cJ}.
	\end{align*}
	Similarly, the other terms can be bounded by $\halpha_{\!\cJ}$ using $(\cI_\mbH-\cP_{\!\cJ}) \cT_{\!\cJ}^{1/2} = |\cR|_{\!\cJ}^{1/2}$ and $\|\cI_\mbH-\cP_{\!\cJ}\|_{\op,\mbH}=1$.
	This completes the proof.
\end{proof}

We define $|\cR|_{\!\cJ}^{-1/2}$ as
\begin{equation}
	|\cR|_{\!\cJ}^{-1/2} = \sum_{k \notin \cJ} \sqrt{\min_{m \in \cJ} |\lambda_k-\lambda_m|} \cQ_k.
\end{equation}
Note that $|\cR|_{\!\cJ}^{-1/2} |\cR|_{\!\cJ}^{1/2} = |\cR|_{\!\cJ}^{1/2}|\cR|_{\!\cJ}^{-1/2} = \cI_\mbH- \cP_{\!\cJ}$.	

With $|\cR|_{\!\cJ}^{-1/2}$, we have the following lemma.

\begin{lemma} \label{spec:comp2:aux3}
	Suppose Assumptions~\refsupp{spec:assump1} and \refsupp{spec:assump2} hold.
	If $\halpha_{\!\cJ} <1/2$, we have
	\begin{gather}
		\big\| |\cR|_{\!\cJ}^{-1/2} (\cI_\mbH - \hcP_{\!\cJ}) \cP_\cJ \big\|_{\HS,\mbH} 
		\leq \big\| |\cR|_{\!\cJ}^{-1/2} \hcP_{\!\cJ} \big\|_{\HS,\mbH} 
		\leq \frac{\sqrt{\gammaJ}\hbeta_{\!\cJ}}{1-2\halpha_{\!\cJ}},
		\label{spec:contraction2-res1} \\
		\big\| \cP_{\!\cJ} (\cI_\mbH-\hcP_{\!\cJ}) \big\|_{\HS,\mbH} = \big\|(\cI_\mbH-\cP_{\!\cJ}) \hcP_{\!\cJ} \big\|_{\HS,\mbH} 
		\leq \frac{\hbeta_{\!\cJ}}{1-2\halpha_{\!\cJ}}.
		\label{spec:contraction2-res2}
	\end{gather}
\end{lemma}

\begin{proof}[Proof of Lemma~\ref{spec:comp2:aux3}]
	We first show \eqref{spec:contraction2-res1}.
	
	The first inequality in \eqref{spec:contraction2-res1} is immediate from $\|\cP_\cJ\|_{\op,\mbH}=1$ and
	\begin{equation*}
		|\cR|_{\!\cJ}^{-1/2} (\cI_\mbH - \hcP_{\!\cJ}) \cP_\cJ
		= |\cR|_{\!\cJ}^{-1/2} \cP_\cJ - |\cR|_{\!\cJ}^{-1/2} \hcP_{\!\cJ} \cP_\cJ
		= - |\cR|_{\!\cJ}^{-1/2} \hcP_{\!\cJ} \cP_\cJ.
	\end{equation*}
	
	To show the second inequality, using Lemma~\ref{spec:comp2:aux1}, we have
	\begin{align}
		\| |\cR|_{\!\cJ}^{-1/2} \hcP_{\!\cJ} \|_{\HS,\mbH}^2
		&= \sum_{k \notin \cJ} \sum_{\ell \in \cJ} \min_{m \in \cJ} |\lambda_k-\lambda_m| \cdot \|\cQ_k \hcQ_\ell\|_{\HS,\mbH}^2.
		\notag\\&= \sum_{k \notin \cJ} \sum_{\ell \in \cJ} \frac{\min_{m \in \cJ} |\lambda_k-\lambda_m|}{|\lambda_k - \hlambda_\ell|^2} \cdot \|\cQ_k (\hcH-\cH) \hcQ_\ell\|_{\HS,\mbH}^2.
		\label{spec:contraction2-eq1}
	\end{align}
	By Lemma~\ref{spec:comp1:aux2} and the definition of $\gammaJ$, under $\halpha_{\!\cJ} <1/2$, we have for any $k \notin \cJ$ and $\ell \in \cJ$,
	\begin{align*}
		|\lambda_\ell - \hlambda_\ell| 
		&\leq \gammaJ \halpha_{\!\cJ} 
		\leq \halpha_{\!\cJ} \min_{m \in \cJ} |\lambda_k - \lambda_m|  \quad\text{and}\quad
		|\lambda_k - \lambda_\ell| \ge \min_{m \in \cJ} |\lambda_k - \lambda_m|,
	\end{align*}
	which yields
	\begin{align}
		|\lambda_k - \hlambda_\ell| 
		&\ge |\lambda_k - \lambda_\ell| - |\lambda_\ell - \hlambda_\ell|
		\ge (1-\halpha_{\!\cJ}) \min_{m \in \cJ} |\lambda_k - \lambda_m|.
		\label{spec:contraction2-eq2}
	\end{align}
	Hence, applying \eqref{spec:contraction2-eq2} to \eqref{spec:contraction2-eq1} gives
	\begin{align*}
		\| |\cR|_{\!\cJ}^{-1/2} \hcP_{\!\cJ} \|^2_{\HS,\mbH}
		&\leq \frac{1}{(1-\halpha_{\!\cJ})^2}\sum_{k \notin \cJ} \sum_{\ell \in \cJ} \frac{1}{\min_{m \in \cJ} |\lambda_k-\lambda_m|} \cdot \|\cQ_k (\hcH-\cH) \hcQ_\ell\|_{\HS,\mbH}^2
		\\&= \frac{1}{(1-\halpha_{\!\cJ})^2} \big\| |\cR|_{\!\cJ}^{1/2} (\hcH-\cH) \hcP_{\!\cJ} \big\|_{\HS,\mbH}^2,
	\end{align*}
	which implies
	\begin{equation}
		\| |\cR|_{\!\cJ}^{-1/2} \hcP_{\!\cJ} \|_{\HS,\mbH}
		\leq \frac{1}{1-\halpha_{\!\cJ}} \big\| |\cR|_{\!\cJ}^{1/2} (\hcH-\cH) \hcP_{\!\cJ} \big\|_{\HS,\mbH}.
		\label{spec:contraction2-eq3}
	\end{equation}
	
	To control $\big\| |\cR|_{\!\cJ}^{1/2} (\hcH-\cH) \hcP_{\!\cJ} \big\|_{\HS,\mbH}$, using $\cI_\mbH = \cP_{\!\cJ} + |\cR|_{\!\cJ}^{-1/2}|\cR|_{\!\cJ}^{1/2}$, we bound
	\begin{align*}
		\big\| |\cR|_{\!\cJ}^{1/2} (\hcH-\cH) \hcP_{\!\cJ} \big\|_{\HS,\mbH}
		&\leq \big\| |\cR|_{\!\cJ}^{1/2} (\hcH-\cH) \cP_{\!\cJ} \hcP_{\!\cJ} \big\|_{\HS,\mbH}
		\\&\quad + \big\| |\cR|_{\!\cJ}^{1/2} (\hcH-\cH) |\cR|_{\!\cJ}^{1/2} |\cR|_{\!\cJ}^{-1/2} \hcP_{\!\cJ} \big\|_{\HS,\mbH}.
		\\&\leq \big\| |\cR|_{\!\cJ}^{1/2} (\hcH-\cH) \cP_{\!\cJ} \big\|_{\HS,\mbH}
		\\&\quad + \big\| |\cR|_{\!\cJ}^{1/2} (\hcH-\cH) |\cR|_{\!\cJ}^{1/2} \big\|_{\op,\mbH} \cdot \big\| |\cR|_{\!\cJ}^{-1/2} \hcP_{\!\cJ} \big\|_{\HS,\mbH}.
	\end{align*}
	Since $\| |\cR|_{\!\cJ}^{1/2} (\hcH-\cH) \cP_{\!\cJ} \|_{\HS,\mbH} = \sqrt{\gammaJ}\hbeta_{\!\cJ}$, and $\| |\cR|_{\!\cJ}^{1/2} (\hcH-\cH) |\cR|_{\!\cJ}^{1/2} \|_{\op,\mbH} \leq \halpha_{\!\cJ}$ by Lemma~\ref{spec:comp2:aux2},
	we have
	\begin{equation}
		\big\| |\cR|_{\!\cJ}^{1/2} (\hcH-\cH) \hcP_{\!\cJ} \big\|_{\HS,\mbH}
		\leq \sqrt{\gammaJ}\hbeta_{\!\cJ} + \halpha_{\!\cJ} \| |\cR|_{\!\cJ}^{-1/2} \hcP_{\!\cJ} \|_{\HS,\mbH}.
		\label{spec:contraction2-eq4}
	\end{equation}
	Combining \eqref{spec:contraction2-eq3} and \eqref{spec:contraction2-eq4} gives
	\begin{align*}
		\| |\cR|_{\!\cJ}^{-1/2} \hcP_{\!\cJ} \|_{\HS,\mbH} 
		&\leq \frac{\sqrt{\gammaJ}\hbeta_{\!\cJ}}{1-\halpha_{\!\cJ}} 
		+ \frac{\halpha_{\!\cJ}}{1-\halpha_{\!\cJ}} \big\| |\cR|_{\!\cJ}^{-1/2} \hcP_{\!\cJ} \big\|_{\HS,\mbH}.
	\end{align*}
	Solving this for $\big\| |\cR|_{\!\cJ}^{-1/2} \hcP_{\!\cJ} \big\|_{\HS,\mbH}$ yields \eqref{spec:contraction2-res1}. 
	
	To show \eqref{spec:contraction2-res2}, it suffices to bound $\|(\cI_\mbH-\cP_{\!\cJ}) \hcP_{\!\cJ}\|_{\op,\mbH}$ since $\cP_{\!\cJ}$ and $\hcP_{\!\cJ}$ are orthogonal projections.
	Applying Lemma~\ref{spec:comp2:aux1} and \eqref{spec:contraction2-eq2} gives
	\begin{align*}
		\| (\cI_\mbH-\cP_{\!\cJ}) \hcP_{\!\cJ} \|_{\HS,\mbH}^2 
		&= \sum_{k\notin\cJ} \sum_{\ell \in \cJ} \|\cQ_k \hcQ_\ell\|_{\HS,\mbH}^2
		= \sum_{k\notin\cJ} \sum_{\ell \in \cJ} \frac{1}{|\lambda_k-\hlambda_\ell|^2} \|\cQ_k (\hcH-\cH) \hcQ_\ell\|_{\HS,\mbH}^2
		\\&\leq \frac{1}{(1-\halpha_{\!\cJ})^2} \sum_{k\notin\cJ} \sum_{\ell \in \cJ} \frac{1}{\min_{m \in \cJ} |\lambda_k-\lambda_m|^2} \|\cQ_k (\hcH-\cH) \hcQ_\ell\|_{\HS,\mbH}^2
		\\&\leq \frac{1}{\gamma_\cJ(1-\halpha_{\!\cJ})^2 } \sum_{k\notin\cJ} \sum_{\ell \in \cJ} \frac{1}{\min_{m \in \cJ} |\lambda_k-\lambda_m|} \|\cQ_k (\hcH-\cH) \hcQ_\ell\|_{\HS,\mbH}^2
		\\&= \frac{1}{\gamma_\cJ(1-\halpha_{\!\cJ})^2} \||\cR|_{\!\cJ}^{1/2} (\hcH-\cH) \hcP_{\!\cJ}\|_{\HS,\mbH}^2,
	\end{align*}
	which implies
	\begin{equation}
		\| (\cI-\cP_{\!\cJ}) \hcP_{\!\cJ} \|_{\HS,\mbH} 
		\leq \frac{1}{\sqrt{\gamma_\cJ}(1-\halpha)} \big\| |\cR|_{\!\cJ}^{1/2} (\hcH-\cH) \hcP_{\!\cJ} \big\|_{\HS,\mbH}.
		\label{spec:contraction2-eq5}
	\end{equation}
	Hence, by applying \eqref{spec:contraction2-res1} and \eqref{spec:contraction2-eq4} to \eqref{spec:contraction2-eq5}, we have
	\begin{align*}
		\| (\cI-\cP_{\!\cJ}) \hcP_{\!\cJ} \|_{\HS,\mbH} 
		&\leq \frac{1}{\sqrt{\gamma_\cJ}(1-\halpha)} \bigg( \sqrt{\gammaJ}\hbeta_{\!\cJ} + \halpha_{\!\cJ} \| |\cR|_{\!\cJ}^{-1/2} \hcP_{\!\cJ} \|_{\HS,\mbH} \bigg)
		\\&\leq \frac{1}{\sqrt{\gamma_\cJ}(1-\halpha)} \bigg( \sqrt{\gammaJ}\hbeta_{\!\cJ} + \halpha_{\!\cJ} \cdot \frac{\sqrt{\gammaJ} \hbeta_{\!\cJ}}{1-2\halpha_{\!\cJ}}  \bigg)
		= \frac{\hbeta_{\!\cJ}}{1-2\halpha},
	\end{align*}
	which implies \eqref{spec:contraction2-res2}.
	This completes the proof.
\end{proof}

\begin{lemma} \label{spec:comp2:aux4}
	Suppose Assumptions~\refsupp{spec:assump1} and \refsupp{spec:assump2} hold.
	Under $\halpha_{\!\cJ} <1/2$, we have
	\begin{gather*}
		\big\| \cP_{\!\cJ} (\cI_\mbH-\hcP_{\!\cJ})  \big\|_{\op,\mbH} 
		= \big\| (\cI_\mbH-\cP_{\!\cJ}) \hcP_{\!\cJ} \big\|_{\op,\mbH} 
		\leq \frac{(2-\halpha_{\!\cJ})\sqrt{D}\halpha_{\!\cJ}}{(1-\halpha_{\!\cJ})^2}
	\end{gather*}
\end{lemma}

\begin{proof}[Proof of Lemma~\ref{spec:comp2:aux4}]
	It suffices to bound $\|\cP_{\!\cJ} (\cI_\mbH- \hcP_{\!\cJ})\|_{\op,\mbH}$ since $\cP_{\!\cJ}$ and $\hcP_{\!\cJ}$ are orthogonal projections.
	To this end, we derive some inequalities.
	
	For each $d = 1,\dots,D$, define
	\begin{equation*}
		\cR_{d,\cJ} = \sum_{k \notin \cJ} (\lambda_k - \theta_d)^{-1} \cQ_k
		\quad\text{and}\quad
		\hcR_{d,\cJ} = \sum_{k \notin \cJ} (\hlambda_k - \theta_d)^{-1} \hcQ_k.
	\end{equation*}
	It can be easily shown that
	\begin{equation}
		\| |\cR|_{\!\cJ}^{-1/2} \cR_{d,\cJ} \|_{\op,\mbH} \leq \gammaJ^{-1/2}, \quad
		\| |\cR|_{\!\cJ}^{-1/2} \cR_{d,\cJ} |\cR|_{\!\cJ}^{-1/2} \|_{\op,\mbH} \leq 1, \quad
		\|\hcR_{d,\cJ}\|_{\op,\mbH} \leq \frac{1}{\gammaJ(1-\halpha_{\!\cJ})}.
		\label{spec:comp2:aux4-eq1}
	\end{equation}
	In particular, the last inequality can be shown using \eqref{spec:contraction2-eq2}.
	
	Next, we show
	\begin{equation}
		\| |\cR|_{\cJ}^{-1/2} \hcR_{d,\cJ} \|_{\op,\mbH} \leq \frac{1}{\sqrt{\gammaJ}(1-\halpha_{\!\cJ})^2}.
		\label{spec:comp2:aux4-eq2}
	\end{equation}	
	Using Lemma~\ref{spec:comp2:aux1}, we have
	\begin{align*}
		\cR_{d,\cJ} (\hcH-\cH) \hcR_{d,\cJ} 
		&= \sum_{k\notin\cJ} \sum_{\ell\notin\cJ} \frac{\cQ_k (\hcH-\cH) \hcQ_\ell}{(\lambda_k-\theta_d)(\hlambda_\ell-\theta_d)}
		= \sum_{k\notin\cJ} \sum_{\ell\notin\cJ} \frac{(\hlambda_\ell-\lambda_k)\cQ_k \hcQ_\ell}{(\lambda_k-\theta_d)(\hlambda_\ell-\theta_d)}
		\\&= \sum_{k\notin\cJ} \sum_{\ell\notin\cJ} \frac{\cQ_k \hcQ_\ell}{\lambda_k-\theta_d}
		- \sum_{k\notin\cJ} \sum_{\ell\notin\cJ} \frac{\cQ_k \hcQ_\ell}{\hlambda_\ell-\theta_d}
		\\&= (\cI_\mbH-\cP_{\!\cJ}) \hcR_{d,\cJ} - \cR_{d,\cJ} (\cI_\mbH-\hcP_{\!\cJ}).
	\end{align*}
	Using this and $\cI_\mbH = \cP_{\!\cJ} + |\cR|_{\!\cJ}^{1/2} |\cR|_{\!\cJ}^{-1/2} = \cP_{\!\cJ} + |\cR|_{\!\cJ}^{-1/2} |\cR|_{\!\cJ}^{1/2}$,
	we have the decomposition
	\begin{align*}
		|\cR|_{\cJ}^{-1/2} \hcR_{d,\cJ}
		&= |\cR|_{\cJ}^{-1/2} (\cI_\mbH - \cP_{\!\cJ}) \hcR_{d,\cJ}
		= |\cR|_{\cJ}^{-1/2} \cR_{d,\cJ} (\cI_\mbH-\hcP_{\!\cJ}) + |\cR|_{\cJ}^{-1/2} \cR_{d,\cJ} (\hcH-\cH) \hcR_{d,\cJ}
		\\&= |\cR|_{\cJ}^{-1/2} \cR_{d,\cJ} (\cI_\mbH-\hcP_{\!\cJ}) 
		+ |\cR|_{\cJ}^{-1/2} \cR_{d,\cJ} |\cR|_{\!\cJ}^{-1/2} |\cR|_{\!\cJ}^{1/2} (\hcH-\cH) \cP_{\!\cJ} \hcR_{d,\cJ}
		\\&\quad+ |\cR|_{\cJ}^{-1/2} \cR_{d,\cJ} |\cR|_{\!\cJ}^{-1/2} |\cR|_{\!\cJ}^{1/2} (\hcH-\cH) |\cR|_{\!\cJ}^{1/2} |\cR|_{\!\cJ}^{-1/2} \hcR_{d,\cJ}.
	\end{align*}
	Using \eqref{spec:comp2:aux4-eq1} and Lemma~\ref{spec:comp2:aux2}, we bound the OP norms of the above terms as follows:
	\begin{align*}
		\| |\cR|_{\cJ}^{-1/2} \cR_{d,\cJ} (\cI_\mbH-\hcP_{\!\cJ}) \|_{\op,\mbH}
		\leq \| |\cR|_{\cJ}^{-1/2} \cR_{d,\cJ} \|_{\op,\mbH} \leq \gammaJ^{-1/2},
	\end{align*}
	\begin{align*}
		&\quad \| |\cR|_{\cJ}^{-1/2} \cR_{d,\cJ} |\cR|_{\!\cJ}^{-1/2} |\cR|_{\!\cJ}^{1/2} (\hcH-\cH) \cP_{\!\cJ} \hcR_{d,\cJ} \|_{\op,\mbH}
		\\&\leq \| |\cR|_{\cJ}^{-1/2} \cR_{d,\cJ} |\cR|_{\!\cJ}^{-1/2} \|_{\op,\mbH} \cdot
		\| |\cR|_{\!\cJ}^{1/2} (\hcH-\cH) \cP_{\!\cJ} \|_{\op,\mbH} \cdot \| \hcR_{d,\cJ} \|_{\op,\mbH}
		\\&\leq \sqrt{\gammaJ} \halpha_{\!\cJ} \cdot \frac{1}{\gammaJ(1-\halpha_{\!\cJ})}
		= \frac{\halpha_{\!\cJ}}{\sqrt{\gammaJ}(1-\halpha_{\!\cJ})},
	\end{align*}
	\begin{align*}
		&\quad \| |\cR|_{\cJ}^{-1/2} \cR_{d,\cJ} |\cR|_{\!\cJ}^{-1/2} |\cR|_{\!\cJ}^{1/2} (\hcH-\cH) |\cR|_{\!\cJ}^{1/2} |\cR|_{\!\cJ}^{-1/2} \hcR_{d,\cJ} \|_{\op,\mbH}
		\\&\leq \| |\cR|_{\cJ}^{-1/2} \cR_{d,\cJ} |\cR|_{\cJ}^{-1/2} \|_{\op,\mbH} \cdot \| |\cR|_{\cJ}^{1/2} (\hcH-\cH) |\cR|_{\!\cJ}^{1/2} \|_{\op,\mbH} \cdot \| |\cR|_{\!\cJ}^{-1/2} \hcR_{d,\cJ} \|_{\op,\mbH}
		\\&\leq \halpha_{\!\cJ} \||\cR|_{\!\cJ}^{-1/2} \hcR_{d,\cJ}\|_{\op,\mbH}.
	\end{align*}
	Combining these bounds yields
	\begin{align*}
		\| |\cR|_{\cJ}^{-1/2} \hcR_{d,\cJ} \|_{\op,\mbH}
		&\leq \frac{1}{\sqrt{\gammaJ}} + \frac{\halpha_{\!\cJ}}{\sqrt{\gammaJ}(1-\halpha_{\!\cJ})} + \halpha_{\!\cJ} \||\cR|_{\!\cJ}^{-1/2} \hcR_{d,\cJ}\|_{\op,\mbH}
		\\&= \frac{1}{\sqrt{\gammaJ}(1-\halpha_{\!\cJ})} + \halpha_{\!\cJ} \||\cR|_{\!\cJ}^{-1/2} \hcR_{d,\cJ}\|_{\op,\mbH}.
	\end{align*}
	Solving this for $\| |\cR|_{\cJ}^{-1/2} \hcR_{d,\cJ} \|_{\op,\mbH}$ yields \eqref{spec:comp2:aux4-eq2}.
	
	Now, we bound $\|\cP_{\!\cJ} (\cI_\mbH- \hcP_{\!\cJ})\|_{\op,\mbH}$.
	Using Lemma~\ref{spec:comp2:aux1} and $\lambda_k = \theta_d$ for $k \in \cJ_d$, we have
	\begin{align*}
		\cP_{\!\cJ} (\cI_\mbH- \hcP_{\!\cJ})
		= \sum_{k \in \cJ} \sum_{\ell \notin \cJ} \frac{\cQ_k(\hcH-\cH)\hcQ_\ell}{\hlambda_\ell - \lambda_k}
		= \sum_{d=1}^D \sum_{k \in \cJ_d} \sum_{\ell \notin \cJ} \frac{\cQ_k(\hcH-\cH)\hcQ_\ell}{\hlambda_\ell - \theta_d}
		= \sum_{d=1}^D \cP_{\!\cJ_d} (\hcH-\cH) \hcR_{d,\cJ}.
	\end{align*}
	Using \eqref{spec:comp2:aux4-eq1}, \eqref{spec:comp2:aux4-eq2}, Lemma~\ref{spec:comp2:aux2}, $\cI_\mbH= \cP_{\!\cJ} + |\cR|_{\!\cJ}^{1/2} |\cR|_{\!\cJ}^{-1/2}$, $\cP_{\!\cJ}= \cP_{\!\cJ_d} \cP_{\!\cJ} = \cP_{\!\cJ}\cP_{\!\cJ_d}$ and $\|\cP_{\!\cJ_d}\|_{\op,\mbH}=1$, we bound
	\begin{align*}
		\|\cP_{\!\cJ_d} (\hcH-\cH) \hcR_{d,\cJ}\|_{\op,\mbH}
		&\leq \| \cP_{\!\cJ} (\hcH-\cH) \cP_{\!\cJ} \|_{\op,\mbH} \cdot \| \hcR_{d,\cJ}\|_{\op,\mbH}
		\\&\quad+ \|\cP_{\!\cJ} (\hcH-\cH) |\cR|_{\!\cJ}^{1/2} \|_{\op,\mbH} \cdot \| |\cR|_{\!\cJ}^{-1/2} \hcR_{d,\cJ} \|_{\op,\mbH}
		\\&\leq \gammaJ \halpha_{\!\cJ} \cdot \frac{1}{\gammaJ(1-\halpha_{\!\cJ})} + \sqrt{\gammaJ} \halpha_{\!\cJ} \cdot \frac{1}{\sqrt{\gammaJ}(1-\halpha_{\!\cJ})^2}
		= \frac{(2-\halpha_{\!\cJ})\halpha_{\!\cJ}}{(1-\halpha_{\!\cJ})^2}.
	\end{align*}
	Since $\cP_{\cJ_1},\ldots,\cP_{\!\cJ_d}$ are mutually orthogonal projections, we obtain	
	\begin{align*}
		\|\cP_{\!\cJ} (\cI_\mbH- \hcP_{\!\cJ})\|_{\op,\mbH}
		\leq \sqrt{\sum_{d=1}^D \|\cP_{\!\cJ_d} (\hcH-\cH) \hcR_{d,\cJ}\|^2_{\op,\mbH}}
		\leq \frac{(2-\halpha_{\!\cJ})\sqrt{D}\halpha_{\!\cJ}}{(1-\halpha_{\!\cJ})^2}.
	\end{align*}
	This completes the proof.
\end{proof}

With these lemmas in hand, we proceed to the proof of Proposition~\refsupp{spec:comp2}.
\begin{proof}[Proof of Proposition~\refsupp{spec:comp2}]
	We divide the proof into several steps
	
	\mycase{1. Decomposition}
	Recall \eqref{spec:comp1-eq1} that
	\begin{equation*}
		f_\cJ (\hcH) - f_\cJ (\cH) - \nabla_\cH f_\cJ  (\hcH)
		= \frac{1}{2 \pi i} \oint_{\partial\Gamma_{\!\cJ}} f(z) \fR_z (\hcH-\cH) \hfR_z (\hcH-\cH) \fR_z\,dz.
	\end{equation*}
	By letting $f \equiv 1$ and using the definitions of $\fR_z$ and $\hfR_z$, the right-hand side of the above equation can be expressed as
	\begin{align*}
		\sum_{k=1}^\infty \sum_{\ell=1}^\infty \sum_{m=1}^\infty
		\Bigg(\frac{1}{2 \pi i} \oint_{\partial\Gamma_{\!\cJ}} \frac{1}{(z-\lambda_k)(z-\lambda_\ell)(z-\hlambda_m)} dz \Bigg)
		\cQ_k (\hcH-\cH) \hcQ_m (\hcH-\cH) \cQ_\ell.
	\end{align*}
	By applying Cauchy's integral formula (see Lemma~\ref{sup:cauchy2}), we obtain the decomposition
	\begin{align}
		&\quad f_\cJ (\hcH) - f_\cJ (\cH) - \nabla_\cH f_\cJ  (\hcH)
		= -\sum_{k \in \cJ} \sum_{\ell \in \cJ} \sum_{m \notin \cJ} \frac{\cQ_k (\hcH-\cH) \hcQ_m (\hcH-\cH) \cQ_\ell}
		{(\hlambda_m-\lambda_k)(\hlambda_m-\lambda_\ell)} 
		\notag\\& -\sum_{k \in \cJ} \sum_{\ell \notin \cJ} \sum_{m \in \cJ} \frac{\cQ_k (\hcH-\cH) \hcQ_m (\hcH-\cH) \cQ_\ell}
		{(\lambda_\ell-\lambda_k)(\lambda_\ell-\hlambda_m)} 
		-\sum_{k \notin \cJ} \sum_{\ell \in \cJ} \sum_{m \in \cJ} \frac{\cQ_k (\hcH-\cH) \hcQ_m (\hcH-\cH) \cQ_\ell}{(\lambda_\ell - \lambda_k)(\hlambda_m - \lambda_k)}  
		\notag\\&+ \sum_{k \in \cJ} \sum_{\ell \notin \cJ} \sum_{m \notin \cJ} \frac{\cQ_k (\hcH-\cH) \hcQ_m (\hcH-\cH) \cQ_\ell}
		{(\lambda_k - \lambda_\ell)(\lambda_k - \hlambda_m)}  
		+ \sum_{k \notin \cJ} \sum_{\ell \in \cJ} \sum_{m \notin \cJ} \frac{\cQ_k (\hcH-\cH) \hcQ_m (\hcH-\cH) \cQ_\ell}
		{(\lambda_\ell - \lambda_k)(\lambda_\ell - \hlambda_m)}  
		\notag\\&+\sum_{k \notin \cJ} \sum_{\ell \notin \cJ} \sum_{m \in \cJ} \frac{\cQ_k (\hcH-\cH) \hcQ_m (\hcH-\cH) \cQ_\ell}
		{(\hlambda_m - \lambda_k)(\hlambda_m - \lambda_\ell)}  
		=: \cE_1 + \cE_2 + \cE_3 + \cE_4 + \cE_5 + \cE_6.
		\label{spec:comp2-eq1}
	\end{align}
	From this, we also see that
	\begin{equation}
		[f_\cJ (\hcH) - f_\cJ (\cH) - \nabla_\cH f_\cJ  (\hcH)] \cP_{\!\cJ} 
		= \cE_1 + \cE_3 + \cE_5.
	\end{equation}
	Hence, to complete the proof, we bound the Hilbert--Schmidt norms of $\cE_1,\dots,\cE_6$.
	
	\mycase{2. Bounding $\|\cE_1\|_{\HS,\mbH}$}
	Using Lemma~\ref{spec:comp2:aux1}, we express $\cE_1$ as
	\begin{align*}
		\cE_1 
		= - \sum_{k \in \cJ} \sum_{\ell \in \cJ} \sum_{m \notin \cJ} \frac{\cQ_k (\hcH-\cH) \hcQ_m (\hcH-\cH) \cQ_\ell}{(\hlambda_m-\lambda_k)(\hlambda_m-\lambda_\ell)}
		= - \sum_{k \in \cJ} \sum_{\ell \in \cJ} \sum_{m \notin \cJ} \cQ_k \hcQ_m \cQ_\ell
		= -\cP_{\!\cJ}(\cI_\mbH - \hcP_{\!\cJ}) \cP_{\!\cJ}.
	\end{align*}
	Noting $\cI_\mbH-\hcP_{\!\cJ}=(\cI_\mbH-\hcP_{\!\cJ})^2$, applying Lemmas~\ref{spec:comp2:aux3} and \ref{spec:comp2:aux4} yields
	\begin{equation}
		\|\cE_1\|_{\HS,\mbH} 
		\leq \|\cP_{\!\cJ}(\cI_\mbH-\hcP_{\!\cJ})\|_{\op,\mbH} \cdot \|(\cI_\mbH-\hcP_{\!\cJ})\cP_{\!\cJ}\|_{\HS,\mbH}
		\leq \frac{(2-\halpha_{\!\cJ})\sqrt{D}\halpha_{\!\cJ} \hbeta_{\!\cJ}}{(1-\halpha_{\!\cJ})^2(1-2\halpha_{\!\cJ})}.
		\label{spec:comp2-eq2}
	\end{equation}
	
	\mycase{3. Bounding $\|\cE_2\|_{\HS,\mbH}$ and $\|\cE_3\|_{\HS,\mbH}$}
	By symmetry, it suffices to bound $\|\cE_2\|_{\HS,\mbH}$.
	
	Using Lemma~\ref{spec:comp2:aux1}, we express $\cE_2$ as
	\begin{align*}
		\cE_2 &= -\sum_{k \in \cJ} \sum_{\ell \notin \cJ} \sum_{m \in \cJ} \frac{\cQ_k (\hcH-\cH) \hcQ_m (\hcH-\cH) \cQ_\ell}
		{(\lambda_\ell-\lambda_k)(\lambda_\ell-\hlambda_m)}
		\\&= \sum_{k \in \cJ} \sum_{\ell \notin \cJ} \sum_{m \in \cJ} \frac{\cQ_k (\hcH-\cH) \hcQ_m \cQ_\ell}
		{\lambda_\ell-\lambda_k}
		= \sum_{k \in \cJ} \sum_{\ell \notin \cJ} \frac{\cQ_k (\hcH-\cH) \hcP_{\!\cJ} \cQ_\ell}{\lambda_\ell-\lambda_k},
	\end{align*}
	and this yields
	\begin{align*}
		\|\cE_2\|_{\HS,\mbH} 
		&= \sqrt{\sum_{k \in \cJ} \sum_{\ell \notin \cJ} \frac{\|\cQ_k (\hcH-\cH) \hcP_{\!\cJ} \cQ_\ell\|_{\HS,\mbH}^2}{|\lambda_\ell-\lambda_k|^2}}
		\leq \sqrt{\sum_{k \in \cJ} \sum_{\ell \notin \cJ} \frac{\|\cQ_k (\hcH-\cH) \hcP_{\!\cJ} \cQ_\ell\|_{\HS,\mbH}^2}{\gammaJ \min_{m \in \cJ} |\lambda_\ell-\lambda_m|}}
		\\&= \gammaJ^{-1/2} \|\cP_{\!\cJ} (\hcH-\cH) \hcP_{\!\cJ} |\cR|_{\!\cJ}^{1/2} \|_{\HS,\mbH}.
	\end{align*}
	Using $\cI_\mbH = \cP_{\!\cJ}+|\cR|_{\!\cJ}^{1/2} |\cR|_{\!\cJ}^{-1/2}$ and Lemmas~\ref{spec:comp2:aux2} and \ref{spec:comp2:aux3}, we separate and bound
	\begin{align*}
		\|\cP_{\!\cJ} (\hcH-\cH) \hcP_{\!\cJ} |\cR|_{\!\cJ}^{1/2} \|_{\HS,\mbH}
		&\leq \|\cP_{\!\cJ} (\hcH-\cH) \cP_{\!\cJ} \|_{\op,\mbH} \cdot \|\hcP_{\!\cJ} |\cR|_{\!\cJ}^{-1/2} \|_{\HS,\mbH} \cdot \||\cR|_{\!\cJ}^{1/2}\|_{\op,\mbH}^2
		\\&\quad + \|\cP_{\!\cJ} (\hcH-\cH) |\cR|_{\!\cJ}^{1/2} \|_{\op,\mbH} \cdot \||\cR|_{\!\cJ}^{-1/2} \hcP_{\!\cJ}\|_{\HS,\mbH} \cdot \| \hcP_{\!\cJ} |\cR|_{\!\cJ}^{1/2} \|_{\op,\mbH}
		\\&\leq \gammaJ \halpha_{\!\cJ} \cdot \frac{\sqrt{\gammaJ} \hbeta_{\!\cJ}}{1-2\halpha_{\!\cJ}} \cdot \gammaJ^{-1}
		+ \gammaJ^{1/2}\halpha_{\!\cJ} \cdot \frac{\sqrt{\gammaJ} \hbeta_{\!\cJ}}{1-2\halpha_{\!\cJ}} \cdot \gammaJ^{-1/2}
		= \frac{2\sqrt{\gammaJ} \halpha_{\!\cJ} \hbeta_{\!\cJ}}{1-2\halpha_{\!\cJ}}.
	\end{align*}
	Therefore, we obtain
	\begin{equation}
		\|\cE_2\|_{\HS,\mbH} = \|\cE_3\|_{\HS,\mbH} \leq \frac{2 \halpha_{\!\cJ} \hbeta_{\!\cJ}}{1-2\halpha_{\!\cJ}}.
		\label{spec:comp2-eq3}
	\end{equation}
	
	\mycase{4. Bounding $\|\cE_4\|_{\HS,\mbH}$ and $\|\cE_5\|_{\HS,\mbH}$}
	By symmetry, It suffices to bound $\|\cE_4\|_{\HS,\mbH}$.
	
	Using Lemma~\ref{spec:comp2:aux1}, we express $\cE_4$ as
	\begin{align*}
		\cE_4 &= \sum_{k \in \cJ} \sum_{\ell \notin \cJ} \sum_{m \notin \cJ} 
		\frac{\cQ_k (\hcH-\cH) \hcQ_m (\hcH-\cH) \cQ_\ell}{(\lambda_k - \lambda_\ell)(\lambda_k - \hlambda_m)}
		\\&= \sum_{k \in \cJ} \sum_{\ell \notin \cJ} \sum_{m \notin \cJ} 
		\frac{\cQ_k \hcQ_m (\hcH-\cH) \cQ_\ell}{\lambda_k - \lambda_\ell}
		= \sum_{k \in \cJ} \sum_{\ell \notin \cJ} 
		\frac{\cQ_k (\cI_\mbH-\hcP_{\!\cJ}) (\hcH-\cH) \cQ_\ell}{\lambda_k - \lambda_\ell},
	\end{align*}
	and this yields
	\begin{align*}
		\|\cE_4\|_{\HS,\mbH} 
		&= \sqrt{\sum_{k \in \cJ} \sum_{\ell \notin \cJ} \frac{\|\cQ_k (\cI_\mbH-\hcP_{\!\cJ}) (\hcH-\cH) \cQ_\ell\|_{\HS,\mbH}^2}{|\lambda_\ell-\lambda_k|^2}}
		\\&\leq \sqrt{\sum_{k \in \cJ} \sum_{\ell \notin \cJ} \frac{\|\cQ_k (\cI_\mbH-\hcP_{\!\cJ}) (\hcH-\cH) \cQ_\ell\|_{\HS,\mbH}^2}{\gammaJ \min_{m \in \cJ} |\lambda_\ell-\lambda_m|}}
		\\&= \gammaJ^{-1/2} \| \cP_{\!\cJ} (\cI_\mbH-\hcP_{\!\cJ})(\hcH-\cH)|\cR|_{\!\cJ}^{1/2} \|_{\HS,\mbH}.
	\end{align*}
	Using $\cI_\mbH = \cP_{\!\cJ}+|\cR|_{\!\cJ}^{-1/2} |\cR|_{\!\cJ}^{1/2}$
	and Lemmas~\ref{spec:comp2:aux2} and \ref{spec:comp2:aux3}, we separate and bound
	\begin{align*}
		\| \cP_{\!\cJ} (\cI_\mbH-\hcP_{\!\cJ})(\hcH-\cH)|\cR|_{\!\cJ}^{1/2} \|_{\HS,\mbH}
		&\leq \|\cP_{\!\cJ} (\cI_\mbH-\hcP_{\!\cJ}) \cP_{\!\cJ}\|_{\HS,\mbH} \cdot \| \cP_{\!\cJ} (\hcH-\cH)|\cR|_{\!\cJ}^{1/2} \|_{\op,\mbH}
		\\&\quad + \|\cP_{\!\cJ} (\cI_\mbH-\hcP_{\!\cJ}) |\cR|_{\!\cJ}^{-1/2} \|_{\HS,\mbH} \cdot \||\cR|_{\!\cJ}^{1/2} (\hcH-\cH)|\cR|_{\!\cJ}^{1/2} \|_{\op,\mbH}
		\\&\leq \frac{\hbeta_{\!\cJ}}{1-2\halpha_{\!\cJ}} \cdot \sqrt{\gammaJ} \halpha_{\!\cJ} + \frac{\sqrt{\gammaJ} \hbeta_{\!\cJ}}{1-2\halpha_{\!\cJ}} \cdot \halpha_{\!\cJ}
		= \frac{2\sqrt{\gammaJ} \halpha_{\!\cJ}\hbeta_{\!\cJ}}{1-2\halpha_{\!\cJ}}.
	\end{align*}
	Thus, we obtain
	\begin{equation}
		\|\cE_4\|_{\HS,\mbH} = \|\cE_5\|_{\HS,\mbH} \leq \frac{2\halpha_{\!\cJ}\hbeta_{\!\cJ}}{1-2\halpha_{\!\cJ}}
		\label{spec:comp2-eq4}
	\end{equation}
	
	\mycase{5. Bounding $\|\cE_6\|_{\HS,\mbH}$}
	Using Lemma~\ref{spec:comp2:aux1}, we express $\cE_6$ as
	\begin{align*}
		\cE_6 &= \sum_{k \notin \cJ} \sum_{\ell \notin \cJ} \sum_{m \in \cJ} \frac{\cQ_k (\hcH-\cH) \hcQ_m (\hcH-\cH) \cQ_\ell}
		{(\hlambda_m - \lambda_k)(\hlambda_m - \lambda_\ell)}
		\\&= \sum_{k \notin \cJ} \sum_{\ell \notin \cJ} \sum_{m \in \cJ} \cQ_k \hcQ_m \cQ_\ell
		= (\cI_\mbH- \cP_{\!\cJ}) \hcP_{\!\cJ} (\cI_\mbH- \cP_{\!\cJ}).
	\end{align*}
	Noting $\hcP_{\!\cJ}=\hcP_{\!\cJ}^2$, applying Lemmas~\ref{spec:comp2:aux3} and \ref{spec:comp2:aux4} yields
	\begin{equation}
		\|\cE_6\|_{\HS,\mbH} \leq \|(\cI_\mbH-\cP_{\!\cJ})\hcP_{\!\cJ} \|_{\op,\mbH} \cdot \|\hcP_{\!\cJ}(\cI_\mbH-\cP_{\!\cJ})\|_{\HS,\mbH}
		\leq \frac{(2-\halpha_{\!\cJ})\sqrt{D}\halpha_{\!\cJ} \hbeta_{\!\cJ}}{(1-\halpha_{\!\cJ})^2(1-2\halpha_{\!\cJ})}.
		\label{spec:comp2-eq5}
	\end{equation}
	
	\mycase{5. Concluding the proof}
	Combining \eqref{spec:comp2-eq1}-\eqref{spec:comp2-eq5}, we obtain
	\begin{align*}
		\| [f_\cJ (\hcH) - f_\cJ (\cH) - \nabla_\cH f_\cJ  (\hcH)] \cP_{\!\cJ} \|_{\HS,\mbH}
		&\leq \|\cE_1\|_{\HS,\mbH} + \|\cE_3\|_{\HS,\mbH} + \|\cE_5\|_{\HS,\mbH}
		\\&\leq \bigg( \frac{(2-\halpha_{\!\cJ})\sqrt{D}}{(1-\halpha_{\!\cJ})^2}+ 4\bigg) \frac{\halpha_{\!\cJ} \hbeta_{\!\cJ}}{1-2\halpha_{\!\cJ}},
		\\
		\| f_\cJ (\hcH) - f_\cJ (\cH) - \nabla_\cH f_\cJ  (\hcH) \|_{\HS,\mbH}
		&\leq \sum_{k=1}^6 \|\cE_k\|_{\HS,\mbH}
		\leq \bigg( \frac{2(2-\halpha_{\!\cJ})\sqrt{D}}{(1-\halpha_{\!\cJ})^2}+ 8\bigg) \frac{\halpha_{\!\cJ} \hbeta_{\!\cJ}}{1-2\halpha_{\!\cJ}}.
	\end{align*}
	Hence, the proof is complete by considering the condition $\halpha_{\!\cJ} \leq 1/4$.
\end{proof}

\subsection{Proof of Proposition~\refsupp{spec:comp3}}

\begin{proof}[Proof of \eqrefsupp{spec:comp3-res1} in Proposition~\refsupp{spec:comp3}]
	We split the proof in several steps.

	\mycase{1. Decomposition}
	Recall \eqref{spec:comp1-eq1} that
	\begin{gather*}
		f_\cJ (\hcH) - f_\cJ (\cH) - \nabla_\cH f_\cJ  (\hcH)
		= \frac{1}{2 \pi i} \oint_{\partial\Gamma_{\!\cJ}} f(z) \fR_z (\hcH-\cH) \hfR_z (\hcH-\cH) \fR_z\,dz.
	\end{gather*}
	From $\cP_{\cJ_{d}} \fR_z = \fR_z \cP_{\cJ_{d}} = \sum_{k \in \cJ}(z-\lambda_k)^{-1} \cQ_k$ and $\hfR_z = \sum_{m \in \N} (z-\hlambda_m)^{-1} \hcQ_m$,
	it follows that
	\begin{align*}
		\cP_{\!\cJ} \fR_z (\hcH-\cH) \hfR_z (\hcH-\cH) \fR_z \cP_{\!\cJ}
		&= \sum_{k \in \cJ}\sum_{\ell \in\cJ} \sum_{m=1}^\infty\frac{\cQ_k (\hcH-\cH) \hcQ_m (\hcH-\cH) \cQ_\ell}{(z-\lambda_k)(z-\lambda_\ell)(z-\hlambda_m)}.
	\end{align*}
	Hence, for $f(z) = \xi_1 z + \xi_2$, we have
	\begin{multline*}
		\cP_{\!\cJ} [ f_\cJ (\hcH) - f_\cJ (\cH) - \nabla_\cH f_\cJ  (\hcH) ] \cP_{\!\cJ}
		\\= \sum_{k \in \cJ} \sum_{\ell\in\cJ} \sum_{m=1}^\infty \bigg( \frac{1}{2 \pi i } \oint_{\partial\Gamma_{\!\cJ}} \frac{\xi_1 z + \xi_2}{(z-\lambda_k)(z-\lambda_\ell)(z-\hlambda_m)} dz \bigg) \cQ_k (\hcH-\cH) \hcQ_m (\hcH-\cH) \cQ_\ell.
	\end{multline*}
	By applying Cauchy's integral formula (see Lemma~\ref{sup:cauchy3}), we obtain
	\begin{equation}
		\cP_{\!\cJ} [ f_\cJ (\hcH) - f_\cJ (\cH) - \nabla_\cH f_\cJ  (\hcH) ] \cP_{\!\cJ}
		= -(\cE_1 + \cE_2 + \cE_3),
		\label{spec:comp3-eq1}
	\end{equation}
	where
	\begin{align*}
		\cE_1 &= \frac{\xi_1}{2} \sum_{k \in \cJ} \sum_{\ell \in \cJ} \sum_{m \notin \cJ} \frac{\cQ_k (\hcH-\cH)\hcQ_m (\hcH-\cH)\cQ_\ell}{\hlambda_m-\lambda_k},
		\\ \cE_2 &= \frac{\xi_1}{2} \sum_{k \in \cJ} \sum_{\ell \in \cJ} \sum_{m \notin \cJ} \frac{\cQ_k (\hcH-\cH)\hcQ_m (\hcH-\cH)\cQ_\ell}{\hlambda_m-\lambda_\ell},
		\\ \cE_3 &= \sum_{k \in \cJ} \sum_{\ell \in \cJ} \sum_{m \notin \cJ} 
		\frac{(\xi_1 (\lambda_k+\lambda_\ell) + 2\xi_2)\cQ_k (\hcH-\cH)\hcQ_m (\hcH-\cH)\cQ_\ell}{2(\hlambda_m-\lambda_k)(\hlambda_m-\lambda_\ell)}.
	\end{align*}
	Below, we bound the HS norms of the terms $\cE_k$.
	
	\mycase{2. Bounding $\|\cE_1\|_{\HS,\mbH}$ and $\|\cE_2\|_{\HS,\mbH}$}
	By symmetry, it suffices to consider $\cE_1$. 
	
	Using Lemma~\ref{spec:comp2:aux1}, $\cP_{\!\cJ}= \sum_{k \in \cJ} \cQ_k$, and $\cI_\mbH-\hcP_{\!\cJ}= \sum_{m \notin \cJ} \hcQ_m$,
	we express $\cE_1$ as
	\begin{align*}
		\cE_1 &= \frac{\xi_1}{2} \sum_{k \in \cJ} \sum_{\ell \in\cJ} \sum_{m \notin \cJ} \cQ_k \hcQ_m (\hcH-\cH) \cQ_\ell
		= \frac{\xi_1}{2} \cP_{\!\cJ}(\cI_\mbH - \hcP_{\!\cJ})(\hcH-\cH)\cP_{\!\cJ}.
	\end{align*}
	Using $\cI_\mbH = \cP_{\!\cJ} \cP_{\!\cJ} + |\cR|_{\!\cJ}^{-1/2} |\cR|_{\!\cJ}^{1/2}$, we separate and bound
	\begin{align*}
		\| \cP_{\!\cJ}(\cI_\mbH - \hcP_{\!\cJ})(\hcH-\cH)\cP_{\!\cJ} \|_{\HS,\mbH}
		&\leq \|\cP_{\!\cJ}(\cI_\mbH - \hcP_{\!\cJ}) \cP_{\!\cJ} \cP_{\!\cJ} (\hcH-\cH)\cP_{\!\cJ}\|_{\HS,\mbH}
		\\&\quad + \| \cP_{\!\cJ}(\cI_\mbH - \hcP_{\!\cJ}) |\cR|_{\!\cJ}^{-1/2}|\cR|_{\!\cJ}^{1/2} (\hcH-\cH)\cP_{\!\cJ} \|_{\HS,\mbH}
		\\&\leq \|\cP_{\!\cJ}(\cI_\mbH - \hcP_{\!\cJ}) \cP_{\!\cJ}\|_{\HS,\mbH} \|\cP_{\!\cJ} (\hcH-\cH)\cP_{\!\cJ}\|_{\op,\mbH}
		\\&\quad + \| \cP_{\!\cJ}(\cI_\mbH - \hcP_{\!\cJ}) |\cR|_{\!\cJ}^{-1/2} \|_{\HS,\mbH} \| |\cR|_{\!\cJ}^{1/2} (\hcH-\cH)\cP_{\!\cJ} \|_{\op,\mbH}.
	\end{align*}
	Since $\|\cP_{\!\cJ}(\cI_\mbH - \hcP_{\!\cJ}) \cP_{\!\cJ}\|_{\HS,\mbH} \leq \min\{ \|\cP_{\!\cJ}(\cI_\mbH - \hcP_{\!\cJ})\|_{\HS,\mbH}, \|\cP_{\!\cJ}(\cI_\mbH - \hcP_{\!\cJ})\|_{\HS,\mbH}^2 \}$,
	applying Lemmas~\ref{spec:comp2:aux2} and \ref{spec:comp2:aux3} gives
	\begin{align*}
		\| \cP_{\!\cJ}(\cI_\mbH - \hcP_{\!\cJ})(\hcH-\cH)\cP_{\!\cJ} \|_{\HS,\mbH}
		&\leq \min \bigg\{ \frac{\hbeta_{\!\cJ}}{1-2\halpha_{\!\cJ}},\, \frac{\hbeta^2_{\!\cJ}}{(1-2\halpha_{\!\cJ})^2} \bigg\} \cdot \gammaJ \halpha_{\!\cJ} 
		\\&\quad+ \frac{\sqrt{\gammaJ}\hbeta_{\!\cJ}}{1-2\halpha_{\!\cJ}} \cdot \min \{ \sqrt{\gammaJ}\halpha_{\!\cJ},\sqrt{\gammaJ}\hbeta_{\!\cJ}\}
		\\&\leq \min \bigg\{ \frac{2\gammaJ\halpha_{\!\cJ}\hbeta_{\!\cJ}}{1-2\halpha_{\!\cJ}},\, \frac{(1-\halpha_{\!\cJ})\gammaJ\hbeta_{\!\cJ}^2}{(1-2\halpha_{\!\cJ})^2} \bigg\}.
	\end{align*}
	Thus, we obtain
	\begin{equation}
		\|\cE_1\|_{\HS,\mbH} = \|\cE_2\|_{\HS,\mbH} 
		\leq \frac{|\xi_1|}{2} \cdot \min \bigg\{ \frac{2\gammaJ\halpha_{\!\cJ}\hbeta_{\!\cJ}}{1-2\halpha_{\!\cJ}},\, \frac{(1-\halpha_{\!\cJ})\gammaJ\hbeta_{\!\cJ}^2}{(1-2\halpha_{\!\cJ})^2} \bigg\}.
		\label{spec:comp3-eq2}
	\end{equation}
	
	\mycase{3. Bounding $\|\cE_3\|_{\HS,\mbH}$}
	Using Lemma~\ref{spec:comp2:aux1} and $\sum_{m \notin \cJ} \hcQ_m = \cI_\mbH- \hcP_{\!\cJ}$, we express $\cE_3$ as
	\begin{align*}
		\cE_3
		&= \sum_{k \in \cJ} \sum_{\ell \in \cJ} \sum_{m \not\in \cJ} \frac{(\xi_1 \lambda_k + \xi_2)+(\xi_1 \lambda_\ell + \xi_2)}{2} \cQ_k \hcQ_m \cQ_\ell,	
		\\&=\sum_{k \in \cJ} \sum_{\ell \in \cJ} \frac{(\xi_1 \lambda_k + \xi_2)+(\xi_1 \lambda_\ell + \xi_2)}{2} \cQ_k (\cI_\mbH-\hcP_{\!\cJ}) \cQ_\ell.
	\end{align*}
	Noting $\max\limits_{k,\ell \in \cJ} \Big| \frac{(\xi_1 \lambda_k + \xi_2)+(\xi_1 \lambda_\ell + \xi_2)}{2} \Big| 
	\leq \max\limits_{d \in [D]} |\xi_1 \theta_d + \xi_2|$ and $\cI_\mbH-\hcP_{\!\cJ}=(\cI_\mbH-\hcP_{\!\cJ})^2$, we bound
	\begin{align*}
		\|\hcE_3\|_{\HS,\mbH} &\leq \max_{d \in [D]} |\xi_1 \theta_d + \xi_2| \sqrt{ \sum_{k \in \cJ} \sum_{\ell \in \cJ} \|\cQ_k(\cI_\mbH-\hcP_{\!\cJ})\cQ_\ell \|_{\HS,\mbH}^2 }
		\\&= \max_{d \in [D]} |\xi_1 \theta_d + \xi_2| \cdot \|\cP_{\!\cJ} (\cI_\mbH-\hcP_{\!\cJ}) \cP_{\!\cJ} \|_{\HS,\mbH}
		\\&\leq \max_{d \in [D]} |\xi_1 \theta_d + \xi_2| \cdot \|\cP_{\!\cJ}(\cI_\mbH-\hcP_{\!\cJ}) \|_{\op,\mbH} \cdot 
		\|(\cI_\mbH-\hcP_{\!\cJ}) \cP_{\!\cJ}\|_{\HS,\mbH}
	\end{align*}
	Applying Lemmas~\ref{spec:comp2:aux3} and \ref{spec:comp2:aux4} gives
	\begin{equation}
		\|\cE_3\|_{\HS,\mbH} \leq \max_{d \in [D]} |\xi_1 \theta_d + \xi_2 | 
		\cdot \min \bigg\{ \frac{(2-\halpha_{\!\cJ}) \sqrt{D}\halpha_{\!\cJ} \hbeta_{\!\cJ}}{(1-\halpha_{\!\cJ})^2(1-2\halpha_{\!\cJ})} ,\, \frac{\hbeta_{\!\cJ}^2}{(1-2\halpha_{\!\cJ})^2}\bigg\}.
		\label{spec:comp3-eq3}
	\end{equation}
	
	\mycase{4. Concluding the proof}	
	By combining \eqref{spec:comp3-eq1}-\eqref{spec:comp3-eq3}, we obtain
	\begin{align*}
		&\quad \Big\| \Big( \langle [f_\cJ(\hcH) - f_\cJ(\cH)- \nabla_\cH f(\hcH)]\psi_k , \psi_\ell \rangle_\mbH \Big)_{k,\ell \in \cJ} \Big\|_\F
		\\&\leq \|\cE_1\|_{\HS,\mbH}+\|\cE_2\|_{\HS,\mbH}+\|\cE_3\|_{\HS,\mbH}
		\\&\leq |\xi_1| \cdot \min \bigg\{ \frac{2\gammaJ\halpha_{\!\cJ}\hbeta_{\!\cJ}}{1-2\halpha_{\!\cJ}},\, \frac{(1-\halpha_{\!\cJ})\gammaJ\hbeta_{\!\cJ}^2}{(1-2\halpha_{\!\cJ})^2} \bigg\}
		\\&\quad + \max_{d \in [D]} |\xi_1 \theta_d + \xi_2 | 
		\cdot \min \bigg\{ \frac{(2-\halpha_{\!\cJ}) \sqrt{D}\halpha_{\!\cJ} \hbeta_{\!\cJ}}{(1-\halpha_{\!\cJ})^2(1-2\halpha_{\!\cJ})} ,\, \frac{\hbeta_{\!\cJ}^2}{(1-2\halpha_{\!\cJ})^2}\bigg\}
		\\&\leq \min \bigg\{ \, \bigg( 2 |\xi_1| \gammaJ + \frac{(2-\halpha_{\!\cJ})\sqrt{D}}{(1-\halpha_{\!\cJ})^2} \max_{d \in [D]} |\xi_1 \theta_d + \xi_2| \bigg) \frac{\halpha_{\!\cJ} \hbeta_{\!\cJ}}{1-2\halpha_{\!\cJ}}
		,\, 
		\\& \quad \Big( (1-\halpha_{\!\cJ})|\xi_1| \gammaJ + \max_{d \in [D]}|\xi_1 \theta_d + \xi_2| \Big)\frac{\hbeta_{\!\cJ}^2}{(1-2\halpha_{\!\cJ})^2} \bigg\}.
	\end{align*}
	Hence, the proof is complete by considering $\halpha_{\!\cJ}\leq 1/4$.
\end{proof}

\begin{proof}[Proof of \eqrefsupp{spec:comp3-res2} in Proposition~\refsupp{spec:comp3}]
	We split the proof in several steps.
	
	\mycase{1. Decomposition}
	From the previous proof, we show that for $f(z) = \xi_1 z + \xi_2$,
	\begin{gather*}
		\cP_{\!\cJ} [ f_\cJ (\hcH) - f_\cJ (\cH) - \nabla_\cH f_\cJ  (\hcH) ] \cP_{\!\cJ}
		= -(\cE_1 + \cE_2 + \cE_3), \\
		\cE_1 = \frac{\xi_1}{2} \cP_{\!\cJ}(\cI_\mbH - \hcP_{\!\cJ})(\hcH-\cH)\cP_{\!\cJ}, \quad\quad
		\cE_2 = \frac{\xi_1}{2} \cP_{\!\cJ}(\hcH-\cH) (\cI_\mbH - \hcP_{\!\cJ}) \cP_{\!\cJ}, \\
		\cE_3 = \sum_{k \in \cJ} \sum_{\ell \in \cJ} \frac{(\xi_1 \lambda_k + \xi_2)+(\xi_1 \lambda_\ell + \xi_2)}{2} \cQ_k (\cI_\mbH-\hcP_{\!\cJ}) \cQ_\ell.
	\end{gather*}
	Also, using the same Cauchy integral formula argument as in the first step of the previous proof, it can be shown that
	\begin{align*}
		\cE_4 &:= (\cI_\mbH - \cP_{\!\cJ}) [ f_\cJ (\hcH) - f_\cJ (\cH) - \nabla_\cH f_\cJ  (\hcH) ] (\cI_\mbH - \cP_{\!\cJ})
		\\&= \sum_{k \notin \cJ} \sum_{\ell \notin \cJ} \sum_{m \in \cJ} (\xi_1 \hlambda_m + \xi_2) \cQ_k \hcQ_m \cQ_\ell.
	\end{align*}
	These yields
	\begin{align}
		\tr\Big( f_\cJ (\hcH) - f_\cJ (\cH) - \nabla_\cH f_\cJ  (\hcH) \Big) = -\tr(\cE_1) -\tr(\cE_2) -\tr(\cE_3) + \tr(\cE_4).
		\label{spec:comp3-res2-eq1}
	\end{align}
	Below, we bound the traces of the terms $\cE_k$.
	
	\mycase{2. Bounding $|\tr(\cE_1)|$ and $|\tr(\cE_2)|$}
	By symmetry, it suffices to consider $\cE_1$.
	
	Using $\cI_\mbH = \cP_{\!\cJ} \cP_{\!\cJ} + |\cR|_{\!\cJ}^{-1/2} |\cR|_{\!\cJ}^{1/2}$, $\cI_\mbH - \hcP_{\!\cJ} = (\cI_\mbH - \hcP_{\!\cJ})^2$ and Lemmas~\ref{spec:comp2:aux2} and \ref{spec:comp2:aux3}, we separate and bound
	\begin{align*}
		|\tr (\cP_{\!\cJ}(\cI_\mbH - \hcP_{\!\cJ})(\hcH-\cH)\cP_{\!\cJ})|
		&\leq |\tr( \cP_{\!\cJ}(\cI_\mbH - \hcP_{\!\cJ}) (\cI_\mbH - \hcP_{\!\cJ}) \cP_{\!\cJ} \cP_{\!\cJ} (\hcH-\cH)\cP_{\!\cJ} ) |
		\\&\quad + |\tr( \cP_{\!\cJ}(\cI_\mbH - \hcP_{\!\cJ}) |\cR|_{\!\cJ}^{-1/2}|\cR|_{\!\cJ}^{1/2} (\hcH-\cH)\cP_{\!\cJ} |
		\\&\leq \|\cP_{\!\cJ}(\cI_\mbH - \hcP_{\!\cJ}) \|_{\HS,\mbH}^2 \|\cP_{\!\cJ} (\hcH-\cH)\cP_{\!\cJ}\|_{\op,\mbH}
		\\&\quad + \| \cP_{\!\cJ}(\cI_\mbH - \hcP_{\!\cJ}) |\cR|_{\!\cJ}^{-1/2} \|_{\HS,\mbH} \| |\cR|_{\!\cJ}^{1/2} (\hcH-\cH)\cP_{\!\cJ} \|_{\HS,\mbH}.
		\\&\leq \frac{\hbeta_{\!\cJ}^2}{1-2\halpha_{\!\cJ}} \cdot \gammaJ \halpha_{\!\cJ} + \frac{\sqrt{\gammaJ}\hbeta_{\!\cJ}}{1-2\halpha_{\!\cJ}} \cdot \sqrt{\gammaJ}\hbeta_{\!\cJ}
		= \frac{(1-\halpha_{\!\cJ}) \gammaJ \hbeta_{\!\cJ}^2}{1-2\halpha_{\!\cJ}}.
	\end{align*}
	Thus, we obtain
	\begin{equation}
		|\tr(\cE_1)| = |\tr(\cE_2)| \leq \frac{|\xi_1|}{2} \cdot \frac{(1-\halpha_{\!\cJ}) \gammaJ \hbeta_{\!\cJ}^2}{1-2\halpha_{\!\cJ}}.
		\label{spec:comp3-res2-eq2}
	\end{equation}
	
	\mycase{3. Bounding $|\tr(\cE_3)|$}
	Noting $\langle |\cQ_k (\cI_\mbH - \hcP_{\!\cJ}) \cQ_k \psi_k , \psi_k \rangle_\mbH| = \|(\cI_\mbH - \hcP_{\!\cJ}) \cQ_k\|_\mbH^2$,	
	we have
	\begin{align*}
		|\tr(\cE_3)|
		&= \bigg| \sum_{k \in \cJ} (\xi_1 \lambda_k + \xi_2) \langle \cQ_k (\cI_\mbH - \hcP_{\!\cJ}) \cQ_k \psi_k , \psi_k \rangle_\mbH \bigg|
		\\&\leq \max_{k \in \cJ} |\xi_1 \lambda_k + \xi_2| \cdot \bigg| \sum_{k \in \cJ}\langle \cQ_k (\cI_\mbH - \hcP_{\!\cJ}) \cQ_k \psi_k , \psi_k \rangle_\mbH \bigg|
		\\&= \max_{d \in [D]} |\xi_1 \theta_d + \xi_2| \cdot \sum_{k \in \cJ} \|(\cI_\mbH-\hcP_{\!\cJ})\cQ_k\|_\mbH^2
		= \max_{d \in [D]} |\xi_1 \theta_d + \xi_2| \cdot \|(\cI_\mbH - \hcP_{\!\cJ}) \cP_{\!\cJ}\|_{\HS,\mbH}^2.
	\end{align*}
	Applying Lemma~\ref{spec:comp2:aux3}, it follows that
	\begin{equation}
		|\tr(\cE_3)| \leq \max_{d \in [D]} |\xi_1 \theta_d + \xi_2| \cdot \frac{\hbeta_{\!\cJ}^2}{(1-2\halpha_{\!\cJ})^2}.
		\label{spec:comp3-res2-eq3}
	\end{equation}
	
	\mycase{4. Bounding $|\tr(\cE_4)|$}
	Noting $\langle \cQ_k \hcQ_m \cQ_k \psi_k , \psi_k \rangle_\mbH = \|\hcQ_m \cQ_k\|_\mbH^2$, we have	
	\begin{align*}
		|\tr(\cE_4)|
		&= \bigg| \sum_{k \notin \cJ} \sum_{m \in \cJ} (\xi_1 \hlambda_m + \xi_2) \langle \cQ_k \hcQ_m \cQ_k \psi_k , \psi_k \rangle_\mbH \bigg|
		\\&\leq \max_{m \in \cJ} |\xi_1 \hlambda_m + \xi_2| \cdot \bigg| \sum_{k \notin \cJ} \sum_{m \in \cJ} \langle \cQ_k \hcQ_m \cQ_k \psi_k , \psi_k \rangle_\mbH \bigg|
		\\&\leq \bigg( |\xi_1| \max_{m \in \cJ}|\hlambda_m - \lambda_m| + \max_{m \in \cJ} |\xi_1 \lambda_m + \xi_2| \bigg) 
		\cdot \sum_{k \notin \cJ} \sum_{m \in \cJ} \|\hcQ_m \cQ_k\|_\mbH^2
		\\&= \bigg( |\xi_1| \max_{m \in \cJ}|\hlambda_m - \lambda_m| + \max_{d \in [D]} |\xi_1 \theta_d + \xi_2| \bigg) 
		\| \hcP_{\!\cJ} (\cI_\mbH-\cP_{\!\cJ}) \|_{\HS,\mbH}^2.
	\end{align*}
	Applying Lemmas~\ref{spec:comp1:aux2} and \ref{spec:comp2:aux3}, it follows that
	\begin{equation}
		|\tr(\cE_4)|
		\leq \bigg( |\xi_1| \gammaJ \halpha_{\!\cJ} + \max_{d \in [D]} |\xi_1 \theta_d + \xi_2| \bigg) 
		\frac{\hbeta_{\!\cJ}^2}{(1-2\halpha_{\!\cJ})^2}.
		\label{spec:comp3-res2-eq4}
	\end{equation}
	
	\mycase{5. Concluding the proof}
	By combining \eqref{spec:comp3-res2-eq1}-\eqref{spec:comp3-res2-eq4}, we obtain
	\begin{multline*}
		\tr\Big( f_\cJ (\hcH) - f_\cJ (\cH) - \nabla_\cH f_\cJ  (\hcH) \Big)
		\leq |\tr(\cE_1)| + |\tr(\cE_2)| + |\tr(\cE_3)| + |\tr(\cE_4)|
		\\ \leq |\xi_1| \frac{(1-\halpha_{\!\cJ})\gammaJ\hbeta_{\!\cJ}^2}{1-2\halpha_{\!\cJ}}
		+ \max_{d \in [D]} |\xi_1 \theta_d + \xi_2| \cdot \frac{\hbeta_{\!\cJ}^2}{(1-2\halpha_{\!\cJ})^2}
		+ \bigg( |\xi_1| \gammaJ \halpha_{\!\cJ} + \max_{d \in [D]} |\xi_1 \theta_d + \xi_2| \bigg) 
		\frac{\hbeta_{\!\cJ}^2}{(1-2\halpha_{\!\cJ})^2}
		\\ = \bigg( |\xi_1| \gammaJ (1-2\halpha_{\!\cJ}+2\halpha_{\!\cJ}^2) + 2\max_{d \in [D]} |\xi_1 \theta_d + \xi_2| \bigg)  \frac{\hbeta_{\!\cJ}^2}{(1-2\halpha_{\!\cJ})^2}.
	\end{multline*}
	Hence, the proof is complete by considering $\halpha_{\!\cJ} \leq 1/4$.
\end{proof}

\begin{proof}[Proof of \eqrefsupp{spec:comp3-res3} in Proposition~\refsupp{spec:comp3}]
	From the previous proof, we show that for $f(z) = z$,
	\begin{gather*}
		\cP_{\!\cJ} [ f_\cJ (\hcH) - f_\cJ (\cH) - \nabla_\cH f_\cJ  (\hcH) ] \cP_{\!\cJ}
		= -(\cE_1 + \cE_2 + \cE_3),	\\
		\cE_1 = \frac{1}{2} \cP_{\!\cJ}(\cI_\mbH - \hcP_{\!\cJ})(\hcH-\cH)\cP_{\!\cJ}, \quad
		\cE_2 = \frac{1}{2} \cP_{\!\cJ}(\hcH-\cH) (\cI_\mbH - \hcP_{\!\cJ}) \cP_{\!\cJ}, \\
		\cE_3 = \sum_{k \in \cJ} \sum_{\ell \in \cJ} \frac{\lambda_k + \lambda_\ell}{2} \cQ_k (\cI_\mbH-\hcP_{\!\cJ}) \cQ_\ell.
	\end{gather*}
	With $\|\Lambda_{\!\cJ}^{-1/2}\|_{\op,\mbH} = \theta_\tmin^{-1/2}$, using the same argument when bounding $|\tr(\cE_1)|$, it can be easily shown that 
	\begin{equation}
		|\tr( \Lambda_{\!\cJ}^{-1/2} \cE_1 \Lambda_{\!\cJ}^{-1/2} )|
		+ |\tr( \Lambda_{\!\cJ}^{-1/2} \cE_2 \Lambda_{\!\cJ}^{-1/2} )|
		\leq \frac{(1-\halpha_{\!\cJ}) \gammaJ \hbeta_{\!\cJ}^2}{\theta_\tmin(1-2\halpha_{\!\cJ})}, \quad
	\end{equation}
	For $\Lambda_{\!\cJ}^{-1/2}\cE_3\Lambda_{\!\cJ}^{-1/2}$, a direct calculation yields
	\begin{align*}
		\tr(\Lambda_{\!\cJ}^{-1/2}\cE_3\Lambda_{\!\cJ}^{-1/2})
		&= \sum_{k \in \cJ} \lambda_k \langle \cQ_k (\cI_\mbH-\hcP_{\!\cJ}) \cQ_k \lambda_k^{-1/2} \psi_k, \lambda_k^{-1/2} \psi_k \rangle_\mbH
		= \sum_{k \in \cJ} \langle \cQ_k (\cI_\mbH-\hcP_{\!\cJ}) \cQ_k \psi_k, \psi_k \rangle_\mbH
		\\&= \sum_{k \in \cJ} \| (\cI_\mbH-\hcP_{\!\cJ}) \cQ_k \psi_k \|_\mbH^2
		= \|(\cI_\mbH-\hcP_{\!\cJ}) \cP_{\!\cJ} \|_{\HS,\mbH}^2,
	\end{align*}
	and by applying \ref{spec:comp2:aux3}, it follows that
	\begin{equation}
		|\tr(\Lambda_{\!\cJ}^{-1/2}\cE_3\Lambda_{\!\cJ}^{-1/2})|
		\leq \frac{\hbeta_{\!\cJ}^2}{(1-2\halpha_{\!\cJ})^2}.
	\end{equation}
	Therefore, combining the above results yields
	\begin{multline*}
		\Big| \tr( \Lambda_{\!\cJ}^{-1/2} \cP_{\!\cJ} [ f_\cJ (\hcH) - f_\cJ (\cH) - \nabla_\cH f_\cJ  (\hcH) ] \cP_{\!\cJ} \Lambda_{\!\cJ}^{-1/2} ) \Big|
		\\ \leq |\tr(\Lambda_{\!\cJ}^{-1/2}\cE_1\Lambda_{\!\cJ}^{-1/2})| + |\tr(\Lambda_{\!\cJ}^{-1/2}\cE_2\Lambda_{\!\cJ}^{-1/2})| + |\tr(\Lambda_{\!\cJ}^{-1/2}\cE_3\Lambda_{\!\cJ}^{-1/2})|
		\\ \leq \Big( \frac{\gammaJ}{\theta_\tmin} (1-\halpha_{\!\cJ})(1-2\halpha_{\!\cJ})  + 1 \Big) \frac{\hbeta_{\!\cJ}^2}{(1-2\halpha_{\!\cJ})^2}.
	\end{multline*}
	Hence, the proof is complete by considering $\halpha_{\!\cJ} \leq 1/4$.	
\end{proof}

\subsection{Proof of Lemma~\refsupp{spec:eigproj:aux}}

\begin{proof}[Proof of Lemma~\ref{spec:eigproj:aux}]
	Using $\cI_\mbH- \cP_{\!\cJ}=(\cI_\mbH- \cP_{\!\cJ})^2$, Lemma~\ref{spec:comp2:aux3} and the condition $\halpha_{\!\cJ}\le1/4$, we have
	\begin{align}
		\|\hcS_{\!\cJ} (\cI_\mbH- \cP_{\!\cJ})\hcP_{\!\cJ}\|_{\HS,\mbH}
		&\leq \|\hcS_{\!\cJ} (\cI_\mbH- \cP_{\!\cJ})\|_{\op,\mbH} \cdot \|(\cI_\mbH- \cP_{\!\cJ})\hcP_{\!\cJ}\|_{\HS,\mbH}
		\notag\\&\leq \|\hcS_{\!\cJ} (\cI_\mbH- \cP_{\!\cJ})\|_{\op,\mbH} \cdot \frac{\hbeta_{\!\cJ}}{1-2\halpha_{\!\cJ}}
		\leq 2 \|\hcS_{\!\cJ} (\cI_\mbH- \cP_{\!\cJ})\|_{\op,\mbH}\cdot  \hbeta_{\!\cJ}.
		\label{spec:eigproj:aux-eq1}
	\end{align}
	Hence, it suffices to bound $\|\hcS_{\!\cJ} (\cI_\mbH- \cP_{\!\cJ})\|_{\op,\mbH}$.
	 
	Recall from the proof of Lemma~\ref{spec:comp2:aux4} that
	\begin{equation*}
		|\cR|_{\!\cJ}^{-1/2} = \sum_{k \notin \cJ} \sqrt{\min_{m \in \cJ} |\lambda_k-\lambda_m|} \cQ_k,
		\quad
		\cR_{d,\cJ} = \sum_{\ell \notin \cJ} \frac{\cQ_\ell}{\theta_d-\lambda_\ell},
		\quad
		\| |\cR|_{\!\cJ}^{-1/2} \cR_{d,\cJ} \|_{\op,\mbH} \leq \gammaJ^{-1/2}.
	\end{equation*}
	Using $\cP_{\!\cJ_d} = \sum_{k \in \cJ_d} \cQ_k$ and $\lambda_k = \theta_d$ for $k \in \cJ_d$, we expand $\hcS_{\!\cJ} (\cI_\mbH - \cP_{\!\cJ})$ as 
	\begin{align*}
		\hcS_{\!\cJ} (\cI_\mbH - \cP_{\!\cJ})
		&= \sum_{k \in \cJ} \sum_{\ell \notin \cJ} \frac{\cQ_k (\hcH-\cH) \cQ_\ell}{\lambda_k-\lambda_\ell}
		= \sum_{d=1}^D \sum_{k \in \cJ_d} \sum_{\ell \notin \cJ} \frac{\cQ_k (\hcH-\cH) \cQ_\ell}{\theta_d-\lambda_\ell}
		= \sum_{d=1}^D \cP_{\!\cJ_d} (\hcH-\cH) \cR_{d,\cJ}.
	\end{align*}
	Since 
	\begin{align*}
		\| \cP_{\!\cJ_d} (\hcH-\cH) \cR_{d,\cJ} \|_{\op,\mbH}
		\leq \| \cP_{\!\cJ} (\hcH-\cH) |\cR|_{\cJ}^{1/2} \|_{\op,\mbH} \cdot \| |\cR|_{\!\cJ}^{-1/2} \cR_{d,\cJ} \|_{\op,\mbH}
		\leq \halpha_{\!\cJ},
	\end{align*}
	we obtain
	\begin{align}
		\| \hcS_{\!\cJ} (\cI_\mbH - \cP_{\!\cJ}) \|_{\op,\mbH}
		\leq \sqrt{\sum_{d=1}^D \|\cP_{\!\cJ_d} (\hcH-\cH) \cR_{d,\cJ}\|^2_{\op,\mbH}}
		\leq \sqrt{D} \halpha_{\!\cJ}.
		\label{spec:eigproj:aux-eq2}
	\end{align}
	We conclude the proof by combining \eqref{spec:eigproj:aux-eq2} and \eqref{spec:eigproj:aux-eq2}.
\end{proof}

\section{Auxiliary lemmas}

We state several lemmas used to prove the main results and provide their proofs.

\begin{lemma} \label{sup:asymp.ortho}
	Suppose $U,E \in \R^{m \times m}$ with $ U U^\top - I_m = E$.
	For any symmetric matrix $A \in \R^{m \times m}$, we have
	\begin{equation}
		\| \specd(UAU^\top) - \specd(A) \|_2 \leq 
		\begin{cases}
			\|A\|_\op \|E\|_\F, & \text{if}\hquad m=1, \\
			\sqrt{2} \|A\|_\op \|E\|_\F, & \text{otherwise}.
		\end{cases}
	\end{equation}
\end{lemma}

\begin{proof}[Proof of Lemma~\ref{sup:asymp.ortho}]
    Note that for any square matrices $B$ and $C$, $BC$ and $CB$ have the same eigenvalues.
    This implies $\specd(UAU^\top) = \specd(AU^\top U)$ and $\|\tilde{E}\|_\F = \|E\|_\F$ where $\tilde{E} = U^\top U - I_m$.
	
	Using the modified version of the Hoffman–Wielandt theorem stated in Theorem~4.2 of \citesupp{xu2017new}, we have
	\begin{align*}
		\| \specd(UAU^\top) - \specd(A) \|_2
		&= \| \specd(A U^\top U) - \specd(A) \|_2
		\\&= \| \specd(A + A \tilde{E}) - \specd(A) \|_2
		\leq \sqrt{ 2 \|A \tilde{E}\|_\F^2 - \frac{1}{m} |\tr(A \tilde{E})|^2 }.
	\end{align*}
	If $m=1$, we see that $\|A \tilde{E}\|_\F^2 = \frac{1}{m} |\tr(A \tilde{E})|^2$, hence
	\begin{align*}
		\| \specd(UAU^\top) - \specd(A) \|_2 \leq \|A \tilde{E}\|_\F \leq \|A\|_\op \|\tilde{E}\|_\F = \|A\|_\op \|E\|_\F.
	\end{align*}
	If $m>1$, then
	\begin{align*}
		\| \specd(UAU^\top) - \specd(A) \|_2 \leq \sqrt{2} \|A \tilde{E}\|_\F \leq \sqrt{2}\|A\|_\op \|\tilde{E}\|_\F = \sqrt{2}\|A\|_\op \|E\|_\F.
	\end{align*}
	This completes the proof.
\end{proof}

The following three lemmas are well-known consequences of Cauchy's integral formula.
For completeness, we provide their proofs.

\begin{lemma} \label{sup:cauchy}
	Let $\Gamma \subset \C$ be an open disk with positively oriented boundary $\partial \Gamma$,
	and let $f: \C \to \C$ be holomorphic on a neighborhood of $\Gamma$.
	Then, for $a,b \in \C \backslash \partial\Gamma$, we have
	\begin{equation*}
		C(a,b) = \frac{1}{2 \pi i} \oint_{\partial\Gamma} \frac{f(z)}{(z-a)(z-b)} \, dz =
		\begin{cases}
			f'(a), & \text{if } a,b \in \Gamma \text{ and } a=b, \\[.25em]
			\frac{f(a)-f(b)}{a-b}, & \text{if } a,b \in \Gamma \text{ and } a\not=b, \\[.75em]
			\frac{f(a)}{a-b}, & \text{if } a \in \Gamma, b \notin \Gamma, \\[.75em]
			\frac{f(b)}{b-a}, & \text{if } a \notin \Gamma, b \in \Gamma, \\[.25em]
			0, & \text{if } a,b \notin \Gamma.
		\end{cases}
	\end{equation*}
\end{lemma}

\begin{proof}[Proof of Lemma~\ref{sup:cauchy}]
	According to Cauchy's integral formula, for a holomorphic function $g:\C \to \C$, $c \in \Gamma$ and $k \in \N$, it holds that
	\begin{equation}
		\frac{1}{2 \pi i} \oint_{\partial\Gamma} g(z) \, dz = 0,
		\quad 
		\frac{1}{2\pi i} \oint_{\partial\Gamma} \frac{g(z)}{(z-c)^k} \, dz = \frac{g^{(k-1)}(c)}{(k-1)!}.
		\label{sup:cauchy-eq1}
	\end{equation}
	Using the above, we evaluate $C(a,b)$ by cases.
	
	\mycase{Case 1. $a,b \in \Gamma$, $a=b$} Clearly,
	$$ C(a,b) = \frac{1}{2 \pi i} \oint_{\partial\Gamma} \frac{f(z)}{(z-a)^2} \, dz = f'(a).$$
	
	\mycase{Case 2. $a,b \in \Gamma$, $a\not=b$}
	Using the partial fraction decomposition, we have
	$$  \frac{f(z)}{(z-a)(z-b)} = \frac{1}{a-b} \left( \frac{f(z)}{z-a} - \frac{f(z)}{z-b} \right). $$
	This gives
	\begin{align*}
		C(a,b)
		&= \frac{1}{a-b} \left( \frac{1}{2 \pi i} \oint_{\partial\Gamma} \frac{f(z)}{z-a} \, dz - \frac{1}{2 \pi i}
		\oint_{\partial\Gamma} \frac{f(z)}{z-b} \, dz \right)
		= \frac{1}{a-b} (f(a) - f(b)).
	\end{align*}
	
	\mycase{Case 3. $a \in \Gamma, b \notin \Gamma$}
	Since $f(z)/(z-b)$ is holomorphic on $\Gamma$, we have
	\begin{equation*}
		C(a,b)
		= \frac{1}{a-b} \left( \frac{1}{2 \pi i} \oint_{\partial\Gamma} \frac{f(z)}{z-a} \, dz - \frac{1}{2 \pi i}
		\oint_{\partial\Gamma} \frac{f(z)}{z-b} \, dz \right)
		= \frac{1}{a-b} (f(a) - 0)
		= \frac{f(a)}{a-b}.
	\end{equation*}
	By symmetry, we can obtain the case with $a \notin \Gamma$ and $b \in \Gamma$.
	
	\mycase{Case 4. $a, b \notin \Gamma$}
	Since $f(z)/(z-a)(z-b)$ is holomorphic on $\Gamma$, we have
	$$ C(a,b) =  \frac{1}{2 \pi i} \oint_{\partial\Gamma} \frac{f(z)}{(z-a)(z-b)} \, dz = 0. $$
	
	Combining all the above cases completes the proof.
\end{proof}

\begin{lemma} \label{sup:cauchy2}
	For the same $\Gamma$ as in Lemma~\ref{sup:cauchy}, if $a,b,c \in \C \backslash \partial \Gamma$, we have
	\begin{align*}
		C(a,b,c) &:= \frac{1}{2 \pi i} \oint_{\partial\Gamma} \frac{1}{(z-a)(z-b)(z-c)} \, dz
		=
		\begin{cases}
			0, & \text{if } a,b,c \in \Gamma, \\
			-\frac{1}{(c-a)(c-b)}, & \text{if } a,b \in \Gamma, c \notin \Gamma, \\[.75em]
			\frac{1}{(a-b)(a-c)}, & \text{if } a \in \Gamma, b,c \notin \Gamma, \\
			0, & \text{if } a,b,c \notin \Gamma,
		\end{cases}
	\end{align*}
	and the remaining cases follow by symmetry in $(a,b,c)$.
\end{lemma}

\begin{proof}[Proof of Lemma~\ref{sup:cauchy2}]
	Using the partial fraction decomposition, we write
	\begin{align*}
		&\quad \frac{1}{(z-a)(z-b)(z-c)} 
		\\&=
		\begin{cases}
			\frac{1}{(z-a)^3}, & \text{if } a=b=c, \\
			-\frac{1}{(c-a)(z-a)^2} - \frac{1}{(c-a)^2(z-a)} + \frac{1}{(c-a)^2(z-c)}, & \text{if } a=b\not=c, \\
			\frac{1}{(a-b)(a-c)(z-a)} + \frac{1}{(b-c)(b-a)(z-b)} + \frac{1}{(c-a)(c-b)(z-c)} & \text{if } \text{$a,b,c$ are distinct}, \\
		\end{cases}
	\end{align*}
	and the remaining cases follow by permuting $a,b,c$.
	Combining this with Cauchy's integral formula, we compute $C(a,b,c)$ by cases.
	
	\mycase{Case 1. $a,b,c \in \Gamma$}
	If $a=b=c$, then $C(a,b,c) = 0$ obviously.
	If $a=b\not=c$, then $ C(a,b,c) = - \frac{1}{(c-a)^2} + \frac{1}{(c-a)^2} = 0$. If $a,b,c$ are distinct, then
	$C(a,b,c) = \frac{1}{(a-b)(a-c)}+\frac{1}{(b-c)(b-a)}+\frac{1}{(c-a)(c-b)}
	= \frac{-(b-c)-(c-a)-(a-b)}{(a-b)(b-c)(c-a)}=0.$
	Hence, in all cases, we obtain $C(a,b,c)=0$ if $a,b,c \in \Gamma$.
	
	\mycase{Case 2. $a,b \in \Gamma$, $c \notin \Gamma$}
	If $a=b$, then $ C(a,b,c) = -\frac{1}{(c-a)^2}$.
	If $a\not=b$, then $$ C(a,b,c) = \frac{1}{(a-b)(a-c)}+\frac{1}{(b-c)(b-a)} = 
	\frac{(b-c)-(a-c)}{(a-b)(a-c)(b-c)} = -\frac{1}{(c-a)(c-b)}.$$
	So, we obtain $C(a,b,c)=-\frac{1}{(a-b)(a-c)}$ if $a,b\in \Gamma, c \notin \Gamma$.
	
	\mycase{Case 3. $a \in \Gamma$, $b,c \notin \Gamma$}
	If $b=c$, using $$\frac{1}{(z-a)(z-b)^2} = -\frac{1}{(a-b)(z-b)^2} -\frac{1}{(a-b)^2(z-b)} + \frac{1}{(a-b)^2(z-a)},$$
	we have $C(a,b,c)=\frac{1}{(a-b)^2}$.
	If $b\not=c$, $C(a,b,c) = \frac{1}{(a-b)(a-c)}$.
	
	Combining all the above cases completes the proof.
\end{proof}

\begin{lemma} \label{sup:cauchy3}
	For the same $\Gamma$ as in Lemma~\ref{sup:cauchy}, let $\xi_1,\xi_2 \in \C$ and $a,b,c \in \C \backslash \partial\Gamma$.
	If $a,b \in \Gamma$, we have
	\begin{align*}
		C(a,b,c)
		&:= \frac{1}{2 \pi i} \oint_{\partial\Gamma} \frac{\xi_1 z + \xi_2}{(z-a)(z-b)(z-c)} \, dz
		\\&= \begin{cases}
			-\frac{\xi_1}{2(c-a)}-\frac{\xi_1}{2(c-b)}-\frac{\xi_1 (a+b) + 2 \xi_2}{2(c-a)(c-b)},& \text{if } c \notin \Gamma, \\
			0, & \text{otherwise.}
		\end{cases}
	\end{align*}
	If $a,b \notin \Gamma$, then 
	\begin{align*}
		C(a,b,c)
		= \begin{cases}
			\frac{\xi_1 c + b}{(c-a)(c-b)} & \text{if } c \in \Gamma, \\
			0, & \text{otherwise.}
		\end{cases}
	\end{align*}
\end{lemma}

\begin{proof}[Proof of Lemma~\ref{sup:cauchy3}]
	Let $f(z) = \xi_1 z + \xi_2$. Using Cauchy's integral formula, we compute $C(a,b,c)$ by cases.
	
	\mycase{Case 1. $a=b$, $c\in\Gamma$, $c=a$}
	Clearly,
	$$ C(a,b,c) = \frac{1}{2 \pi i} \oint_{\partial\Gamma} \frac{f(z)}{(z-a)^3} \, dz = \frac{f''(a)}{2} = 0.$$
	
	\mycase{Case 2. $a=b$, $c \in \Gamma$, $c\ne a$}
	Using the partial fraction decomposition
	$$ \frac{f(z)}{(z-a)^2(z-c)} = -\frac{f(z)}{(c-a)(z-a)^2} - \frac{f(z)}{(c-a)^2(z-a)} + \frac{f(z)}{(c-a)^2(z-c)}, $$
	applying Cauchy's integral formula gives
	\begin{align*}
		C(a,b,c)
		&= -\frac{f'(a)}{c-a} - \frac{f(a)}{(c-a)^2} + \frac{f(c)}{(c-a)^2}
		= \frac{f(c) - f(a) - f'(a)(c-a)}{(c-a)^2}
		= 0,
	\end{align*}
	where the last equality uses that $f$ is linear.
	
	\mycase{Case 3. $a=b$, $c \notin \Gamma$}
	Using the above decomposition, we obtain
	\begin{align*}
		C(a,b,c)
		&= -\frac{f'(a)}{c-a} - \frac{f(a)}{(c-a)^2} + 0 
		= -\frac{f'(a)}{c-a} - \frac{f(a)}{(c-a)^2}
		= -\frac{\xi_1}{c-a} - \frac{\xi_1 a + \xi_2}{(c-a)^2}.
	\end{align*}
	Since $a=b$, this can be rewritten as
	\begin{equation*}
		C(a,b,c) = -\frac{\xi_1}{2(c-a)} -\frac{\xi_1}{2(c-b)} - \frac{\xi_1 (a+b) + 2\xi_2}{2(c-a)(c-b)}.
		\label{sup:cauchy3-eq1}
	\end{equation*}
	
	\mycase{Case 4. $a\ne b$, $c\in\Gamma$, $c=a$}
	Using the following partial fraction decomposition
	$$ \frac{f(z)}{(z-a)^2(z-b)} = \frac{f(z)}{(a-b)(z-a)^2} - \frac{f(z)}{(a-b)^2(z-a)} + \frac{f(z)}{(a-b)^2(z-b)}, $$
	we obtain
	\begin{align*}
		C(a,b,c) 
		&= \frac{f'(a)}{a-b} - \frac{f(a)}{(a-b)^2} + \frac{f(b)}{(a-b)^2}
		= \frac{f'(a)(a-b) - (f(a)-f(b))}{(a-b)^2}
		= 0,
	\end{align*}
	where the last equality uses the fact that $f$ is linear.
	
	\mycase{Case 5. $a\ne b$, $c \in \Gamma$, $c \ne a,b$}
	Using the partial fraction decomposition
	$$ \frac{f(z)}{(z-a)(z-b)(z-c)} = \frac{f(z)}{(a-b)(a-c)(z-a)} + \frac{f(z)}{(b-a)(b-c)(z-b)} + \frac{f(z)}{(c-a)(c-b)(z-c)}, $$
	we obtain
	\begin{align*}
		C(a,b,c) &= \frac{f(a)}{(a-b)(a-c)} + \frac{f(b)}{(b-a)(b-c)} + \frac{f(c)}{(c-a)(c-b)}
		\\&= -\frac{f(a)(a-c) + f(b)(c-a) + f(c)(a-b)}{(a-b)(b-c)(c-a)} = 0.
	\end{align*}
	
	\mycase{Case 6. $a\ne b$, $c \notin \Gamma$}
	Using the previous partial fraction decomposition, we have
	\begin{align*}
		C(a,b,c) &= \frac{f(a)}{(a-b)(a-c)} + \frac{f(b)}{(b-a)(b-c)}
		\\&= \frac{f(a)-f(b)}{(a-b)(a-c)} + f(b) \Bigg\{ \frac{1}{(b-a)(b-c)} - \frac{1}{(a-b)(a-c)}\Bigg\}
		\\&= -\frac{\xi_1}{c-a} - \frac{\xi_1 b + \xi_2}{(c-a)(c-b)}. 
	\end{align*}
	By interchanging $a$ and $b$, we also have
	\begin{align*}
		C(a,b,c) = -\frac{\xi_1}{c-b} - \frac{\xi_1 a + \xi_2}{(c-a)(c-b)}. 
	\end{align*}
	Averaging the two representations yields
	\begin{equation*}
		C(a,b,c) = -\frac{\xi_1}{2(c-a)} -\frac{\xi_1}{2(c-b)} - \frac{\xi_1 (a+b) + 2\xi_2}{2(c-a)(c-b)}.
		\label{sup:cauchy3-eq2}
	\end{equation*}
	
	\mycase{Case 7. $a,b \notin \Gamma$, $c \in \Gamma$}
	Since $\frac{f(z)}{(z-a)(z-b)}$ is analytic on $\Gamma$, we have
	\begin{align*}
		C(a,b,c) = \frac{f(c)}{(c-a)(c-b)} = \frac{\xi_1 c + \xi_2}{(c-a)(c-b)}.
	\end{align*}
	
	\mycase{Case 8. $a,b,c \notin \Gamma$}
	Since $f$ is analytic on $\Gamma$, we have $C(a,b,c) = 0$.
	
	Combining all cases, the claim follows.
\end{proof}

\begin{lemma} \label{sup:Bernstein1}
	Let $\cX_1, \dots, \cX_n$ be a sequence of i.i.d. self-adjoint Hilbert-Schmidt operators on a separable Hilbert space $\mbH$. 
	Suppose $\E \cX_1 = 0$, $\|\cX_1\|_{\op,\mbH} \leq R$, $\| \E \cX_1^2 \|_{\op,\mbH} \leq V$ and $\tr(\E \cX_1^2) \leq DV$ for some $R,V,D >0$.
	If $ t \geq \sqrt{\frac{V}{n}} + \frac{R}{3n}$,
	we have
	\begin{equation} 
		\P \Bigg( \,\bigg\| \frac{1}{n} \sum_{i=1}^{n} \cX_i \bigg\|_{\op,\mbH} \geq t \Bigg) 
		\leq 4D \exp \left( -\frac{3nt^2}{6V + 2Rt} \right).
		\label{sup:Bernstein1-res1}
	\end{equation}
	Alternatively, we have for $\tau \in (0, 2D)$,
	\begin{equation} 
		\P\Bigg( \,\bigg\| \frac{1}{n} \sum_{i=1}^{n} \cX_i \bigg\|_{\op,\mbH} < \sqrt{\frac{2 V \log(4D/\tau)}{n}} + \frac{2R \log(4D/\tau)}{3n} \Bigg) \geq 1- \tau.
		\label{sup:Bernstein1-res2}
	\end{equation}
\end{lemma}

\begin{proof}[Proof of Lemma \ref{sup:Bernstein1}]
	\eqref{sup:Bernstein1-res1} is a direct consequence of Lemma~5 in \citesupp{dicker2017Kernel}.
	
	To show \eqref{sup:Bernstein1-res2}, let $\tau = 4D \exp \left( -\frac{3nt^2}{6V + 2Rt} \right).$
	Solving this for $t$ gives
	\begin{align*}
		t &= \frac{R \log (4D/\tau)}{3n} + \sqrt{ \left(\frac{R \log (4D/\tau)}{3n} \right)^2 + \frac{2V \log(4D/\tau)}{n} }
		&\leq \frac{2R \log (4D/\tau)}{3n} + \sqrt{\frac{2V \log(4D/\tau)}{n}}.
	\end{align*}
	Hence, \eqref{sup:Bernstein1-res2} follows by \eqref{sup:Bernstein1-res1}.
	
	It remains to determine the admissible range of $\tau$.
	Note that $4D \exp \left( -\frac{3nt^2}{6V + 2Rt} \right)$ is decreasing in $t$.
	Hence, we have
	\begin{equation*}
		0 < \tau \leq 4d \exp \left( -\frac{3nt_0^2}{6V + 2Rt_0} \right) = 4d \exp \left( - \frac{3V + 2R \sqrt{V/n} + R^2/(3n)}{6 V +2R\sqrt{V/n} + 2R^2/(3n)} \right),
	\end{equation*}
	where $t_0 = \sqrt{\frac{V}{n}} + \frac{R}{3n}.$
	Since
	$$ \exp \left( - \frac{3V + 2R \sqrt{V/n} + R^2/(3n)}{6 V +2R\sqrt{V/n} + 2R^2/(3n)} \right) $$
	increases with $n$, letting $n \to \infty$  yields
	$$ 0 < \tau \leq 4D \exp \left( -\frac{1}{2} \right) \approx 2.426D,$$
	Hence, we can roughly set $0 < \tau < 2D$, which completes the proof.
\end{proof}

\begin{lemma} \label{sup:Bernstein2}
	Let $\cX_1, \dots, \cX_n$ be a sequence of i.i.d. Hilbert-Schmidt operators on a separable Hilbert space $\mbH$. 
	Suppose $\E \cX_1 = 0$, $\|\cX_1\|_{\HS,\mbH} \leq R$ and $\E \| \cX_1 \|^2_{\HS,\mbH} \leq V$ for some $R,V >0$.
	Then, we have for $0<\tau<1$,
	\begin{equation} 
		\P\Bigg( \,\bigg\| \frac{1}{n} \sum_{i=1}^{n} \cX_i \bigg\|_{\HS,\mbH} 
		\leq \sqrt{\frac{2V \log(2/\tau)}{n}} +  \frac{R\log(2/\tau)}{3n} \bigg) \geq 1- \tau.
	\end{equation}
\end{lemma}

\begin{proof}[Proof of lemma~\ref{sup:Bernstein2}]
	By Pinelis' Bernstein inequalities for smooth Banach spaces (see Section~3.2 of \citesupp{martinez-taboada2026empirical} and the original proof in \citesupp{pinelis1994optimuma}), we have for any $t>0$,
	\begin{equation*}
		\P \Bigg( \,\bigg\| \frac{1}{n} \sum_{i=1}^{n} \cX_i \bigg\|_{\HS,\mbH} \geq t \Bigg) 
		\leq 2 \exp \bigg( -\frac{3nt^2}{6V + Rt} \bigg).
	\end{equation*}
	Using the same argument as in the previous proof, the claim follows.
	
\end{proof}

\section{Additional results}

\subsection{Details of Example~\refsupp{spec:perturb:ex1}}

\subsubsection{Weighted perturbation under relative form boundedness condition}

Recall the RFB condition:
\begin{equation}
	|\langle (\hcH-\cH)v, v \rangle_\mbH | \leq \langle (\e_1 \cH + \e_2 \cI_\mbH)v, v \rangle_\mbH
	\quad \text{for all $v \in \mbH$}, \hquad\text{where $\e_1, \e_2 \ge 0$.}
	\label{spec:sup:RFB-cond}
\end{equation}

Under this condition, $\halpha_{\!\cJ}$ and $\hbeta_{\!\cJ}$ can be controlled as follows.
\begin{lemma} \label{spec:sup:RFB}
	Suppose Assumptions~\refsupp{spec:assump1} and \refsupp{spec:assump2}.
	If the RFB condition \eqref{spec:sup:RFB-cond} is satisfied, we have
	\begin{align*}
		\halpha_{\!\cJ} \leq \gammaJ^{-1} \big( \e_1 (\theta_\tmax+\gammaJ) +\e_2\big), 
		\quad
		\hbeta_{\!\cJ} \leq \gammaJ^{-1} \sqrt{(\e_1 \sum_{k\in\cJ} \lambda_k + \e_2 |\cJ|)\big(\e_1 (\theta_\tmax+\gammaJ)+ \e_2\big)}.
	\end{align*}
\end{lemma}

\begin{proof}[Proof of Lemma~\ref{spec:sup:RFB}]
	Note that \eqref{spec:sup:RFB-cond} implies
	\begin{equation*}
		\| (\e_1 \cH + \e_2 \cI_\mbH)^{-1/2} (\hcH-\cH) (\e_1 \cH + \e_2 \cI_\mbH)^{-1/2} \|_{\op,\mbH} \leq 1.
	\end{equation*}
	Let $\cU = (\e_1 \cH + \e_2 \cI_\mbH)^{-1/2} (\hcH-\cH) (\e_1 \cH + \e_2 \cI_\mbH)^{-1/2}$ so that
	$\|\cU\|_{\op,\mbH}\leq1$.
	From the identity
	\begin{align*}
		\cT_{\!\cJ}^{1/2} (\hcH-\cH) \cT_{\!\cJ}^{1/2}
		&= \cT_{\!\cJ}^{1/2} (\e_1 \cH + \e_2 \cI_\mbH)^{1/2} \cU 
		(\e_1 \cH + \e_2 \cI_\mbH)^{1/2} \cT_{\!\cJ}^{1/2},
	\end{align*}
	we bound $\halpha_{\!\cJ}$ as
	\begin{equation}
		\halpha_{\!\cJ} \leq \big\|(\e_1 \cH + \e_2 \cI_\mbH)^{1/2} \cT_{\!\cJ}^{1/2} \big\|^2_{\op,\mbH}.
		\label{spec:sup:RFB-eq2}
	\end{equation}
	From the expansion
	\begin{align*}
		(\e_1 \cH + \e_2 \cI_\mbH)^{1/2} \cT_{\!\cJ}^{1/2} 
		= \sum_{k \in \cJ} \sqrt{\frac{\e_1 \lambda_k + \e_2}{\gammaJ}} \cQ_k
		+ \sum_{k \notin \cJ} \sqrt{\frac{\e_1 \lambda_k + \e_2}{\min_{m \in \cJ} |\lambda_k-\lambda_m|}} \cQ_k,
	\end{align*}
	we see that
	\begin{align}
		\| (\e_1 \cH + \e_2 \cI_\mbH)^{1/2} \cT_{\!\cJ}^{1/2} \|_{\op,\mbH}^2
		= \max \bigg\{ \max_{k \in \cJ} \frac{\e_1 \lambda_k + \e_2}{\gammaJ}, \,
		\sup_{k \notin \cJ} \frac{\e_1 \lambda_k + \e_2}{\min_{m \in \cJ} |\lambda_k-\lambda_m|} \bigg\}.
		\label{spec:sup:RFB-eq3}
	\end{align}
	Obviously,
	\begin{equation}
		\max_{k \in \cJ} \frac{\e_1 \lambda_k + \e_2}{\gammaJ} = \frac{\e_1 \theta_\tmax + \e_2}{\gammaJ}.
		\label{spec:sup:RFB-eq4}
	\end{equation} 
	On the other hand, for each $k \notin \cJ$, we see that
	\begin{align*}
		\frac{\e_1 \lambda_k + \e_2}{\min_{m \in \cJ} |\lambda_k-\lambda_m|}
		= \max_{m \in \cJ} \bigg( \frac{\e_1 \lambda_m + \e_2}{|\lambda_k-\lambda_m|} + \e_1 \frac{\lambda_k -\lambda_m}{|\lambda_k-\lambda_m|} \bigg)
		\leq \frac{\e_1 \theta_\tmax + \e_2}{\gammaJ} + \e_1 
		= \frac{\e_1 (\theta_\tmax+\gammaJ)+ \e_2}{\gammaJ}.
	\end{align*}
	Consequently, we obtain
	\begin{equation}
		\sup_{k \notin \cJ} \frac{\e_1 \lambda_k + \e_2}{\min_{m \in \cJ} |\lambda_k-\lambda_m|}
		\leq \frac{\e_1 (\theta_\tmax+\gammaJ)+ \e_2}{\gammaJ}.
		\label{spec:sup:RFB-eq5}
	\end{equation}
	Combining \eqref{spec:sup:RFB-eq2}-\eqref{spec:sup:RFB-eq5} yields the first claim.
	
	Similarly, from the identity
	\begin{equation*}
		|\cR|_{\!\cJ}^{1/2} (\hcH-\cH) \cP_{\!\cJ} = |\cR|_{\!\cJ}^{1/2} (\e_1\cH+\e_2\cI_\mbH)^{1/2} \cU (\e_1\cH+\e_2\cI_\mbH)^{1/2} \cP_{\!\cJ},
	\end{equation*}
	we bound $\hbeta_{\!\cJ}$ as
	\begin{equation}
		\hbeta_{\!\cJ}
		\leq \gammaJ^{-1/2} \big\| |\cR|_{\!\cJ}^{1/2} (\e_1\cH+\e_2\cI_\mbH)^{1/2} \big\|_{\op,\mbH} \big\|(\e_1\cH+\e_2\cI_\mbH)^{1/2} \cP_{\!\cJ} \big\|_{\HS,\mbH}.
		\label{spec:sup:RFB-eq6}
	\end{equation}
	We see that using \eqref{spec:sup:RFB-eq5},
	\begin{equation}
		\big\| |\cR|_{\!\cJ}^{1/2} (\e_1\cH+\e_2\cI_\mbH)^{1/2} \big\|_{\op,\mbH}
		= \sqrt{\sup_{k \notin \cJ} \frac{\e_1 \lambda_k + \e_2}{\min_{m \in \cJ} |\lambda_k-\lambda_m|}}
		= \sqrt{\frac{\e_1 (\theta_\tmax+\gammaJ)+ \e_2}{\gammaJ}}.
		\label{spec:sup:RFB-eq7}
	\end{equation}
	and by a direct calculation,
	\begin{align}
		\big\|(\e_1\cH+\e_2\cI_\mbH)^{1/2} \cP_{\!\cJ}\big\|_{\HS,\mbH}
		&= \sqrt{ \sum_{k \in \cJ} \big\| (\e_1\cH+\e_2\cI_\mbH)^{1/2} \cP_{\!\cJ} \psi_k \big\|_\mbH^2 }
		\notag\\&= \sqrt{ \sum_{k \in \cJ} \big\| \sqrt{\e_1 \lambda_k + \e_2} \psi_k \big\|_\mbH^2 }
		= \sqrt{\e_1 \sum_{k \in \cJ} \lambda_k + \e_2 |\cJ| }.
		\label{spec:sup:RFB-eq8}
	\end{align}
	Combining \eqref{spec:sup:RFB-eq6}-\eqref{spec:sup:RFB-eq8} yields the second claim, which completes the proof.
\end{proof}

\subsubsection{Weighted perturbation under relative noise condition}

Recall the RN condition:
\begin{equation}
	|\langle (\hcH-\cH) \psi_k, \psi_\ell \rangle_\mbH | \leq \e \sqrt{\lambda_k \lambda_\ell}
	\quad\text{for all $k,\ell \in \N$,\, where $\e\ge0$.}
	\label{spec:sup:RN-cond}
\end{equation}

We first present how RFB and RN are related, as stated below.
\begin{lemma} \label{spec:sup:RFBtoRN}
	Under the RFB condition \eqref{spec:sup:RFB-cond}, we have
	\begin{equation}
		|\langle (\hcH-\cH)v, w \rangle_\mbH| \leq \sqrt{\langle (\e_1 \cH +\e_2) v, v \rangle_\mbH} \sqrt{\langle (\e_1 \cH +\e_2) w, w \rangle_\mbH},
	\end{equation}
	for all $v,w \in \mbH$ with either $\langle (\e_1 \cH + \e_2) v, v \rangle_\mbH \not=0 $ or $\langle (\e_1 \cH + \e_2) w, w \rangle_\mbH \not=0 $

	In particular, when $\e_2=0$, RFB implies RN with $\e=\e_1$.
\end{lemma}

\begin{proof}[Proof of Lemma~\ref{spec:sup:RFBtoRN}]
	The second claim follows directly from the first.
	
	To show the first claim, for notational simplicity, let $\cB = \hcH-\cH$ and $\cA = \e_1 \cH +\e_2$.
	For any $t \in \R$ and any $v,w \in \mbH$, the polarization identity gives
	\begin{equation*}
		4t \langle \cB v, w \rangle_\mbH = \langle \cB(tv+w), tv+w \rangle_\mbH + \langle \cB(tv-w), tv-w \rangle_\mbH.
	\end{equation*}
	Together with the RFB condition \eqref{spec:sup:RFB-cond}, we have
	\begin{align*}
		4|t| \cdot |\langle \cB v, w \rangle_\mbH|
		&\leq |\langle \cB(tv+w), tv+w \rangle_\mbH| + |\langle \cB(tv-w), tv-w \rangle_\mbH|
		\\&\leq \langle \cA(tv+w), tv+w \rangle_\mbH + \langle \cA(tv-w), tv-w \rangle_\mbH
		= 2t^2 \langle \cA v ,v \rangle_\mbH + 2 \langle \cA w, w \rangle_\mbH,
	\end{align*}
	which further implies
	\begin{equation*}
		|\langle \cB v, w \rangle_\mbH| \leq \frac{1}{2} \Big( |t| \langle \cA v ,v \rangle_\mbH + |t|^{-1} \langle \cA w, w \rangle_\mbH \Big),
		\quad\text{for any nonzero $t \in \R$}.
	\end{equation*}
	Hence, by letting either $t = \sqrt{\langle \cA w, w \rangle_\mbH} /\sqrt{\langle \cA v, v \rangle_\mbH} $ or $t = \sqrt{\langle \cA v, v \rangle_\mbH}/\sqrt{\langle \cA w, w \rangle_\mbH}$, the first claim follows.
\end{proof}

Under the RN condition, we bound $\halpha_{\!\cJ}$ and $\hbeta_{\!\cJ}$ as follows.
\begin{lemma} \label{spec:sup:RN}
	Suppose Assumptions~\refsupp{spec:assump1} and \refsupp{spec:assump2}.
	If the RN condition \eqref{spec:sup:RN-cond} is satisfied, we have
	\begin{gather}
		\halpha_{\!\cJ} \leq \|\cT_{\!\cJ}^{1/2} (\hcH-\cH) \cT_{\!\cJ}^{1/2} \|_{\HS,\mbH} 
		\leq \e \bigg(\frac{1}{\gammaJ} \sum_{k \in \cJ} \lambda_k + \sum_{k \notin \cJ} \frac{\lambda_k}{\min_{m \in \cJ} |\lambda_k -\lambda_m|}\bigg), 
		\label{spec:sup:RN-res1}
		\\
		\hbeta_{\!\cJ} \leq \e \sqrt{\frac{1}{\gammaJ} \sum_{k \in \cJ} \lambda_k} \sqrt{\sum_{k \notin \cJ} \frac{\lambda_k}{\min_{m \in \cJ} |\lambda_k -\lambda_m|}}.
		\label{spec:sup:RN-res2}
	\end{gather}
\end{lemma}

\begin{proof}[Proof of Lemma~\refsupp{spec:sup:RN}]
	For $\alpha_{\!\cJ}$, the first inequality in \eqref{spec:sup:RN-res1} follows directly from the fact that the HS norm dominates the OP norm.
	For the second inequality in \eqref{spec:sup:RN-res1}, we have
	\begin{align*}
		&\|\cT_{\!\cJ}^{1/2} ( \hcH - \cH) \cT_{\!\cJ}^{1/2}\|_{\HS,\mbH}^2
		= \sum_{k=1}^\infty \sum_{\ell=1}^\infty \langle \cT_{\!\cJ}^{1/2} ( \hcH - \cH) \cT_{\!\cJ}^{1/2} \psi_k, \psi_\ell \rangle_\mbH^2
		\\&
		\quad= \frac{1}{\gammaJ^2}\sum_{k\in\cJ} \sum_{\ell\in\cJ} \langle (\hcH - \cH) \psi_k, \psi_\ell \rangle_\mbH^2
		+ \frac{1}{\gammaJ}\sum_{k\in\cJ} \sum_{\ell\notin\cJ} \frac{\langle (\hcH - \cH)  \psi_k, \psi_\ell \rangle_\mbH^2}{\min_{m\in\cJ}|\lambda_\ell - \lambda_m|} 
		\\&
		\quad\quad + \frac{1}{\gammaJ} \sum_{k\notin\cJ} \sum_{\ell\in\cJ} \frac{\langle (\hcH - \cH) \psi_k, \psi_\ell \rangle_\mbH^2}{\min_{m\in\cJ}|\lambda_\ell - \lambda_k|} 
		+ \sum_{k\notin\cJ} \sum_{\ell\notin\cJ} \frac{\langle (\hcH - \cH) \psi_k, \psi_\ell \rangle_\mbH^2}{\min_{m_1\in\cJ}|\lambda_\ell - \lambda_{m_1}| \cdot \min_{m_2\in\cJ}|\lambda_\ell - \lambda_{m_2}|}
		\\&
		\quad \leq \frac{\e^2}{\gammaJ^2}\sum_{k\in\cJ} \sum_{\ell\in\cJ} \lambda_k^2 \lambda_\ell^2
		+ \frac{\e^2}{\gammaJ}\sum_{k\in\cJ} \sum_{\ell\notin\cJ} \frac{\lambda_k^2 \lambda_\ell^2}{\min_{m\in\cJ}|\lambda_\ell - \lambda_m|} 
		\\&
		\quad\quad + \frac{\e^2}{\gammaJ} \sum_{k\notin\cJ} \sum_{\ell\in\cJ} \frac{\lambda_k^2 \lambda_\ell^2}{\min_{m\in\cJ}|\lambda_\ell - \lambda_k|} 
		+ \e^2\sum_{k\notin\cJ} \sum_{\ell\notin\cJ} \frac{\lambda_k^2 \lambda_\ell^2}{\min_{m_1\in\cJ}|\lambda_\ell - \lambda_{m_1}| \cdot \min_{m_2\in\cJ}|\lambda_\ell - \lambda_{m_2}|}
		\\&
		\quad = \e^2 \Bigg( \frac{1}{\gammaJ} \sum_{k\in \cJ} \lambda_k
		+ \sum_{k \not\in \cJ} \frac{\lambda_k}{\min_{m \in \cJ}|\lambda_k-\lambda_m|} \Bigg)^2.
	\end{align*}
	Hence, \eqref{spec:sup:RN-res1} holds.

	For \eqref{spec:sup:RN-res2}, we have
	\begin{align*}
		\gammaJ \hbeta_\cJ^2
		&= \sum_{k=1}^\infty \sum_{\ell=1}^\infty
		\langle |\cR|_\cJ^{1/2} (\hcH-\cH) \cP_{\!\cJ} \psi_k, \psi_\ell \rangle_\mbH^2
		= \sum_{k \in \cJ} \sum_{\ell \notin \cJ}
		\frac{\langle (\hcH-\cH) \psi_k, \psi_\ell \rangle_\mbH^2}{\min_{m \in \cJ} |\lambda_\ell - \lambda_m|} 
		\\&\leq \sum_{k \in \cJ} \sum_{\ell \notin \cJ}	\frac{\e^2 \lambda_k \lambda_\ell}{\min_{m \in \cJ} |\lambda_\ell - \lambda_m|}
		= \e^2 \sum_{k \in \cJ} \lambda_k \sum_{\ell \notin \cJ} \frac{\lambda_\ell}{\min_{m \in \cJ} |\lambda_\ell - \lambda_m|}.
	\end{align*}
	Therefore, \eqref{spec:sup:RN-res2} holds, and this concludes the proof.
\end{proof}

\subsubsection{Comparison with existing results under relative noise condition}

This part is devoted to comparing our results with existing results considering the RN condition alone, showing that our results are compatible with them, even though we consider a more general setting.

We assume Assumptions~\refsupp{spec:assump1} and \refsupp{spec:assump2}, as well as the RN condition \eqref{spec:sup:RN-cond}, hold.
For ease of comparison, we consider a simple eigenvalue $\lambda_j$, i.e., $\cJ = \{j\}$.
Let $\bfr_j = \frac{\lambda_j}{\gammaJ} + \sum_{k: k\not= j} \frac{\lambda_k}{|\lambda_k-\lambda_j|} $, which is the relative rank of $\lambda_j$. 

Under $\e \bfr_j < 1/4$, applying Lemma~\ref{spec:sup:RN} to our Theorems~\refsupp{spec:eigvec:simple} and \refsupp{spec:eigval-separate} provides bounds for the first-order asymptotic expansion of the eigenvector $\hpsi_j$ and eigenvalue $\hlambda_j$ as
\begin{gather*}
	E_{j1} :=\, \bigg| \hpsi_j - \psi_j - \sum_{k: k \not= j} \frac{\langle (\hcH-\cH) \lambda_j, \lambda_k \rangle_\mbH}{\lambda_j - \lambda_k} \bigg| 
	\lesssim \e^2 \bfr_j\sqrt{ \frac{\lambda_j}{\gammaJ} \sum_{k \notin \cJ} \frac{\lambda_k}{|\lambda_k-\lambda_j|} }, 
	\\
	E_{j2} :=\, |\hlambda_j - \lambda_j - \langle (\hcH-\cH) \lambda_j, \lambda_j \rangle_\mbH| 
	\lesssim \e^2 \lambda_j \sum_{k: k\not= j} \frac{\lambda_k}{|\lambda_k-\lambda_j|}.	
\end{gather*}
While, under the same conditions, \citesupp{jirak2023relative} provides in its Theorems~1 and 2 that
\begin{align*}
	E_{j1} \lesssim \e^2 \bfr_j\sqrt{ \lambda_j \sum_{k \notin \cJ} \frac{\lambda_k}{|\lambda_k-\lambda_j|^2} }
	\quad\text{and}\quad
	E_{j2} \lesssim \e^2 \lambda_j \bfr_j.	
\end{align*}
From above, we see that our bound is looser for the eigenvector, but sharper for the eigenvalue.

To illustrate the comparison in a concrete example, assume the eigenvalues of $\cH$ have polynomial decay, i.e., $\lambda_k = k^{-a}$ for some $a >0$. Then, it holds that (see Lemmas~1 and 2 in \citesupp{cardot2007clt}, as well as formula (3.18) in \citesupp{jirak2023relative})
\begin{equation}
	\frac{1}{\gammaJ} \lesssim j^{a+1}, \quad
	\bfr_j \lesssim j \log j, \quad
	\sum_{k:k \not=j} \frac{\lambda_k}{|\lambda_k-\lambda_j|} \lesssim j \log j, \quad
	\lambda_j \sum_{k \notin \cJ} \frac{\lambda_k}{|\lambda_k-\lambda_j|^2} \lesssim j^2.
\end{equation}
Thus, under $ \e j \log j \lesssim 1/4$, i.e., $\e = o(j^{-1})$, 
our results yield $E_{1j} \lesssim \e^2 j^2 \log j \sqrt{\log j}$ and $E_{2j} \lesssim \e^2 j^{1-a} \log j$;
while, under the same conditions, \citesupp{jirak2023relative} provides $E_{1j} \lesssim \e^2 j^2 \log j$ and $E_{2j} \lesssim \e^2 j^{1-a} \log j$.
We see that our eigenvector bound is looser by a mild factor of $\sqrt{\log j}$, which remains moderate even for very large $j$,
while our eigenvalue bound has the same order.

On the other hand, our results can provide sharper bounds in this setting under RFB.
As shown in \ref{spec:sup:RFBtoRN}, RFB implies RN if $\e_1=\e$ and $\e_2=0$.
Under this and the setting above, Lemma~\ref{spec:sup:RFB} provides $\halpha_{\!\cJ}, \hbeta_{\!\cJ} \lesssim \e j$, and consequently, we obtain $E_{1j} \lesssim \e^2 j^2$ and $E_{2j} \lesssim \e^2 j^{1-a}$, which are sharper than the corresponding bounds above.

\subsection{Details of Example~\refsupp{mercer:perturb:ex1}}

We formulate the result as a lemma below.

\begin{lemma} \label{spec:sup:ex2}
	Suppose Assumptions~\refsupp{mercer:assump1} and \refsupp{mercer:assump2} hold.
	Under the conditions $\cJ= \{j\}$, $\lambda_k \sim e^{-k}$ and $\sup_{x \in \M} |\phi_k (x)| \lesssim k^2$,
	we have
	\begin{equation}
		\bfkappa_{\!\cJ} \lesssim j^5, \hquad 
		\bfsigma_{\!\!\cJ} \lesssim j, \hquad
		\bfvarkappa_{\!\cJ} \lesssim j^{9/2}, \hquad
		\bfvarsigma_{\!\cJ} \lesssim j^5.
	\end{equation}
\end{lemma}

\begin{proof}[Proof of Lemma~\ref{spec:sup:ex2}]
	To simplify notation, we let 
	\begin{equation*}
		f(x) = \frac{1}{\gammaJ}\sum_{k \in \cJ} \lambda_k \phi_k^2(x), \quad
		g(x) = \sum_{\ell \notin \cJ} \frac{ \lambda_\ell \phi_\ell^2 (x) }{\min_{m \in \cJ} |\lambda_\ell - \lambda_m|},
	\end{equation*}
	so that
	\begin{gather*}
		\bfkappa_{\!\cJ} = \sup_{x \in \M} (f(x) + g(x)), \quad
		\bfsigma_{\!\!\cJ} = \E[ f(X) + g(X)], \quad
		\bfvarkappa_{\!\cJ}	= \sup_{x \in \M} \sqrt{ f(x) g(x) }, \\
		\bfvarsigma_\cJ	= \E [f(X)g(X)] \leq \sup_{x \in \M} f(x) \E g(X).
	\end{gather*}
	
	For $f$, we have
	\begin{equation}
		\E f(X) = \frac{1}{\gammaJ} \sum_{k \in \cJ} \lambda_k \sim \frac{e^{-j}}{e^{-j}(1-e^{-1})} \sim 1, \quad
		\sup_{x \in \M} f(x) \lesssim \frac{1}{\gammaJ} \sum_{k \in \cJ} \lambda_k k^4 \sim j^4.
		\label{spec:sup:ex2-eq1}
	\end{equation} 
	
	For $g$, we have
	\begin{align}
		\E g(X) = \sum_{k \notin \cJ} \frac{ \lambda_k }{\min_{m \in \cJ} |\lambda_k - \lambda_m|}
		&\lesssim \sum_{k=1}^{j-1} \frac{e^{-k}}{e^{-k}-e^{-j}} + \sum_{k=j+1}^\infty \frac{e^{-k}}{e^{-j}-e^{-k}}
		\notag\\&= (j-1) + \sum_{k=1}^{j-1} \frac{1}{e^{j-k} - 1} + \sum_{k=j+1}^\infty \frac{1}{e^{k-j}-1}
		\notag\\&\leq (j-1) + 4\sum_{\ell=1}^\infty e^{-\ell}
		\lesssim j,
		\label{spec:sup:ex2-eq2}
	\end{align}
	and
	\begin{align}
		\sup_{x \in \M} g(x) = \sup_{x \in \M} \sum_{k \notin \cJ} \frac{\lambda_k \phi_k^2(x) }{\min_{m \in \cJ} |\lambda_k - \lambda_m|}
		&\lesssim \sum_{k=1}^{j-1} \frac{k^4 e^{-k}}{e^{-k}-e^{-j}} + \sum_{k=j+1}^\infty \frac{k^4e^{-k}}{e^{-j}-e^{-k}}
		\notag\\&= \sum_{k=1}^{j-1} k^4 + \sum_{k=1}^{j-1} \frac{k^4}{e^{j-k}-1} + \sum_{k=j+1}^\infty \frac{k^4}{e^{k-j}-1}
		\notag\\&\leq \sum_{k=1}^{j-1} k^4 + 4 \sum_{\ell=1}^\infty (\ell+j)^4 e^{-\ell} \lesssim j^5.
		\label{spec:sup:ex2-eq3}
	\end{align}
	
	Hence, the proof is complete by \eqref{spec:sup:ex2-eq1}-\eqref{spec:sup:ex2-eq3}.
\end{proof}

\end{document}